\documentclass[nonblindrev]{informs3}
\usepackage[T1]{fontenc}
\usepackage[utf8]{inputenc}

\usepackage[dvipsnames]{xcolor}
\usepackage{tikz}
\usetikzlibrary{patterns}

\OneAndAHalfSpacedXI

\usepackage{endnotes}
\let\footnote=\endnote
\let\enotesize=\normalsize
\def\notesname{Endnotes}%
\def\makeenmark{\hbox to1.275em{\theenmark.\enskip\hss}}
\def\enoteformat{\rightskip0pt\leftskip0pt\parindent=1.275em
  \leavevmode\llap{\makeenmark}}

\usepackage{hyperref}
\newcommand{\thmref}[2]{\hyperref[#1]{#2 \ref*{#1}}}
\newcommand\myshade{100}
\colorlet{mylinkcolor}{MidnightBlue}
\hypersetup{
	linkcolor  = mylinkcolor!\myshade!black,
	citecolor  = mylinkcolor!\myshade!black,
	urlcolor   = mylinkcolor!\myshade!black,
	colorlinks = true,
}

\usepackage{natbib}
 \bibpunct[, ]{(}{)}{,}{a}{}{,}%
 \def\bibfont{\footnotesize}%

\usepackage{verbatim}
\usepackage{amssymb}
\usepackage{enumerate}
\usepackage{mathrsfs}
\usepackage{graphicx}
\usepackage{epstopdf}

\usepackage{multirow} 
\usepackage{booktabs} 
\usepackage{array} 
\usepackage{makecell} 
\usepackage{diagbox} 
\usepackage{tabularx}
\usepackage{threeparttable}

\usepackage[ruled,vlined,linesnumbered]{algorithm2e}

\usepackage{subcaption}

\usepackage{soul}
\usepackage{url}
\usepackage{bm}
\usepackage{bbm}
\usepackage{mathtools}
\usepackage{pifont}

\usepackage[titletoc,toc,title]{appendix}

\def\qed{\hfill $\square$}

\newcommand\bb[1]{\mathbb{#1}}

\newcommand\romanone{\mathrm{\uppercase\expandafter{\romannumeral1}}}
\newcommand\romantwo{\mathrm{\uppercase\expandafter{\romannumeral2}}}

\newcommand{\yes}{\ding{51}}
\newcommand{\no}{\ding{55}}

\DeclareMathOperator{\dist}{dist}

\definecolor{blue1}{RGB}{138,171,202}
\definecolor{blue2}{RGB}{17,64,108}
\definecolor{blue3}{RGB}{3,41,81}

\newtheorem{condition}{Condition}

\theoremstyle{TH}

\TheoremsNumberedThrough
\ECRepeatTheorems

\EquationsNumberedThrough

\MANUSCRIPTNO{}

\begin{document}

\RUNAUTHOR{C. Tan, Y. Mao, S. Wang, H. Xu}

\RUNTITLE{Integrated Learning and Robust Optimization}

\TITLE{Integrated Learning and Robust Optimization}

\ARTICLEAUTHORS{%
\AUTHOR{Cheng Tan$^1$\!\!, Yuchen Mao$^2$\!\!, Shuming Wang$^1$\!\!, and Huan Xu$^{3,\footnote{Corresponding author.}}$}
\AFF{$^1$School of Economics \& Management, University of Chinese Academy of Sciences, China \\
Emails: tancheng241@mails.ucas.ac.cn, wangshuming@ucas.edu.cn \\
$^2$Department of Statistics, Rutgers University, United States \\ 
Email: yuchen.mao@rutgers.edu \\ 
$^3$Antai College of Economics and Management, Shanghai Jiao Tong University, China \\ 
Email: xuhuan\_antai@sjtu.edu.cn}
}

\ABSTRACT{Many operational decisions require solving a linear program whose cost vector is unknown at decision time and must be predicted from contextual information. Because prediction and decision are only weakly aligned, the emerging integrated learning and optimization (ILO) paradigm trains the predictor through the downstream problem, judging a prediction by the decision it induces. However, predictions are inevitably imprecise, so robustness often enters the decision stage. To address this issue, we propose an integrated learning and robust optimization (ILRO) framework, where a robust decision problem is used both to define the training problem (termed the $\mathcal{RSPO}$ loss problem), and to produce the deployed decision. Thus, this framework simultaneously achieves both robustness and learning--decision alignment. To tackle its computational challenges, we develop a convex surrogate, $\mathcal{RSPO}_+$, and characterize when it is Fisher consistent. Moreover, the $\mathcal{RSPO}$ loss possesses informative gradients, allowing us to develop first-order computational methods. We also derive finite-sample excess risk bounds for both $\mathcal{RSPO}$ and $\mathcal{RSPO}_+$ predictors. Numerical experiments on transportation and portfolio problems, in comparison with multiple benchmarks, show the advantage in decision quality of the proposed framework. The gain is more pronounced for scenarios with limited samples, high-dimensional decisions, and model misspecification.
}


\maketitle

\section{Introduction}
\label{sec:introduction}
Many operational decisions are made after contextual information is observed but before the relevant cost parameters are realized. In transportation planning, arc costs depend not only on observable weather, congestion, and demand signals at the time of routing, but also on unforeseen events---such as sudden traffic accidents or real-time weather shifts---that are revealed only after the vehicle has departed. Similarly, in portfolio allocation, asset returns are influenced by market covariates observed prior to the investment, yet the actual returns are determined by post-decision market shocks that are unknown at the time of allocation. Hence, such applications are naturally modeled as contextual stochastic optimization problems, where the decision is adapted to observed contextual information while the uncertain cost parameters are realized only after the decision is made.

In this paper, we focus on a fundamental and widely used class of contextual stochastic programs with a linear objective and a context-independent feasible region. Specifically, after observing a context vector $\bm{x}\in\mathcal{X}\subseteq\mathbb{R}^{p}$, the decision-maker chooses a feasible solution $\bm{z}\in\mathcal{Z}\subseteq\mathbb{R}^{d}$, where $\mathcal{Z}$ is a convex compact feasible set. The realized cost is linear in the decision, with a random coefficient vector $\bm{y}\in\mathcal Y\subseteq\mathbb{R}^{d}$ whose conditional distribution depends on the observed context $\bm{x}$. Under a risk-neutral objective, the benchmark decision solves
\begin{equation}    \label{eq:CSO}
\min_{\bm{z}\in\mathcal{Z}}
\mathbb{E}_{\bm{y}\sim\mathbb{P}_{\bm{y}|\bm{x}}}
[\bm{y}^\top\bm{z}]
=
\min_{\bm{z}\in\mathcal{Z}}
\mathbb{E}_{\bm{y}\sim\mathbb{P}_{\bm{y}|\bm{x}}}
[\bm{y}]^\top\bm{z}.
\end{equation}

Thus, if the conditional mean $\bar{\bm{y}}(\bm{x}):=\mathbb{E}_{\bm{y}\sim\mathbb{P}_{\bm{y}|\bm{x}}}[\bm{y}]$ were known, the problem would reduce to a deterministic optimization problem with this conditional mean as the cost vector. In practice, however, the conditional distribution of $\bm{y}$, and hence its conditional mean, is unknown and must be inferred from data. Given historical samples $\{(\bm{x}_i,\bm{y}_i)\}_{i\in[N]}$, where $[N]=\{1,\ldots,N\}$, a standard approach is to estimate the conditional mean of $\bm{y}$ and then optimize using the estimate. This predict-then-optimize (PTO) pipeline is modular and easy to implement, but it is \textit{not} designed directly for decision quality. Indeed, as emphasized in \cite{elmachtoub2022smart}, prediction accuracy and decision quality can be weakly aligned: two predictors with similar statistical accuracy may induce remarkably different optimal decisions.

The recent literature on contextual and prescriptive optimization formalizes this issue and develops methods that use covariates directly for decision making \citep{ban2019big,bertsimas2020predictive}. These approaches shift the learning cost from purely statistical accuracy toward downstream operational performance, thereby providing a more decision-aware alternative to the classical PTO paradigm. Integrated learning and optimization (ILO), also called decision-focused learning, takes a further step by training predictive models through the downstream optimization problem itself \citep{donti2017task,wilder2019melding,elmachtoub2022smart}. Rather than evaluating a prediction solely by its discrepancy from the realized cost vector, ILO assesses predictions through the decisions they induce and the corresponding objective values. It therefore provides a framework for training prediction models that are directly tailored for downstream decision quality.

To illustrate the basic ILO framework in the present linear-objective setting, consider a parametric prediction class $\{\bm{g}_\theta:\theta\in\Theta\}$, where $\bm{g}_{\theta}(\bm{x})$ predicts the cost vector conditional on the observed context $\bm x$. For a given context $\bm x$, the predicted cost vector $\hat{\bm y}=\bm g_\theta(\bm x)$ is passed to the nominal decision map, that is
\begin{equation}\label{eq:ilolearningoracle}
    \bm{z}^{\star}(\hat{\bm y})\in \argmin_{\bm{z}\in\mathcal{Z}}\hat{\bm y}^\top\bm{z},
\end{equation}
which returns a decision that is optimal under the predicted cost vector $\hat{\bm y}$. For each training sample $(\bm x_i,\bm y_i)$, the model first predicts the cost vector $\bm g_\theta(\bm x_i)$, which induces the downstream decision $\bm z^\star(\bm g_\theta(\bm x_i))$. The quality of this decision is then evaluated using the realized cost vector $\bm y_i$, yielding the realized decision cost $\bm y_i^\top \bm z^\star\big(\bm g_\theta(\bm x_i)\big)$. Accordingly, the empirical ILO problem selects the model parameter by minimizing the average realized cost of the decisions induced by its predictions:
\begin{equation} \label{eq:ilolearning}
    \theta^\star_N\in \argmin_{\theta\in\Theta} \frac{1}{N}\sum_{i\in[N]} \bm{y}_i^\top\bm{z}^{\star}\big(\bm{g}_{\theta}(\bm{x}_i)\big).
\end{equation}
After training, for a new context $\bm x$, the learned model produces the predicted cost vector $\bm g_{\theta_N^\star}(\bm x)$, and the corresponding prescription is obtained by solving
\begin{equation} \label{eq:iloopt}
    \bm{z}^{\star}\big(\bm{g}_{\theta^\star_N}(\bm{x})\big)\in \argmin_{\bm{z}\in\mathcal{Z}}\bm{g}_{\theta^\star_N}(\bm{x})^\top \bm{z}.
\end{equation}
The formulations in~\eqref{eq:ilolearning}--\eqref{eq:iloopt} capture the central idea of decision-focused learning: the prediction model is trained not according to prediction accuracy alone, but according to the downstream performance of the decisions induced by its predictions. Nevertheless, the framework continues to rely on a nominal linear program and does not explicitly account for uncertainty in the predicted costs. Moreover, in linear optimization, the nominal decision map may be set-valued and unstable with respect to small perturbations in the predicted cost vector. These limitations motivate our \textbf{I}ntegrated \textbf{L}earning and \textbf{R}obust \textbf{O}ptimization (ILRO) framework introduced below.

\subsection{Motivation and Our Approach}

The ILO framework at deployment still commits fully to the nominal decision map evaluated at a single point estimate $\hat{\bm{y}} = \bm{g}_{\theta}(\bm{x})$, so decision quality hinges on the gap $\hat{\bm{y}} - \bar{\bm{y}}(\bm{x})$---and decision-focused training does not close this gap. Sampling variability persists in the data-scarce regimes that motivate ILO in the first place; misspecification and omitted covariates leave a bias that no training criterion removes; distributional shift---drift in the covariates or in the conditional law of $\bm{y}$ given $\bm{x}$---biases even a decision-optimally trained predictor at deployment; and the realized cost stays random regardless. Worse, the nominal decision map in ILO amplifies these gaps: it is piecewise constant, so near a switching boundary a small prediction error sends the decision to a distant vertex. Robustness must therefore enter the decision stage itself. By accepting a controlled loss in nominal optimality, robust optimization hedges against a neighborhood of plausible cost vectors rather than a single prediction~\citep{bental1999robust,bertsimas2004price,bental2009robust,lu2021review}.

To incorporate robustness in the ILO framework, the most direct way is post hoc: train the predictor under the loss induced by the nominal decision map, then replace it with its robust counterpart at deployment. This pipeline is modular and easy to implement, but it optimizes the predictor for a decision rule that is never used---the model is trained for the nominal decision map yet deployed through a different, robust one, and a predictor that excels for the former need \emph{not} perform well for the latter. This mismatch motivates an integrated framework in which the same robust decision map generates the deployed decisions and defines the training loss. Figure~\ref{fig:method_pipeline} compares the four resulting pipelines---PTO, ILO, post-hoc robust ILO, and the proposed ILRO---and Table~\ref{tab:method_comparison1} locates them along the two design dimensions of robust decision-making and learning--decision alignment. ILRO is the only framework that captures both design dimensions.

\begin{figure}[!t]
    \centering
    \includegraphics[width=0.98\linewidth]{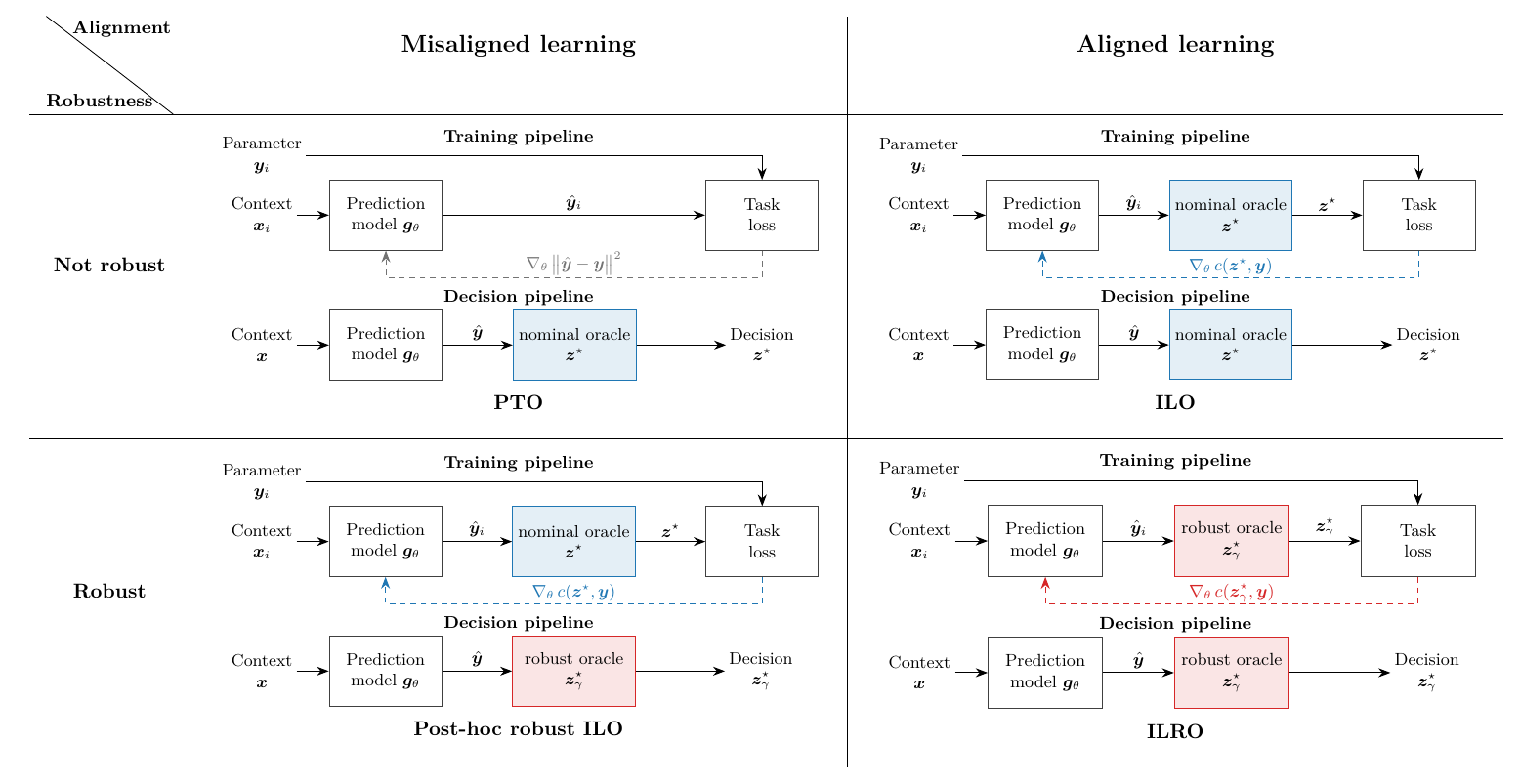}
    \caption{Pipelines of four methods.}
    \label{fig:method_pipeline}\vspace{-4mm}
\end{figure}

\begin{table}[b]\vspace{-2mm}
    \centering \footnotesize
        \renewcommand{\arraystretch}{1.0}\caption{Comparison of different frameworks.}
    \label{tab:method_comparison1}
        \begin{tabular}{l cc}
        \toprule
        \textbf{Method}
          & Robust decision-making
          & Learning--decision alignment\\
        \midrule
        \textbf{PTO}                  & \no  & \no \\
        \textbf{ILO}                  & \no  & \yes\\
        \textbf{Post-hoc robust ILO}  & \yes & \no \\
        \textbf{ILRO}                 & \yes & \yes\\
        \bottomrule
        \end{tabular}
    \vspace{-4mm}
\end{table}

In this paper, we formalize the ILRO framework by developing the following robust learning--decision aligned model components. For a predicted cost vector $\hat{\bm y}$ and a radius $\vartheta$, the robust decision problem
\begin{equation*}
    \argmin_{\bm{z}\in\mathcal{Z}}\;\max_{\|\bm{u}\|_2\leq \vartheta}\left\{(\hat{\bm y}+\bm{u})^\top\bm{z}\right\}
\end{equation*}
is equivalent to 
\begin{equation} \label{ilroopt}\tag{\textsc{ILRO}-\textsc{Decision}}
    \hspace{31mm}\bm{z}_{\gamma}^{\star}(\hat{\bm y})\in \argmin_{\bm{z}\in\mathcal{Z}}\hat{\bm y}^\top\bm{z}+\frac{\gamma}{2}\|\bm{z}\|^2_{2},
\end{equation}
up to a change of hyperparameter ($\vartheta\leftrightarrow\gamma$). See Remark~\ref{rem:soft-robustness} for a detailed explanation. Given a prediction class $\{\bm g_\theta:\theta\in\Theta\}$, ILRO trains the predictor through this robust decision map: 
\begin{equation}\label{ILRO-learning}\tag{\textsc{ILRO}-\textsc{Learning}}
\hspace{33mm}\theta^\star_{\gamma,N}\in \argmin_{\theta\in\Theta} \frac{1}{N}\sum_{i\in[N]} \bm{y}_i^\top\bm{z}_{\gamma}^{\star}(\bm{g}_{\theta}(\bm{x}_i)).    
\end{equation}
For a new context $\bm x$, the deployed decision is then
\begin{equation*}
    \bm{z}^{\star}_{\gamma}(\bm{g}_{\theta^\star_{\gamma,N}}(\bm{x}))\in 
    \argmin_{\bm{z}\in\mathcal{Z}}\bm{z}^\top\bm{g}_{\theta^\star_{\gamma,N}}(\bm{x})+\frac{\gamma}{2}\|\bm{z}\|^2_{2}.
\end{equation*}
ILRO differs from the ILO framework in that the former replaces the nominal linear program with a robust decision problem and uses this same decision problem to define the learning objective. Thus, ILRO aligns learning and deployment by training the predictor for the same robust decision rule. Throughout the paper, we take the hyperparameter $\gamma$ that reflects the robustness attitude of the decision-maker as exogenously given.

\begin{figure}[!t]
    \centering
    \begin{subfigure}[b]{0.49\textwidth}
        \centering
        \includegraphics[width=\textwidth]{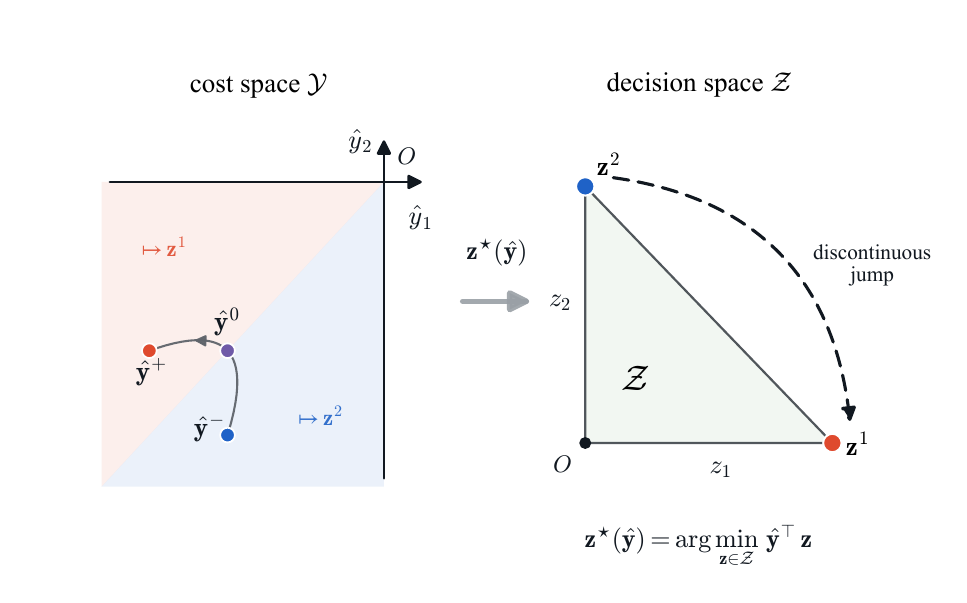}
    \end{subfigure}
    \hfill
    \begin{subfigure}[b]{0.49\textwidth}
        \centering
        \includegraphics[width=\textwidth]{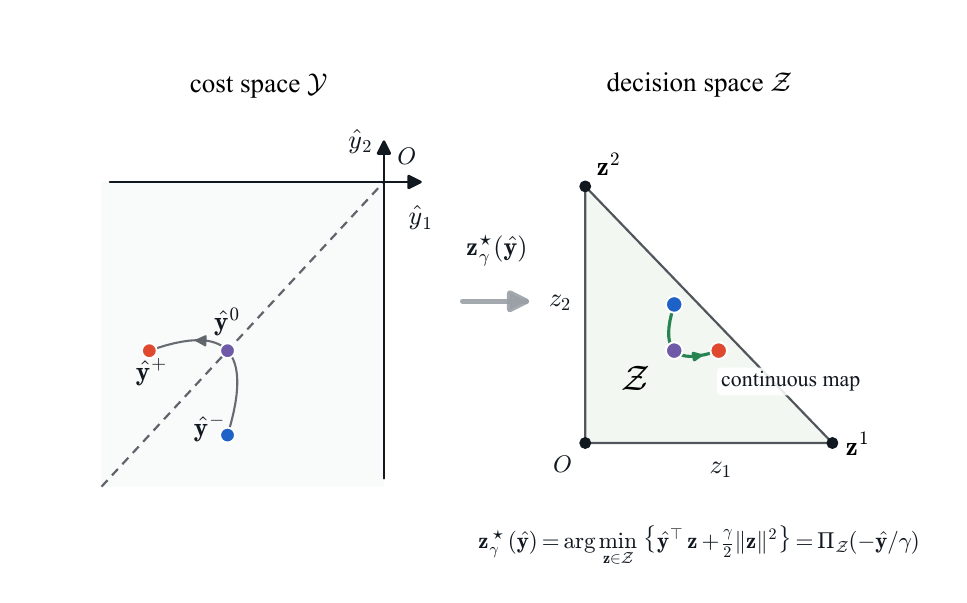}
    \end{subfigure}
    \caption{Comparison of solution paths of nominal and robust decision maps. The left panel shows the nominal decision map: nearby predictions $\hat{\bm y}^{-}$ and $\hat{\bm y}^{+}$ are mapped to different extreme points of $\mathcal{Z}$, producing a discontinuous decision map $\bm z^\star(\cdot)$. The right panel shows the robust decision map: nearby predictions are mapped to nearby feasible decisions, yielding a continuous decision map $\bm z^\star_{\gamma}(\cdot)$.}
    \label{fig:oracle_geometry_comparison}\vspace{-4mm}
\end{figure}
In addition to aligning learning and decision within the robustness modeling, we also emphasize that compared with ILO, the ILRO framework offers a \textit{structural advantage}---continuity of the solution mapping $\bm{z}_{\gamma}^{\star}(\cdot)$. Notice that the nominal decision map in ILO is typically set-valued (e.g., for a polyhedral feasible set), and its selected solution is piecewise constant as a function of the predicted cost vector, providing little useful gradient information for directly optimizing the ILO objective. As illustrated in Figure~\ref{fig:oracle_geometry_comparison}, two nearby predictions lying on opposite sides of a boundary, such as $\hat{\bm y}^{-}$ and $\hat{\bm y}^{+}$, select different extreme points, inducing a discontinuous change in the selected decision. In contrast, our framework smooths the sharp geometric structure of the nominal linear program---as it admits the projection representation $\bm z_\gamma^\star(\hat{\bm y})=\Pi_{\mathcal{Z}}(-\hat{\bm y}/\gamma)$---and hence replaces abrupt extreme-point selection with a continuous decision map. This benefits both the algorithm design of the optimization procedure and the analysis of the decision-focused learning performance.

These alignment and structural properties make ILRO a promising framework. Because the training loss is now defined through the robust decision problem in~\ref{ilroopt}, it evaluates predictions via the realized cost of a decision given by the optimization problem~\ref{ilroopt}, which is generally \textit{nonconvex} in the prediction---as is the ILO objective of the nominal decision-focused learning that shares this difficulty. Nevertheless, incorporating the robustness actually adds no difficulty of its own here; instead, it offers the appealing structural advantage that mitigates the difficulty. We therefore develop two complementary learning approaches: a tractable convex surrogate, and a gradient-based method that descends the original objective directly by exploiting the differentiability of the robust decision map, an option that the nominal $\mathcal{SPO}$ loss does not admit. We then establish theoretical guarantees for the resulting loss and its surrogate, and evaluate the framework through numerical experiments.

\subsection{Summary of Main Contributions}

The main contributions are summarized as follows.

\textbf{An integrated learning and robust optimization framework.} We develop an ILRO framework for contextual linear optimization problems with context-dependent objective coefficients, in which the same robust decision problem generates the deployed decisions and defines the training loss---the $\mathcal{RSPO}$ loss. The deployed decision problem is equivalent to a robust optimization problem up to a change of hyperparameter (Remark~\ref{rem:soft-robustness}), where $\gamma$ represents the robustness attitude of the decision-maker. The ILRO framework thus simultaneously achieves decision robustness and the alignment of learning and decision objectives. 

\textbf{Two complementary training schemes.} To solve the resulting learning problem, we develop two complementary training schemes. First, we construct a convex surrogate loss, $\mathcal{RSPO}_+$, which upper-bounds the $\mathcal{RSPO}$ loss, and is convex in the prediction. Moreover, $\mathcal{RSPO}_+$ admits an explicit gradient, in contrast to $\mathcal{SPO}_+$, which only admits subgradients in general. For polyhedral feasible sets, its empirical minimization admits a finite-dimensional reformulation (Theorem~\ref{thm:rspo_plus_reformulation}), which becomes a convex program under linear predictors and standard convex regularizers. Second, we exploit the differentiability of the robust decision map, whose Jacobian we derive in closed form, to refine the $\mathcal{RSPO}_+$ solution by descending the original $\mathcal{RSPO}$ objective directly via a local first-order scheme. This is unavailable for the nominal $\mathcal{SPO}$ loss, whose decision map is piecewise constant and carries no useful first-order information.

\textbf{Fisher consistency: sufficient conditions.} In contrast to the $\mathcal{SPO}$ setting, where Fisher consistency of the surrogate holds under mild distributional conditions for every feasible region, the population minimizer of the $\mathcal{RSPO}_+$ surrogate is given by $2\bar{\bm{y}}(\bm{x})$ (with $a=1$), whereas the set of target-risk minimizers explicitly depends on $\gamma$. We show that Fisher consistency holds whenever $\gamma$ does not exceed a threshold $\bar{\gamma}(\bm{x})$ defined by an explicitly computable linear program (Theorem~\ref{thm:Fisher_consistency}), and we show that when Fisher consistency fails, surrogate training can converge to a predictor that remains suboptimal under the target $\mathcal{RSPO}$ risk, resulting in a nonvanishing asymptotic bias. As a byproduct, our analysis at the boundary $\gamma = 0$ over bounded polyhedra (Corollary~\ref{cor:spo_plus_a1}) also enriches the Fisher consistency results of nominal $\mathcal{SPO}_+$ established by \citet{elmachtoub2022smart}.

\textbf{Theoretical guarantees.} We establish finite-sample excess-risk guarantees for both target and surrogate predictors. For the $\mathcal{RSPO}$ predictor, we derive a meta bound that depends on the robust decision problem through its Lipschitz factor of order $1/\gamma$ and the prediction class through its Rademacher complexity (Theorem~\ref{thm:rspo_meta_generalization}); specializing the latter yields $N^{-1/2}$ rates (up to logarithmic factors) for four representative classes, including infinite-dimensional ones (Appendix~\ref{app:gen_examples}). For the $\mathcal{RSPO}_+$ predictor, within the Fisher-consistent regime, the target excess risk converges at the rate $N^{-1/4}$ (Theorem~\ref{thm:direct_closure}), improving to $N^{-1/2}$ under an extra quadratic growth condition (Corollary~\ref{cor:direct_fast_rate}). Both analyses leverage the Lipschitz property of the prediction loss, which follows from the projection representation of the robust decision map. Our bounds scale as $1/\gamma$, which means that when the formulation is more robust (larger $\gamma$), fewer samples are required to achieve the same excess risk bound. This suggests that, when data are limited, the decision-maker should seek more robust decisions.

\textbf{Empirical validation and insights.} Numerical experiments on capacitated transportation and portfolio optimization problems---with benchmarks of standard $\mathcal{SPO}_+$, post-hoc robustification, and least-squares---show that $\mathcal{RSPO}_+$ attains the lowest or a comparable normalized decision loss across the reported settings, with its clearest advantages under scenarios of limited samples, high-dimensional decisions, and stronger model misspecification. Two findings merit emphasis. First, the decision gains of $\mathcal{RSPO}_+$ frequently occur despite \emph{higher} relative prediction loss than the benchmarks, consistent with an aligned training loss prioritizing decision-relevant accuracy over aggregate prediction accuracy. Second, its comparison with the post-hoc robust benchmark isolates the value of alignment itself: both pipelines tune over the same class of robust decision maps and differ only in whether the training loss is defined through that decision map. For practitioners, these results indicate that robustness should be embedded in the learning objective itself rather than appended after a predictor has been trained for a nominal problem. Moreover, our empirical results show that gradient-based refinement of the $\mathcal{RSPO}_+$ solution yields a further improvement in out-of-sample decision quality.

\vspace{2mm}
\noindent\textbf{Organization.}~The remainder of the paper is organized as follows. Section~\ref{sec:literature_review} reviews the related literature. Section~\ref{sec:model} formalizes the ILRO framework: it establishes the well-posedness and projection representation of the robust decision map, defines the $\mathcal{RSPO}$ loss through this map, and states the empirical learning problem. Section~\ref{sec:optimization_schemes} develops two complementary training schemes: the convex surrogate $\mathcal{RSPO}_+$ loss, together with its finite-dimensional reformulation (Section~\ref{subsec:convex_surrogate}) and Fisher consistency result characterizing the admissible range of the robustness parameter~$\gamma$ (Section~\ref{subsec:fisher_consistency}), and a gradient-based method that directly refines the empirical $\mathcal{RSPO}$ objective (Section~\ref{subsec:gradient_refinement}). Section~\ref{sec:statistical_guarantee} establishes finite-sample guarantees: excess risk bounds for the $\mathcal{RSPO}$ predictor via a meta generalization bound (Section~\ref{subsec:rspo_excess_risk_bounds}) and target excess risk bounds for the $\mathcal{RSPO}_+$ predictor within the Fisher-consistent regime (Section~\ref{subsec:rspo+_excess_risk_bounds}). Section~\ref{sec:numeric_study} reports numerical experiments on capacitated transportation and portfolio optimization problems, and Section~\ref{sec:conclusion} concludes. The electronic companion collects the technical lemmas and proofs (Appendices~\ref{app:technical_lemmas} and~\ref{app:technical_proofs}), the stochastic gradient-descent algorithm (Appendix~\ref{app:gradient_algorithm}), concrete excess risk bounds for four representative hypothesis classes (Appendix~\ref{app:gen_examples}), and additional experimental details (Appendix~\ref{app:experimental_details}).

\section{Literature Review}
\label{sec:literature_review}

The learning--decision pipelines introduced in Section~\ref{sec:introduction} belong to the contextual stochastic optimization literature surveyed by \cite{sadana2025survey}. To position ILRO, we first distinguish predict-then-optimize methods from integrated learning and optimization according to whether the prediction model is trained independently of or through the downstream problem. We then discuss robust optimization and efforts to incorporate robustness into learning--decision pipelines.

\subsection{Predict-then-Optimize Methods}

Predict-then-optimize methods first fit a prediction model using a statistical criterion and then pass its output to the downstream problem. In the linear-objective setting considered here, a point predictor estimates the conditional mean cost vector. Other implementations approximate the conditional distribution from contextual data. \cite{bertsimas2020predictive} use nonparametric methods to construct context-dependent weights for a weighted sample average approximation and establish asymptotic optimality for their $k$-nearest-neighbor and kernel-based prescriptions. \cite{kannan2025data} combine fitted predictions with empirical residuals to generate scenarios for sample average approximation and provide asymptotic and finite-sample guarantees.

Theoretical studies have compared PTO and ILO in terms of downstream regret. For contextual linear optimization, \cite{hu2022fast} compare the naive plug-in approach with its integrated counterpart over the same prediction class. Under correct specification of that class and a condition limiting near-dual-degeneracy, they show that the plug-in approach can attain faster regret rates. For a general class of nonlinear stochastic optimization problems with parametric distribution models, \cite{elmachtoub2023estimate} compare PTO with ILO. They show that PTO has stochastically smaller asymptotic regret when the model class is well specified and sufficient data are available, whereas integrated estimation can perform better under misspecification. These results show that the relative performance of the two approaches depends on problem structure and model specification.

\subsection{Integrated Learning and Optimization}

Integrated learning and optimization, also referred to as decision-focused learning, trains the prediction model through the downstream optimization problem rather than through a separate statistical criterion. For contextual linear optimization, our work builds most directly on \cite{elmachtoub2022smart}, who introduce the $\mathcal{SPO}$ loss, derive its convex upper bound $\mathcal{SPO}_+$, and establish Fisher consistency under distributional and geometric conditions. Subsequent work develops statistical guarantees for both the target and surrogate losses. \cite{elbalghiti2023generalization} derive generalization bounds for the $\mathcal{SPO}$ loss using margin and combinatorial-complexity arguments. For $\mathcal{SPO}_+$, \cite{liu2021risk} establish risk bounds and calibration results that transfer surrogate excess risk to $\mathcal{SPO}$ excess risk. Beyond these loss-specific results, \cite{honguyen2022risk} identify conditions under which prediction losses yield asymptotic consistency and nonasymptotic guarantees for optimization risk. The key difference between the ILO and ILRO frameworks is that the former uses a nominal decision problem, whereas the latter uses a robust decision problem. At the same time, ILRO essentially preserves the learning--decision alignment of the ILO framework, as each framework uses its deployed decision rule to define the training loss. Replacing the nominal decision map with a robust one leads to a distinct learning problem, for which we provide a comprehensive computational and statistical analysis. On the computational side, we construct the convex $\mathcal{RSPO}_+$ surrogate and exploit the differentiability of the robust decision map to develop a first-order method for optimizing the $\mathcal{RSPO}$ loss directly. On the statistical side, we characterize an explicit robustness threshold for the surrogate's Fisher consistency and establish finite-sample excess-risk guarantees for both predictors.

Another line develops gradient-based methods for solving ILO models through implicit differentiation under suitable regularity conditions or through smoothing techniques. \cite{amos2017optnet} differentiate the KKT system of a quadratic program, and \cite{donti2017task} use this implicit-differentiation approach to train probabilistic models through stochastic programs. Among smoothing methods, \cite{berthet2020learning} use stochastic perturbations to construct a differentiable expected optimizer, whereas \cite{wilder2019melding} add quadratic smoothing to a continuous relaxation of a linear combinatorial problem to obtain training gradients. Overall, these methods still focus on solving ILO problems with a nominal downstream decision map and use differentiation or smoothing to mitigate computational difficulties in training. We refer to \cite{mandi2024decision} for a survey of decision-focused learning.

\subsection{Robust Optimization in Learning--Decision Pipelines}

Robust optimization addresses parameter uncertainty by optimizing against realizations in a prescribed uncertainty set \citep{bental1999robust,bertsimas2004price,bental2009robust} and has been used across operations management \citep{lu2021review}. Most relevant here are models with uncertain objective coefficients, including discrete optimization and network flows with cost uncertainty \citep{bertsimas2003robust} and portfolio models with uncertain return and risk parameters \citep{goldfarb2003robust}. 
When uncertainty concerns a probability distribution rather than individual parameter realizations, distributionally robust optimization (DRO) optimizes against the worst-case distribution in an ambiguity set consistent with the available information \citep{delage2010distributionally,wiesemann2014distributionally,gao2023distributionally,kuhn2025distributionally}.

Several studies incorporate robustness into ILO from different perspectives. \cite{chenreddy2024endtoend} and \cite{yeh2025endtoend} use downstream decision-making losses to learn uncertainty sets for conditional robust optimization rather than point predictors of uncertain objective coefficients. For objective-coefficient predictors, \cite{schutte2024robust} modify the $\mathcal{SPO}$ loss by replacing the benchmark of the nominal clairvoyant optimal value with a robust counterpart. \cite{costa2023distributionally} develop a portfolio-specific DRO pipeline based on a $\phi$-divergence ball. \cite{im2025smart} extend $\mathcal{SPO}$ to robust constraints by replacing its fixed feasible set with a robust feasible set that enforces the constraints for every constraint-parameter realization in a separately constructed uncertainty set. They provide a convex surrogate for the resulting loss and establish the surrogate's Fisher consistency with respect to that loss. ILRO instead robustifies the predicted objective coefficients, leading to different computational and statistical challenges.

In addition, some studies introduce robustness into the downstream decision problem within a PTO pipeline. \cite{kannan2024residuals} first fit a prediction model and then use its residuals to construct an empirical distribution and corresponding ambiguity sets for the downstream DRO problem. \cite{sim2025analytics} develop a residual-based robust satisficing framework using a fitted regression model, followed by a fortification step to account for parameter-estimation uncertainty. Further examples include \cite{hu2025budget} and \cite{mao2026predictive}. In these approaches, predictor fitting remains separate from downstream robustness.

As the studies above illustrate, robustness is important for downstream decision-making when predictions are imprecise. It is therefore natural to incorporate downstream robustness into ILO. To the best of our knowledge, ILRO is the first attempt to do so while maintaining learning--decision alignment in the presence of objective-coefficient prediction errors. We further address the computational challenges and study the statistical properties of the ILRO framework.

\section{Model} 
\label{sec:model}
Recall that in the ILRO framework proposed in Section~\ref{sec:introduction}, the robust decision map~\ref{ilroopt} generates the deployed decision, and the learning problem~\ref{ILRO-learning} trains the predictor through that same decision map. What~\ref{ILRO-learning} minimizes is the average realized cost $\bm{y}_i^\top\bm{z}^{\star}_{\gamma}(\bm{g}_{\theta}(\bm{x}_i))$ of the induced decisions. This section formalizes this criterion as a loss function on prediction--realization pairs, following the route taken by \cite{elmachtoub2022smart} for the nominal decision map.

We begin by recalling their construction. For the ILO framework with a linear cost function, \cite{elmachtoub2022smart} propose the smart predict-then-optimize ($\mathcal{SPO}$) loss, defined as
\begin{equation*}
    \begin{aligned}
            \ell_{\mathcal{SPO}}(\hat{\bm{y}},\bm{y}):=\max_
            {\bm{z}\in \bm{Z}^{\star}(\hat{\bm{y}})}\bm
            {y}^\top\bm{z}-v^{\star}(\bm{y}),
    \end{aligned}
\end{equation*}
where $\bm{Z}^{\star}(\bm{y})=\argmin_{\bm{z}\in \mathcal{Z}}\bm{y}^\top\bm{z}$ denotes the optimal solution set. When this set is a singleton, or when a fixed selection rule is understood, we denote the corresponding optimizer by $\bm{z}^{\star}(\bm{y})$. The associated optimal value is denoted by $v^{\star}(\bm{y})=\min_{\bm{z}\in \mathcal{Z}}\bm{y}^\top\bm{z}$. The loss function characterizes the discrepancy between a predicted parameter $\hat{\bm{y}}$ and the true parameter $\bm{y}$ through their induced decisions. The inner maximization selects the worst decision from the optimal set under $\hat{\bm{y}}$ and evaluates its performance under $\bm{y}$, thereby quantifying the decision loss caused by prediction error. Based on the $\mathcal{SPO}$ loss, the empirical risk minimization problem takes the form
\begin{equation*}
    \min_{\theta\in\Theta} \frac{1}{N}\sum_{i\in[N]} \ell_{\mathcal{SPO}}(\bm{g}_{\theta}(\bm{x}_i),\bm{y}_i)=\min_{\theta\in\Theta}\frac{1}{N}\sum_{i\in[N]} \left(\max_
            {\bm{z}\in\bm{Z}^{\star} (\bm{g}_{\theta}(\bm{x}_i))}\bm
            {y}_i^\top\bm{z}-v^{\star}(\bm{y}_i)\right),
\end{equation*}
which is equivalent to the learning part of the ILO framework up to the choice of the optimal decision. Since the $\mathcal{SPO}$ loss is generally nonconvex and computationally difficult to minimize, \cite{elmachtoub2022smart} propose $\mathcal{SPO}_+$ as a convex surrogate.

Following a similar spirit, we define a robust counterpart of the $\mathcal{SPO}$ loss that evaluates predictions through the robust decision map in~\ref{ilroopt}. Lemma~\ref{lem:oracle_projection} presents several elementary properties of that map, which formalize the structural advantage discussed in Section~\ref{sec:introduction} and are used repeatedly in Sections~\ref{sec:optimization_schemes} and~\ref{sec:statistical_guarantee}.

\begin{lemma}[Well-Posedness]
\label{lem:oracle_projection}
Let $\gamma>0$ and let $\mathcal{Z}\subseteq\mathbb{R}^d$ be a nonempty compact convex set. Then, for every $\bm{y}_0\in\mathbb{R}^d$, the objective of the robust decision problem in~\ref{ilroopt} is strongly convex in $\bm{z}$, so the problem admits a unique minimizer $\bm{z}^{\star}_{\gamma}(\bm{y}_0)$. Moreover, completing the square yields the projection representation
\begin{equation}\label{eq:projection_representation}
    \bm{z}_{\gamma}^{\star}(\bm{y}_0)
    =
    \Pi_{\mathcal{Z}}\!\left(-\frac{\bm{y}_0}{\gamma}\right),
\end{equation}
where $\Pi_{\mathcal{Z}}(\cdot)$ denotes the Euclidean projection onto $\mathcal{Z}$. Thus, $\bm z_{\gamma}^{\star}(\cdot)$ is Lipschitz continuous with constant $1/\gamma$ and differentiable almost everywhere.
\end{lemma}
Since for given $\bm y_0$, the minimizer $\bm{z}^{\star}_{\gamma}(\bm{y}_0)$ is unique when $\gamma>0$, we henceforth write $=$ in place of $\in$. We also write
\begin{equation*}
    v^{\star}_{\gamma}(\bm{y}_0)
    :=
    \min_{\bm{z}\in\mathcal{Z}}\left\{\bm{y}_0^\top\bm{z}+\frac{\gamma}{2}\|\bm{z}\|_2^2\right\}
\end{equation*}
for the corresponding optimal value.

\begin{definition}[$\mathcal{RSPO}$ Loss]
\label{def:rspo_loss}
Given a vector prediction $\hat{\bm{y}}$ and a realized cost vector $\bm{y}$, the  $\mathcal{RSPO}$ loss (short for Robust $\mathcal{SPO}$ loss) is defined by
\begin{equation*}
    \ell_{\mathcal{RSPO}}(\hat{\bm{y}},\bm{y})
    :=
    \bm{y}^\top\bm{z}^{\star}_{\gamma}(\hat{\bm{y}})-v^{\star}(\bm{y}).
\end{equation*}
\end{definition}
The $\mathcal{RSPO}$ loss is well-defined, since the robust decision map $\bm{z}^{\star}_{\gamma}(\cdot)$ is single-valued by Lemma~\ref{lem:oracle_projection}. It measures the excess realized cost of the robust decision induced by $\hat{\bm{y}}$ under the realized cost $\bm{y}$, relative to a benchmark being the \emph{nominal} clairvoyant optimal value $v^{\star}(\bm{y})$. Notice that minimizing the empirical $\mathcal{RSPO}$ risk is equivalent to the learning problem~\ref{ILRO-learning}, regardless of whether the benchmark is taken as $v^{\star}(\bm{y})$ or $v_{\gamma}^{\star}(\bm{y})$. We intentionally choose $v^{\star}(\bm{y})$ as the benchmark, since it ensures $\ell_{\mathcal{RSPO}}\geq 0$, which facilitates our analysis of the excess risk in Section~\ref{sec:statistical_guarantee}.\footnotemark\footnotetext{In contrast, if $v^{\star}_{\gamma}(\bm{y})$ is chosen as the benchmark, the loss may take negative values as it measures the regret against a conservative benchmark.}

Given an i.i.d. sample $\{(\bm{x}_i,\bm{y}_i)\}_{i\in[N]}$, we define the empirical $\mathcal{RSPO}$ learning problem as
\begin{equation}\label{eq:rspo_problem}
    \min_{\bm{g}\in\mathcal{G}}\frac{1}{N}\sum_{i\in[N]}\ell_{\mathcal{RSPO}}(\bm{g}(\bm{x}_i),\bm{y}_i)+\lambda \Omega(\bm{g}),
\end{equation}
where $\mathcal{G}=\{\bm{g}_\theta:\theta\in\Theta\}$ is the hypothesis class and $\lambda \Omega(\bm{g})$ is the regularization term with $\lambda\ge0$. Here, $\Omega:\mathcal{G}\to[0,\infty)$ is a convex regularizer and for a parametric class we write $\Omega(\bm{g}_\theta)$.

Problem~\eqref{eq:rspo_problem} completes the ILRO formulation by training the predictor according to the same robust decision map used at deployment. The resulting problem is nonconvex, as is the $\mathcal{SPO}$ objective it generalizes. We will discuss the computational aspects in the next section. We now turn to some discussions on the components of the ILRO framework.

\begin{remark}[Robustness interpretation of \ref{ilroopt}]
\label{rem:soft-robustness}
We start from the classical robust problem with a ball uncertainty set below
\begin{equation*}
    \min_{\bm{z}\in\mathcal{Z}}\;\max_{\|\bm{u}\|_2\leq \vartheta}\left\{(\hat{\bm{y}}+\bm{u})^\top\bm{z}\right\},
\end{equation*}
which is equivalent to
\begin{equation*}   \min_{\bm{z}\in\mathcal{Z}}\hat{\bm{y}}^\top\bm{z}+\vartheta\|\bm{z}\|_2.
\end{equation*}
Both this problem and \ref{ilroopt} penalize the magnitude $\|\bm z\|_2$ of the decision, and the two families are equivalent at the level of solution paths: for any $\vartheta\geq 0$, every solution of the problem above solves \ref{ilroopt} for some $\gamma\in[0,\infty]$, and vice versa. Moreover, along this correspondence a larger $\gamma$ corresponds to a larger $\vartheta$, thereby yielding a more robust formulation of the downstream decision problem. We formalize these claims in Lemma~\ref{lem:path_equivalence} in Appendix~\ref{app:technical_lemmas}. When $\gamma=0$, \ref{ilroopt} reduces to the nominal ILO counterpart.
\hfill\Halmos

\end{remark}

\begin{remark}[Robustification through a general strongly convex penalty]
\label{rem:general_phi}
The $\ell_2$ ball in Remark~\ref{rem:soft-robustness} is not essential. Let
$\phi:\mathbb R^{d}\to\mathbb R$ be closed and $\mu$-strongly convex with $\mu>0$ and let $\phi^{*}$ be its convex conjugate. Since $\max_{\bm u}\{\bm u^{\top}\bm z-\phi^{*}(\bm u)\}=\phi(\bm z)$, the decision map
\begin{equation*}
    \bm z^{\star}_{\phi}(\hat{\bm y})
    :=\argmin_{\bm z\in\mathcal Z}\;\max_{\bm u\in\mathbb R^{d}}
      \left\{(\hat{\bm y}+\bm u)^{\top}\bm z-\phi^{*}(\bm u)\right\}
    =\argmin_{\bm z\in\mathcal Z}\left\{\hat{\bm y}^{\top}\bm z+\phi(\bm z)\right\}
\end{equation*}
is also an exact robust counterpart of the nominal problem, with an adversary perturbing the predicted cost by $\bm u$ at the price $\phi^{*}(\bm u)$. The choice $\phi=\frac{\gamma}{2}\|\cdot\|_2^{2}$, hence $\phi^{*}=\frac{1}{2\gamma}\|\cdot\|_2^{2}$, recovers \ref{ilroopt}, which leads to another robust optimization interpretation of \ref{ilroopt}.
 
We remark that all our results related to \ref{ilroopt} apply to this general robustification model with minimal modifications ($1/\gamma$ replaced by $1/\mu$). Notice that $\phi$ is $\mu$-strongly convex, which implies $\nabla\phi^{*}$ is $(1/\mu)$-Lipschitz. Thus, if $\phi^{*}$ is conic representable, $\phi$ is twice differentiable, and $L$-smooth, all technical results hold. The details are straightforward yet tedious, and hence we omit them. Accordingly, we model robustness through the $\ell_2$ ball throughout the paper, as it captures all essential properties of the framework.
\hfill\Halmos
\end{remark}

\begin{remark}[Structural property of the learning objective]
    Before concluding this section, we present a structural property of the empirical learning problem with $\mathcal{RSPO}$ loss, which may be of independent interest.
\begin{proposition}[Continuous Piecewise Affine Objective]
    \label{prop:rspo_piecewise_affine}
    Let $\gamma>0$, let $\mathcal Z=\{\bm z\in\mathbb R^{d}:\bm A\bm z\geq\bm b\}$ be a nonempty bounded polyhedron with $\bm A\in\mathbb R^{m\times d}$, $\bm b\in\mathbb R^m$, and let the prediction class be linear, $\bm g_{\theta}(\bm x)=\bm B\bm x$ with $\theta=\bm B\in \mathbb R^{d\times p}$. Then the unregularized ($\lambda=0$) empirical $\mathcal{RSPO}$ objective
    \begin{equation*}
        \bm B\ \longmapsto\ \frac{1}{N}\sum_{i\in[N]} \big[\bm y_i^\top\bm z^{\star}_{\gamma}(\bm B\bm x_i)-v^{\star}(\bm y_i)\big]
    \end{equation*}
    is continuous and piecewise affine: there is a finite family of full-dimensional polyhedra covering $\mathbb R^{d\times p}$ on the interior of each of which the objective is affine, with a constant gradient. 
\end{proposition}
To better understand this structural result, consider the case of $\mathcal{SPO}$. Here, since $\bm z^\star(\bm B \bm x_i)$ belongs to the finite set of extreme points, the resulting objective is piecewise constant and hence discontinuous (except in the trivial degenerate case). Thus, Proposition~\ref{prop:rspo_piecewise_affine} explains why the ILRO framework is more amenable to first-order methods in computation, since its objective is continuous and has informative gradient almost everywhere. Moreover, the piecewise affine structure may inspire design of approximation algorithms, which we leave for future study. \hfill\Halmos
\end{remark}

\section{Training Schemes and Fisher Consistency}
\label{sec:optimization_schemes}

In this section, we develop two complementary schemes to address the computational difficulty of ILRO. The first scheme, in Section~\ref{subsec:convex_surrogate}, replaces the target loss with a convex and differentiable surrogate, the $\mathcal{RSPO}_+$ loss. We develop the Fisher consistency theory of this surrogate in Section~\ref{subsec:fisher_consistency}, which identifies the regime of $\gamma$ where minimizing the surrogate also minimizes the target risk.
 
The second scheme, in Section~\ref{subsec:gradient_refinement}, exploits the structural advantage of the robust decision map discussed in Section~\ref{sec:introduction}---the decision map is Lipschitz continuous and differentiable almost everywhere---to apply gradient-descent-type algorithms directly to the empirical $\mathcal{RSPO}$ learning problem \eqref{eq:rspo_problem}. We can also initialize the gradient approach using the surrogate solution.

\subsection{Convex Surrogate: $\mathcal{RSPO}_+$ Loss}
\label{subsec:convex_surrogate}

Throughout Sections~\ref{subsec:convex_surrogate} and~\ref{subsec:fisher_consistency}, we let $\gamma>0$ and let $\mathcal Z\subseteq\mathbb R^d$ be a nonempty compact convex set, unless otherwise specified. We construct a convex surrogate for the generally nonconvex $\mathcal{RSPO}$ loss in three steps: we first reformulate the target loss as the limit of a family of penalized problems indexed by a scalar $a>0$; we then fix a finite $a$, which yields an upper bound; finally, we linearize the concave value function $v^{\star}_{\gamma}(\cdot)$ at the realized cost, which leads to our proposed surrogate loss that is convex with respect to the prediction.

For the first step, fix $\hat{\bm{y}}$ and $\bm{y}$ and define
\begin{equation}
\label{eq:q-function}
    q(a):=\max_{\bm{z}\in\mathcal{Z}}\left\{\bm{y}^\top\bm{z}-a\hat{\bm{y}}^\top\bm{z}-\frac{a\gamma}{2}\|\bm{z}\|^2_2\right\}+av^{\star}_{\gamma}(\hat{\bm{y}}),\quad a>0.
\end{equation}
The following proposition connects the function $q$ to the target loss.
 
\begin{proposition}[Limiting Representation]\label{prop:dual_representation}
    For fixed $\hat{\bm{y}}$ and $\bm{y}$, the function $q(\cdot)$ is convex and nonincreasing on $(0,\infty)$, and the $\mathcal{RSPO}$ loss admits the representation
\begin{equation*}
    \ell_{\mathcal{RSPO}}(\hat{\bm{y}},\bm{y})=\lim_{a\to+\infty} q(a)-v^{\star}(\bm{y}).
\end{equation*}
\end{proposition}

Since $q$ is nonincreasing, fixing any finite $a>0$ in Proposition~\ref{prop:dual_representation} yields the upper bound
$\ell_{\mathcal{RSPO}}(\hat{\bm{y}},\bm{y})\leq q(a)-v^{\star}(\bm{y})$, which completes the second step. This bound, however, is not yet convex in the prediction: the maximization term is convex in $\hat{\bm{y}}$, but the term $av^{\star}_{\gamma}(\hat{\bm{y}})$ is concave. The third step removes the concave term by linearization. Since the robust problem has a unique
minimizer, $v^{\star}_{\gamma}(\cdot)$ is differentiable with $\nabla v^{\star}_{\gamma}(\bm{y})=\bm{z}^{\star}_{\gamma}(\bm{y})$ by Danskin's theorem, and concavity gives $v^{\star}_{\gamma}(\hat{\bm{y}})\leq v^{\star}_{\gamma}(\bm{y})+\bm{z}^{\star}_{\gamma}(\bm{y})^{\top}(\hat{\bm{y}}-\bm{y})$. Combining the three steps,
\begin{equation*}
    \begin{aligned}
        \ell_{\mathcal{RSPO}}(\hat{\bm{y}},\bm{y})&\leq\max
        _{\bm{z}\in\mathcal{Z}}\left\{\bm{y}^\top\bm{z}-a\hat{\bm{y}}^\top\bm{z}
        -\frac{a\gamma}{2}\|\bm{z}\|^2_2\right\}+av^{\star}_{\gamma}(\hat{\bm{y}})-v^{\star}(\bm{y})\\
        &\leq\max
        _{\bm{z}\in\mathcal{Z}}\left\{\bm{y}^\top\bm{z}-a\hat{\bm{y}}^\top\bm{z}
        -\frac{a\gamma}{2}\|\bm{z}\|^2_2\right\}+a\left\{v^{\star}_{\gamma}(\bm{y})+\bm{z}^{\star}_{\gamma}(\bm{y})^{\top}(\hat{\bm{y}}-\bm{y})\right\}-v^{\star}(\bm{y})\\
        &=\max
        _{\bm{z}\in\mathcal{Z}}\left\{\bm{y}^\top\bm{z}-a\hat{\bm{y}}^\top\bm{z}
        -\frac{a\gamma}{2}\|\bm{z}\|^2_2\right\}+a\left\{\bm{z}^{\star}_{\gamma}(\bm{y})^{\top}\hat{\bm{y}}+\frac{\gamma}{2}\|\bm{z}^{\star}_{\gamma}(\bm{y})\|^2_2\right\}-v^{\star}(\bm{y}),
    \end{aligned}
\end{equation*}
where the equality uses $v^{\star}_{\gamma}(\bm{y})=\bm{y}^\top\bm{z}^{\star}_{\gamma}(\bm{y})+\frac{\gamma}{2}\|\bm{z}^{\star}_{\gamma}(\bm{y})\|_2^2$. The resulting expression is convex in $\hat{\bm y}$ and upper-bounds the $\mathcal{RSPO}$ loss, which leads to the following convex surrogate.
 
\begin{definition}[$\mathcal{RSPO}_+$ Loss]
    \label{def:rspo+_loss}
    Given a vector prediction $\hat{\bm{y}}$, a realized cost vector $\bm{y}$, and a fixed scalar $a>0$, the $\mathcal{RSPO}_+$ loss is defined as
    \begin{equation*}
        \ell_{\mathcal{RSPO}_+}(\hat{\bm{y}},\bm{y}):=\max_{\bm{z}\in\mathcal{Z}}\left\{\bm{y}^\top\bm{z}-a\hat{\bm{y}}^\top\bm{z}
        -\frac{a\gamma}{2}\|\bm{z}\|^2_2\right\}+a\left\{\bm{z}^{\star}_{\gamma}(\bm{y})^\top\hat{\bm{y}}+\frac{\gamma}{2}\|\bm{z}^{\star}_{\gamma}(\bm{y})\|^2_2\right\}-v^{\star}(\bm{y}).
    \end{equation*}
\end{definition}
 
\begin{remark}[Role of the parameter $a$]\label{rem:role_of_a}
The two approximation steps introduce two gaps with opposite behavior in $a$. The penalty gap $q(a)-\lim_{a'\to\infty}q(a')$ vanishes as $a\to+\infty$, whereas the linearization gap equals $a$ times the ($a$-independent) Bregman gap $v^{\star}_{\gamma}(\bm{y})+\bm{z}^{\star}_{\gamma}(\bm{y})^{\top}(\hat{\bm{y}}-\bm{y})-v^{\star}_{\gamma}(\hat{\bm{y}})\geq0$ and therefore grows linearly in $a$. The tightness of $\ell_{\mathcal{RSPO}_+}$ is thus not monotone in $a$, and no single choice dominates pointwise. Section~\ref{subsec:fisher_consistency} singles out $a=1$ on statistical grounds.\hfill\Halmos
\end{remark}
 
\begin{remark}[Relation to the $\mathcal{SPO}_+$ construction]
Setting $\gamma=0$ in Definition~\ref{def:rspo+_loss} yields $\max_{\bm z\in\mathcal Z}\{\bm y^\top\bm z-a\hat{\bm y}^\top\bm z\}+a\bm z^{\star}(\bm y)^\top\hat{\bm y}-v^{\star}(\bm y)$, which for $a=2$ is precisely the $\mathcal{SPO}_+$ loss of~\cite{elmachtoub2022smart}. The $\mathcal{RSPO}_+$ loss is therefore the exact robust counterpart of the $\mathcal{SPO}_+$ loss, while for $\gamma>0$ the quadratic term additionally renders the loss differentiable, as shown next.\hfill\Halmos
\end{remark}
 
We next collect the analytical properties of the surrogate.
 
\begin{proposition}\label{prop:rspo+_loss_properties}
Given a fixed $\bm{y}\in\mathcal Y$, it holds that:
    \begin{enumerate}
        \item[(i)] $0\leq\ell_{\mathcal{RSPO}}(\hat{\bm{y}},\bm{y})\leq \ell_{\mathcal{RSPO}_+}(\hat{\bm{y}},\bm{y})$ for all $\hat{\bm{y}}$;
        \item[(ii)] $\ell_{\mathcal{RSPO}_+}(\cdot,\bm{y})$ is convex on $\mathbb R^d$;
        \item[(iii)] $\ell_{\mathcal{RSPO}_+}(\cdot,\bm{y})$ is differentiable everywhere, with gradient
        at $\hat{\bm{y}}$ given by
        \begin{equation*}
           a \left(\bm{z}^{\star}_{\gamma}(\bm{y})- \bm{z}^{\star}_{\gamma}\left(\hat{\bm{y}}-\frac{1}{a}\bm{y}\right) \right);
        \end{equation*}
        \item[(iv)] $\ell_{\mathcal{RSPO}_+}(\cdot,\bm{y})$ is $a\cdot D_{\mathcal Z}$-Lipschitz continuous,
        where $D_{\mathcal Z}:=\sup_{\bm z_1,\bm z_2\in\mathcal Z}\|\bm z_1-\bm z_2\|_2$ denotes the diameter of the feasible set $\mathcal Z$.
    \end{enumerate}
\end{proposition}
 
The differentiability in (iii) follows from the uniqueness induced by the strong convexity of the decision-problem objective, in contrast to the generally nondifferentiable $\mathcal{SPO}_+$ loss. This proposition is used repeatedly in the risk analysis of Section~\ref{sec:statistical_guarantee}.
 
Having established the convexity and differentiability of the $\mathcal{RSPO}_+$ loss, we now use it to define a surrogate learning problem. Replacing the generally nonconvex $\mathcal{RSPO}$ loss in Problem~\eqref{eq:rspo_problem} with $\mathcal{RSPO}_+$ yields the surrogate empirical risk minimization (ERM) problem:
\begin{equation}\label{eq:rspo+_ermprob}
    \min_{\theta\in\Theta} \frac{1}{N}\sum_{i\in[N]}\ell_{\mathcal{RSPO}_+}(\bm{g}_{\theta}(\bm{x}_i),\bm{y}_i)+\lambda\Omega(\bm{g}_{\theta}),
\end{equation}
where $\{\bm g_\theta:\theta\in\Theta\}$ is the prediction class, $\Omega(\cdot)$ is a regularizer on the prediction model, and $\lambda\geq 0$ controls the regularization strength. The surrogate ERM can be optimized by first-order methods based on the gradient derived in Proposition~\ref{prop:rspo+_loss_properties}, because when $\bm g_\theta$ is linear in $\theta$, the problem \eqref{eq:rspo+_ermprob} is convex in $\theta$. Note that this is consistent with the convexity requirements for empirical risk minimization of the $\mathcal{SPO}_+$ loss and standard statistical losses such as squared loss. Moreover, for a general prediction class $\{\bm g_\theta:\theta\in\Theta\}$, whose members are differentiable in $\theta$ (e.g., multilayer perceptron), the problem can be trained by (stochastic) gradient algorithms.

In addition to first-order methods, we can also resort to a reformulation approach. For a bounded polyhedral feasible set, Problem~\eqref{eq:rspo+_ermprob} admits an explicit single-level reformulation through duality.
 
\begin{theorem}[Reformulation of the $\mathcal{RSPO}_+$ ERM]
\label{thm:rspo_plus_reformulation}
     Suppose $\mathcal Z=\{\bm{z}\in\mathbb{R}^{d}:\bm{A}\bm{z}\geq\bm{b}\}$ is a nonempty bounded polyhedron, and fix $a>0$ and $\gamma>0$. Then Problem~\eqref{eq:rspo+_ermprob} is equivalent to
     \begin{equation*} 
        \begin{aligned}
            \min_{\theta,\bm{p}_i}\ &\frac{1}{N} \!\sum_{i\in[N]}\! \bigg\{ \frac{1}{2a\gamma} \|\bm{y}_i\!-\!a\bm{g}_{\theta}(\bm{x}_i)\!+\!\bm{A}^\top\!\bm{p}
            _i\|^2_2 \!-\! \bm{b}^\top\bm{p}_i\! +\!a\bm{z}^{\star}_{\gamma}(\bm{y}_i)^\top\!\bm{g}_{\theta}(\bm{x}_i)
            \!+\! \frac{a\gamma}{2} \|\bm{z}^{\star}_{\gamma}(\bm{y}_i)\|^2_2\!-\!v^{\star}(\bm{y}_i)\bigg\}\!+\!\lambda\Omega(\bm{g}_{\theta})\\
            \mathrm{s.t.}\ &\bm{p}_i\in \mathbb R^m_+,\, i\in[N],\, \theta\in\Theta,
        \end{aligned}
    \end{equation*}
    in the sense that the two problems have the same optimal value and $\theta^{\star}$ is optimal for Problem~\eqref{eq:rspo+_ermprob} if and only if there exist $\{\bm p_i^{\star}\}$ such that  $(\theta^{\star},\{\bm p_i^{\star}\})$ is optimal for the reformulation.
\end{theorem}
Similarly to Problem~\eqref{eq:rspo+_ermprob}, the reformulation as given by Theorem~\ref{thm:rspo_plus_reformulation} is convex when $\bm g_\theta(\bm x)$ is linear in $\theta$. For example, when $\bm{g}_{\theta}(\bm{x})=\bm{B}\bm{x}$ and $\Omega(\bm{g}_{\theta})=\|\bm B\|_1$ or $\Omega(\bm{g}_{\theta})=\|\bm B\|_F^2$, it can be formulated as a convex quadratic program using standard epigraph reformulations.

To examine whether minimizing this tractable surrogate is statistically aligned with the target loss, we next study the Fisher consistency of $\mathcal{RSPO}_+$ with respect to $\mathcal{RSPO}$.

\subsection{Fisher Consistency}\label{subsec:fisher_consistency}
We now examine the optimality validity of the $\mathcal{RSPO}_+$ loss by asking whether minimizing this surrogate population risk also minimizes the target $\mathcal{RSPO}$ risk. As we show below, this property---known as Fisher consistency (see, e.g., \citealp{elmachtoub2022smart})---need not hold for all values of $\gamma$. We therefore characterize the regime of $\gamma$ under which it holds. We begin with the formal definition.
 
\begin{definition}
    Let $\mathbb{P}$ denote the joint distribution of $(\bm{x},\bm{y})$, and let
    $\mathcal{G}_{\rm all}$ denote the class of all measurable functions mapping $\mathcal X$ to
    $\mathbb R^d$. A loss function $\ell(\cdot,\cdot)$ is said to be \textit{Fisher consistent}
    with respect to the $\mathcal{RSPO}$ loss if every minimizer of
    \begin{equation}\label{eq:fisher_l}
        \min_{\bm{g}\in\mathcal{G}_{\rm all}}\mathbb{E}_{(\bm{x},\bm{y})\sim\mathbb{P}}[\ell(\bm{g}(\bm{x}),\bm{y})]
    \end{equation}
    is also a minimizer of
    \begin{equation}\label{eq:fisher_rspo}
        \min_{\bm{g}\in\mathcal{G}_{\rm all}}\mathbb{E}_{(\bm{x},\bm{y})\sim\mathbb{P}}[\ell_{\mathcal{RSPO}}(\bm{g}(\bm{x}),\bm{y})].
    \end{equation}
\end{definition}
 
Fisher consistency is a population-level property over the unrestricted class of measurable predictors. It rules out any intrinsic bias from replacing $\ell_{\mathcal{RSPO}}$ with $\ell_{\mathcal{RSPO}_+}$ when data are unlimited. Because the minimization in \eqref{eq:fisher_l} and \eqref{eq:fisher_rspo} is over $\mathcal{G}_{\rm all}$, the interchangeability principle (e.g., \citealp[Theorem~14.60]{rockafellar1998variational}) reduces both problems to pointwise minimization: for any loss $\ell$,
\begin{equation*}
    \min_{\bm g\in\mathcal{G}_{\rm all}}\mathbb{E}_{(\bm{x},\bm{y})\sim\mathbb{P}}[\ell(\bm{g}(\bm{x}),\bm{y})]
    =\mathbb{E}_{\bm{x}\sim\mathbb{P}_{\bm{x}}}\left[\min_{\hat{\bm{y}}\in\mathbb{R}^d}\mathbb{E}_{\bm{y}\sim\mathbb{P}_{\bm{y}|\bm{x}}}\left[\ell(\hat{\bm{y}},\bm{y})\right]\right],
\end{equation*}
where $\mathbb{P}_{\bm x}$ and $\mathbb{P}_{\bm y|\bm x}$ denote the marginal and conditional distributions. All statements in this subsection are understood pointwise, and the population statements follow by integrating over $\mathbb{P}_{\bm x}$. For a given $\bm x$, define the sets of minimizers of the pointwise target and surrogate risks,
\begin{equation*}
\mathcal{Y}^{\star}_{\mathcal{RSPO}}:=\argmin_{\hat{\bm{y}}\in\mathbb{R}^d}
\mathbb{E}_{\bm{y}\sim\mathbb{P}_{\bm{y}|\bm{x}}}\left[\ell_{\mathcal{RSPO}}(\hat{\bm{y}},\bm{y})\right],
    \qquad
\mathcal{Y}^{\star}_{\mathcal{RSPO}_+}:=\argmin_{\hat{\bm{y}}\in\mathbb{R}^d}\mathbb{E}_{\bm{y}\sim\mathbb{P}_{\bm{y}|\bm{x}}}\left[\ell_{\mathcal{RSPO}_+}(\hat{\bm{y}},\bm{y})\right].
\end{equation*}
Fisher consistency then amounts to the inclusion $\mathcal{Y}^{\star}_{\mathcal{RSPO}_+}\subseteq\mathcal{Y}^{\star}_{\mathcal{RSPO}}$ holding for $\mathbb{P}_{\bm x}$-almost every $\bm x$. We characterize the two sets in turn---the target set in Proposition~\ref{prop:minRSPO}, the surrogate set in Proposition~\ref{prop:optcondition_RSPO+} and Corollary~\ref{cor:RSPO+a1}---and then determine when the inclusion holds.
 
\begin{proposition}\label{prop:minRSPO}
Suppose that $\mathbb{E}_{\bm{y}\sim\mathbb{P}_{\bm{y}|\bm{x}}}[\|\bm{y}\|_2]<\infty$ and the nominal problem $\min_{\bm{z}\in\mathcal{Z}} \bar{\bm{y}}^\top\bm{z}$ at the conditional mean $\bar{\bm{y}}=\mathbb{E}_{\bm{y}\sim\mathbb{P}_{\bm{y}|\bm{x}}}[\bm{y}]$ admits a unique optimal solution $\bm{z}^{\star}(\bar{\bm{y}})$, and write ${\mathcal{N}}_{\mathcal{Z}}(\bm{z}^{\star}(\bar{\bm{y}}))$ for the normal cone to $\mathcal{Z}$ at $\bm{z}^{\star}(\bar{\bm{y}})$. Then
\begin{equation}\label{eq:RSPO_normalcone}
    \mathcal{Y}^{\star}_{\mathcal{RSPO}}
    =-{\mathcal{N}}_{\mathcal{Z}}\big(\bm{z}^{\star}(\bar{\bm{y}})\big)-\gamma\,\bm{z}^{\star}(\bar{\bm{y}}).
\end{equation}
In particular, if $\mathcal Z=\{\bm{z}\in\mathbb{R}^{d}:\bm{A}\bm{z}\geq\bm{b}\}$ is a nonempty bounded polyhedron with $\bm{A}\in\bb{R}^{m\times d}$, then $-{\mathcal{N}}_{\mathcal{Z}}(\bm{z}^{\star}(\bar{\bm{y}}))=\{\bm{A}^\top\bm{\lambda}:\bm{\lambda}\in\bm{\Lambda}\}$, and 
\begin{equation*}
    \mathcal{Y}^{\star}_{\mathcal{RSPO}}=\{\bm{A}^\top\bm{\lambda}-\gamma\bm{z}^{\star}(\bar{\bm{y}}):\bm{\lambda}\in\bm{\Lambda}\},
\end{equation*}
where $\bm{\Lambda}:=\{\bm{\lambda}\geq\bm{0}:\ \lambda_i=0,\ \forall i\in[m]\setminus\mathcal{I}\}$ and $\mathcal{I}\subseteq[m]$ collects the indices of the constraints $\bm{a}_i^\top\bm{z}\geq b_i$ active at $\bm{z}^{\star}(\bar{\bm{y}})$, with $\bm{a}_i^\top$ being the $i$-th row of $\bm{A}$.
\end{proposition}

Proposition~\ref{prop:minRSPO} shows that the set of target-risk minimizers is a closed convex cone---the negative normal cone at the nominal decision---translated by $-\gamma\bm{z}^{\star}(\bar{\bm{y}})$. As a limiting case when $\gamma=0$, the translation vanishes and the set is a cone, hence invariant under positive scaling of the prediction. We also note that $\mathcal{Y}^{\star}_{\mathcal{RSPO}}$ is typically large and unbounded, since distinct predictions can induce the same robust decision.

It remains to locate the surrogate minimizers, which we do via the first-order optimality condition of the convex $\mathcal{RSPO}_+$ conditional risk.
 
\begin{proposition}\label{prop:optcondition_RSPO+}
     Suppose that $\mathbb{E}_{\bm{y}\sim\mathbb{P}_{\bm{y}|\bm{x}}}[\|\bm{y}\|_2]<\infty$. Then
    \begin{equation*}
      \mathcal{Y}^{\star}_{\mathcal{RSPO}_+}=  \left\{\hat{\bm{y}}\in\mathbb{R}^d:\mathbb{E}_{\bm{y}\sim\mathbb{P}_{\bm{y}|\bm{x}}}\left[\bm{z}^{\star}_{\gamma}(\bm{y})\right]=\mathbb{E}_{\bm{y}\sim\mathbb{P}_{\bm{y}|\bm{x}}}\left[\bm{z}^{\star}_{\gamma}\Big(\hat{\bm{y}}-\frac{1}{a}\bm{y}\Big)\right]\right\}.
    \end{equation*}
\end{proposition}
 
\begin{remark}[No continuity assumption on cost distribution] 
In contrast to the $\mathcal{SPO}_+$ setting, no continuity assumption on the distribution of $\bm{y}$ is needed in Proposition~\ref{prop:optcondition_RSPO+}, since $\ell_{\mathcal{RSPO}_+}(\cdot,\bm{y})$ is differentiable everywhere when $\gamma>0$ as shown in Proposition~\ref{prop:rspo+_loss_properties}.       \hfill\Halmos
\end{remark}
 
Proposition~\ref{prop:optcondition_RSPO+} characterizes the surrogate minimizers through a balance condition: $\hat{\bm{y}}$ must match, in expectation, the two robust decisions appearing in the gradient of Proposition~\ref{prop:rspo+_loss_properties}. To solve this condition in closed form we exploit central symmetry of the conditional distribution, which pairs $\bm{y}$ with $2\bar{\bm{y}}-\bm{y}$. The argument $\hat{\bm{y}}-\frac{1}{a}\bm{y}$ matches this pairing exactly when the coefficient of $\bm{y}$ is one, which is why the choice $a=1$ anticipated in Remark~\ref{rem:role_of_a} (Section~\ref{subsec:convex_surrogate}) yields a closed-form minimizer.
 
\begin{corollary}\label{cor:RSPO+a1} 
    Given $a=1$, suppose that $\mathbb{E}_{\bm{y}\sim\mathbb{P}_{\bm{y}|\bm{x}}}[\|\bm{y}\|_2]<\infty$ and that the distribution of $\bm{y}$ admits a density that is positive on an open set containing $2\bar{\bm{y}}+\gamma\mathcal{Z}$ and is centrally symmetric about its mean     $\bar{\bm{y}}:=\mathbb{E}_{\bm{y}\sim\mathbb{P}_{\bm{y}|\bm{x}}}[\bm{y}]$. Then $2\bar{\bm{y}}\in\mathcal{Y}^{\star}_{\mathcal{RSPO}_+}$. Furthermore, if the interior of $\mathcal{Z}$ is nonempty, then $\mathcal{Y}^{\star}_{\mathcal{RSPO}_+}=\{2\bar{\bm{y}}\}$.
\end{corollary}

Combining the two characterizations shows that Fisher consistency reduces to a single membership condition: does the surrogate optimizer $2\bar{\bm{y}}$ belong to the translated cone \eqref{eq:RSPO_normalcone}? By the scale-invariance discussion following Proposition~\ref{prop:minRSPO}, when $\gamma=0$, if the conditional mean $\bar{\bm{y}}$ belongs to the cone, then so does $2\bar{\bm{y}}$. Based on this fact, for $\gamma>0$, the cone is translated further along $-\bm{z}^{\star}(\bar{\bm{y}})$, while the surrogate optimizer remains $2\bar{\bm{y}}$, so the membership condition holds up to a threshold. We quantify this threshold in the next result. For a given $\bm x$, define
\begin{equation}\label{eq:gamma}
       \bar{\gamma}(\bm x):= \sup\Big\{\gamma\geq0:\
        2\bar{\bm{y}}+\gamma\bm{z}^{\star}(\bar{\bm{y}})\in
        -{\mathcal{N}}_{\mathcal{Z}}\big(\bm{z}^{\star}(\bar{\bm{y}})\big)\Big\}
        \in[0,+\infty].
\end{equation}
When $\mathcal{Z}$ is a bounded polyhedron, \eqref{eq:gamma} is the linear program
\begin{equation*}
    \sup\Big\{\gamma\geq0:\ 2\bar{\bm{y}}=\bm{A}^\top\bm{\lambda}-\gamma\bm{z}^{\star}(\bar{\bm{y}}),\ \bm{\lambda}\in\bm{\Lambda}\Big\}
\end{equation*}
over $(\gamma,\bm{\lambda})$, and is thus computable from the problem data.
 
\begin{theorem}[Fisher Consistency]\label{thm:Fisher_consistency}
     Set $a=1$ and suppose the assumptions of Proposition~\ref{prop:minRSPO} and of Corollary~\ref{cor:RSPO+a1} hold for $\mathbb{P}_{\bm x}$-almost every $\bm x$, and
     $\operatorname{int}\mathcal{Z}\neq\emptyset$. Let $\bar{\gamma}(\bm{x})$ be defined by~\eqref{eq:gamma}. Then:
     \begin{enumerate}
        \item[(i)] The set in \eqref{eq:gamma} is a closed interval containing $0$, so $\bar{\gamma}(\bm{x})$ is well-defined and is attained whenever it is finite. If $0<\gamma\leq\bar{\gamma}(\bm{x})$ for $\mathbb{P}_{\bm x}$-almost every $\bm x$, then for $\mathbb{P}_{\bm x}$-almost every $\bm x$ the unique minimizer $2\bar{\bm{y}}(\bm{x})$ of the pointwise $\mathcal{RSPO}_+$ risk is also a minimizer of the pointwise $\mathcal{RSPO}$ risk. That is, the $\mathcal{RSPO}_+$ loss is Fisher consistent with the $\mathcal{RSPO}$ loss.
        \item[(ii)] Suppose in addition that $\mathcal Z=\{\bm{z}\in\mathbb{R}^d:\bm{A}\bm{z}\geq\bm{b}\}$ is a nonempty bounded polyhedron. The threshold is strictly positive:
        \begin{equation*}
            \bar{\gamma}(\bm{x})
            \;\geq\;
            \frac{2\,\dist\!\big(\bar{\bm{y}}(\bm x),\,\partial \mathcal{C}(\bm{x})\big)}{\|\bm{z}^{\star}(\bar{\bm{y}}(\bm x))\|_2}
            \;>\;0,
            \qquad
            \mathcal{C}(\bm{x}):=-\mathcal{N}_{\mathcal{Z}}\big(\bm{z}^{\star}(\bar{\bm{y}}(\bm x))\big),
        \end{equation*}
        where $\partial \mathcal{C}(\bm{x})$ denotes the boundary of $\mathcal{C}(\bm{x})$ and, when $\bm{z}^{\star}(\bar{\bm{y}}(\bm x))=\bm{0}$, one has $\bar{\gamma}(\bm{x})=+\infty$ and the bound holds trivially. In particular, the regime in {\rm(i)} is nonempty for $\mathbb{P}_{\bm x}$-almost every $\bm x$.
     \end{enumerate}
\end{theorem}

Theorem~\ref{thm:Fisher_consistency} essentially provides sufficient conditions for Fisher consistency. This is not only important to the analysis of the population limit scenario, but also a fundamental requirement for finite-sample performance guarantees, as we elaborate on in Section~\ref{subsec:rspo+_excess_risk_bounds}. The mechanism behind part~(ii) is the same scale invariance identified after Proposition~\ref{prop:minRSPO}: uniqueness of the nominal solution forces $\bar{\bm{y}}(\bm x)$ into the \emph{interior} of the cone $\mathcal{C}(\bm{x})$; the interior of a cone is invariant under positive scaling, so $2\bar{\bm{y}}(\bm x)$ lies in the interior as well---at distance $2\dist(\bar{\bm{y}}(\bm{x}),\partial \mathcal{C}(\bm{x}))$ from the boundary---and the translation $-\gamma\bm{z}^{\star}(\bar{\bm{y}}(\bm x))$ must travel at least that far before the membership can fail.

The following example illustrates both sets of $\mathcal{Y}^{\star}_{\mathcal{RSPO}}$ and $\mathcal{Y}^{\star}_{\mathcal{RSPO}_+}$ in the theorem and shows that the restriction on $\gamma$ is essential: if $\gamma>\bar{\gamma}(\bm{x})$ for some $\bm x$, the unique surrogate minimizer $2\bar{\bm{y}}(\bm{x})$ lies outside $\mathcal{Y}^{\star}_{\mathcal{RSPO}}$ and induces a decision that is suboptimal under the $\mathcal{RSPO}$ criterion---a discrepancy that persists in the population limit.

\begin{example}\label{ex:fisher_consistency}
     Set $a=1$. Fix $\bm{x}$ and let $\bm{y}\mid\bm{x}\sim\mathcal{N}(\bar{\bm{y}},\sigma^2\bm{I}_2)$ with $\bar{\bm{y}}=(1,\frac{3}{2})^\top$ and any $\sigma>0$, so that the density is positive on all of $\mathbb{R}^2$ and centrally symmetric about $\bar{\bm{y}}$, and the assumptions of Corollary~\ref{cor:RSPO+a1} hold. Let the feasible set be
     \begin{equation*}
         \mathcal{Z}:=\Big\{\bm{z}=(z_1,z_2):z_1-z_2\leq1,\ z_1\geq -1,\ z_1,z_2\leq0 \Big\}.
     \end{equation*}
     The nominal solution is $\bm{z}^{\star}(\bar{\bm{y}})=(-1,-2)^\top$, at which the active constraints have normal vectors $(-1,1)^\top$ and $(1,0)^\top$. Therefore, the threshold of \eqref{eq:gamma}
     evaluates to $\bar{\gamma}=3/2$. 
     For a comparison with the nominal situation, we also compute the sets $\mathcal{Y}^{\star}_{\mathcal{SPO}}$ and $\mathcal{Y}^{\star}_{\mathcal{SPO}_+}$, where $\mathcal{Y}^{\star}_{\mathcal{SPO}}=\argmin_{\hat{\bm{y}}\in\mathbb{R}^d} \mathbb{E}_{\bm{y}\sim\mathbb{P}_{\bm{y}
    |\bm{x}}}\left[\ell_{\mathcal{SPO}}(\hat{\bm{y}},\bm{y})\right]$ and
    $\mathcal{Y}^{\star}_{\mathcal{SPO}_+}=\argmin_{\hat{\bm{y}}\in\mathbb{R}^d} \mathbb{E}_{\bm{y}\sim\mathbb{P}_{\bm{y}
    |\bm{x}}}\left[\ell_{\mathcal{SPO}_+}(\hat{\bm{y}},\bm{y})\right]$.
     Figure~\ref{fig:example_for_RSPO+} displays the two sets, together with $\mathcal{Y}^{\star}_{\mathcal{RSPO}}$ and $\mathcal{Y}^{\star}_{\mathcal{RSPO}_+}=\{2\bar{\bm{y}}\}$. For $\gamma=1\leq\bar{\gamma}$ the surrogate minimizer $2\bar{\bm{y}}$ lies in $\mathcal{Y}^{\star}_{\mathcal{RSPO}}$ and Fisher consistency holds; for $\gamma=2>\bar{\gamma}$ it lies outside and the property fails, even though the distribution of $\bm{y}$ is continuous and symmetric. In both panels $\mathcal{Y}^{\star}_{\mathcal{SPO}_+}\subseteq\mathcal{Y}^{\star}_{\mathcal{SPO}}$, so the nominal surrogate remains Fisher consistent regardless of $\gamma$.\hfill\Halmos
\end{example}
 
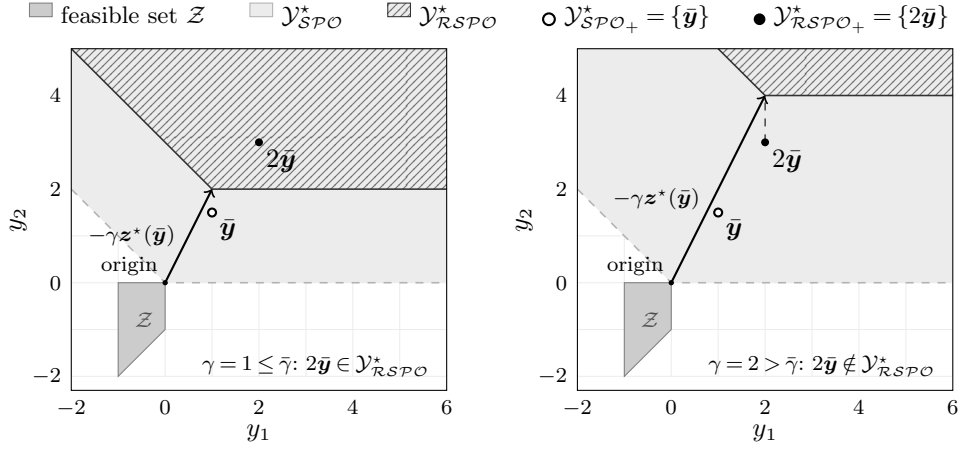
\begin{figure}[!t]
\centering
{\footnotesize
\tikz[baseline=-0.6ex]\fill[gray!40,draw=black!50] (0,0) rectangle (0.30,0.18);~feasible set $\mathcal{Z}$
\hspace{1.1em}
\tikz[baseline=-0.6ex]\fill[gray!15,draw=gray!60] (0,0) rectangle (0.30,0.18);~$\mathcal{Y}^{\star}_{\mathcal{SPO}}$
\hspace{1.1em}
\tikz[baseline=-0.6ex]{\fill[gray!15] (0,0) rectangle (0.30,0.18);
  \draw[pattern=north east lines,pattern color=black!55,draw=black!80] (0,0) rectangle (0.30,0.18);}~$\mathcal{Y}^{\star}_{\mathcal{RSPO}}$
\hspace{1.1em}
\tikz[baseline=-0.35ex]\draw[black,thick,fill=white] (0,0) circle (2.2pt);~$\mathcal{Y}^{\star}_{\mathcal{SPO}_+}=\{\bar{\bm{y}}\}$
\hspace{1.1em}
\tikz[baseline=-0.35ex]\fill[black] (0,0) circle (2.2pt);~$\mathcal{Y}^{\star}_{\mathcal{RSPO}_+}=\{2\bar{\bm{y}}\}$
}\\[1.2ex]
\begin{tikzpicture}[scale=0.62]
  \begin{scope}
    \clip (-2,-2.3) rectangle (6,5);
    \draw[gray!15,very thin] (-2,-2) grid (6,5);
    \fill[gray!15] (0,0) -- (6,0) -- (6,5) -- (-2,5) -- (-2,2) -- cycle;
    \draw[gray!60,semithick,dashed] (-2,2) -- (0,0) -- (6,0);
    \fill[pattern=north east lines,pattern color=black!55]
      (1,2) -- (6,2) -- (6,5) -- (-2,5) -- cycle;
    \draw[black!80,semithick] (-2,5) -- (1,2) -- (6,2);
    \fill[gray!40,draw=black!50] (-1,0) -- (0,0) -- (0,-1) -- (-1,-2) -- cycle;
    \node[black!70,font=\scriptsize] at (-0.45,-0.8) {$\mathcal{Z}$};
    \draw[->,black,thick] (0,0) -- (1,2)
      node[midway,left=1pt,font=\scriptsize] {$-\gamma\bm{z}^{\star}(\bar{\bm{y}})$};
    \fill (0,0) circle (1.6pt) node[above left=-1pt,font=\scriptsize] {origin};
    \draw[black,thick,fill=white] (1,1.5) circle (2.4pt);
    \node[below right=-1pt,font=\small] at (1,1.5) {$\bar{\bm{y}}$};
    \fill[black] (2,3) circle (2.4pt);
    \node[below right=2pt,font=\small,inner sep=0.5pt]
      at (2,3) {$2\bar{\bm{y}}$};
    \node[anchor=south east,font=\scriptsize,align=right] at (5.85,-2.2)
      {$\gamma=1\leq\bar{\gamma}$:\ $2\bar{\bm{y}}\in\mathcal{Y}^{\star}_{\mathcal{RSPO}}$};
  \end{scope}
  \draw (-2,-2.3) rectangle (6,5);
  \foreach \x in {-2,0,2,4,6}
    \node[below,font=\scriptsize] at (\x,-2.3) {$\x$};
  \foreach \y in {-2,0,2,4}
    \node[left,font=\scriptsize] at (-2,\y) {$\y$};
  \node[below=9pt,font=\small] at (2,-2.3) {$y_1$};
  \node[left=13pt,font=\small,rotate=90,anchor=south] at (-2,1.35) {$y_2$};
\end{tikzpicture}
\hspace{2ex}
\begin{tikzpicture}[scale=0.62]
  \begin{scope}
    \clip (-2,-2.3) rectangle (6,5);
    \draw[gray!15,very thin] (-2,-2) grid (6,5);
    \fill[gray!15] (0,0) -- (6,0) -- (6,5) -- (-2,5) -- (-2,2) -- cycle;
    \draw[gray!60,semithick,dashed] (-2,2) -- (0,0) -- (6,0);
    \fill[pattern=north east lines,pattern color=black!55]
      (2,4) -- (6,4) -- (6,5) -- (1,5) -- cycle;
    \draw[black!80,semithick] (1,5) -- (2,4) -- (6,4);
    \fill[gray!40,draw=black!50] (-1,0) -- (0,0) -- (0,-1) -- (-1,-2) -- cycle;
    \node[black!70,font=\scriptsize] at (-0.45,-0.8) {$\mathcal{Z}$};
    \draw[->,black,thick] (0,0) -- (2,4)
      node[pos=0.45,left=1pt,font=\scriptsize] {$-\gamma\bm{z}^{\star}(\bar{\bm{y}})$};
    \draw[black,dashed] (2,3) -- (2,4);
    \fill (0,0) circle (1.6pt) node[above left=-1pt,font=\scriptsize] {origin};
    \draw[black,thick,fill=white] (1,1.5) circle (2.4pt);
    \node[below right=-1pt,font=\small] at (1,1.5) {$\bar{\bm{y}}$};
    \fill[black] (2,3) circle (2.4pt);
    \node[below right=2pt,font=\small,inner sep=0.5pt]
      at (2,3) {$2\bar{\bm{y}}$};
    \node[anchor=south east,font=\scriptsize,align=right] at (5.85,-2.2)
      {$\gamma=2>\bar{\gamma}$:\ $2\bar{\bm{y}}\notin\mathcal{Y}^{\star}_{\mathcal{RSPO}}$};
  \end{scope}
  \draw (-2,-2.3) rectangle (6,5);
  \foreach \x in {-2,0,2,4,6}
    \node[below,font=\scriptsize] at (\x,-2.3) {$\x$};
  \foreach \y in {-2,0,2,4}
    \node[left,font=\scriptsize] at (-2,\y) {$\y$};
  \node[below=9pt,font=\small] at (2,-2.3) {$y_1$};
  \node[left=13pt,font=\small,rotate=90,anchor=south] at (-2,1.35) {$y_2$};
\end{tikzpicture}
\caption{Illustration of Example~\ref{ex:fisher_consistency}.
$\mathcal{Y}^{\star}_{\mathcal{RSPO}}$ (hatched) is the negative normal cone $-{\mathcal{N}}_{\mathcal{Z}}(\bm{z}^{\star}(\bar{\bm{y}}))$ translated by $-\gamma\bm{z}^{\star}(\bar{\bm{y}})$ (solid arrow), so it moves with $\gamma$ while the surrogate minimizer stays at $2\bar{\bm{y}}$. Dashed versus solid boundaries mark the open set $\mathcal{Y}^{\star}_{\mathcal{SPO}}$ against the closed set $\mathcal{Y}^{\star}_{\mathcal{RSPO}}$. \emph{Left}: $\gamma=1\leq\bar{\gamma}$, Fisher consistency holds; \emph{right}: $\gamma=2>\bar{\gamma}$, Fisher consistency fails.}
\label{fig:example_for_RSPO+}
\end{figure}
 
We close the subsection with a byproduct of the analysis by setting $\gamma=0$. For a bounded polyhedron, uniqueness of the nominal solution forces $\bar{\bm{y}}$ into the interior of the cone $\mathcal{C}(\bm{x})$, hence $2\bar{\bm{y}}$ as well by scale-invariance. Therefore, Fisher consistency holds for the nominal surrogate with $a=1$, as presented in Corollary~\ref{cor:spo_plus_a1}. Notice that this result \textit{complements} the Fisher consistency established for the $\mathcal{SPO}_+$ loss by \citet{elmachtoub2022smart}, which concerns $a=2$ over general compact convex feasible sets, whereas Corollary~\ref{cor:spo_plus_a1} shows that, under analogous distributional conditions, the $a=1$ variant holds over bounded polyhedra.

\begin{corollary}[Fisher Consistency of $\mathcal{SPO}_+$ with $a=1$ and Polyhedrality]\label{cor:spo_plus_a1}
    Suppose that $\mathcal{Z}$ is a bounded polyhedron, and that for $\mathbb{P}_{\bm x}$-almost every $\bm x$, the conditional distribution of $\bm{y}$ is centrally symmetric about $\bar{\bm{y}}$ with $\mathbb{E}_{\bm{y}\sim\mathbb{P}_{\bm{y}|\bm{x}}}[\|\bm{y}\|_2]<\infty$ and has a density positive on an open set containing $2\bar{\bm{y}}$. Further assume that the interior of $\mathcal Z$ is nonempty and the nominal problem $\min_{\bm{z}\in\mathcal{Z}}\bar{\bm{y}}^\top\bm{z}$ admits a unique optimal solution. Then the $\mathcal{SPO}_+$ loss with $a=1$ is Fisher consistent with the $\mathcal{SPO}$ loss.
\end{corollary}

\subsection{Gradient Descent Method}
\label{subsec:gradient_refinement}

Sections~\ref{subsec:convex_surrogate} and~\ref{subsec:fisher_consistency} have developed the convex surrogate scheme. We now turn to gradient descent methods, for which we begin by deriving the Jacobian of the decision map, based on which we sketch the algorithm.

\subsubsection*{The Jacobian of the decision map.}
Recall Lemma~\ref{lem:oracle_projection}, which shows the decision map $\bm z^{\star}_{\gamma}(\bm y_0)=\Pi_{\mathcal Z}(-\bm y_0/\gamma)$ is Lipschitz continuous with constant $1/\gamma$ and differentiable almost everywhere. To obtain an implementable algorithm, we need the derivative of the decision map. For a polyhedral feasible set it is available in closed form at points satisfying two standard regularity conditions: the linear independence constraint qualification (LICQ) and strict complementarity.
 
\begin{proposition} \label{prop:zstar_gradient_polyhedral}
    Let $\gamma>0$ and let $\mathcal Z=\{\bm{z}:\bm{A}\bm{z}\geq\bm{b}\}$ be a nonempty bounded polyhedron.
    For a given $\bm{y}_0$, let $S_{\bm{y}_0}=\{i:\bm{a}_i^\top\bm{z}^{\star}_{\gamma}(\bm{y}_0)=b_i\}$ denote the indices of active constraints at $\bm{z}^{\star}_{\gamma}(\bm{y}_0)$ and $\bm{A}_{S_{\bm{y}_0}}$ denote the submatrix of $\bm{A}$ containing only the rows $\bm{a}_i$ with indices $i\in S_{\bm{y}_0}$. 
    Assume that (i) $\bm{A}_{S_{\bm{y}_0}}$ has full row rank (LICQ), and (ii) strict complementarity holds at $\bm{z}^{\star}_{\gamma}(\bm{y}_0)$, \textit{i.e.}, the unique Karush--Kuhn--Tucker (KKT) multiplier $\bm{\lambda}^0$ associated with $\bm{z}^{\star}_{\gamma}(\bm{y}_0)$ satisfies $\lambda^0_i>0$ for all $i\in S_{\bm{y}_0}$. Then there exists a neighborhood of $\bm{y}_0$ on which the active set of $\bm{z}^{\star}_{\gamma}(\cdot)$ is identically $S_{\bm{y}_0}$ and $\bm{z}^{\star}_{\gamma}(\cdot)$ is affine. In particular, $\bm{z}^{\star}_{\gamma}(\cdot)$ is differentiable at $\bm{y}_0$, and its Jacobian is
    \begin{equation}\label{eq:zstar_gradient}
        \nabla\bm{z}^{\star}_{\gamma}(\bm{y}_0)=\frac{1}{\gamma}\big(\bm{A}_{S_{\bm{y}_0}}^
        {\top}(\bm{A}_{S_{\bm{y}_0}}\bm{A}_{S_{\bm{y}_0}}^\top)^{-1}\bm{A}_{S_{\bm{y}_0}}
        -\bm{I}\big).
    \end{equation}
\end{proposition}
It is noted that the Jacobian~\eqref{eq:zstar_gradient} equals $-\frac{1}{\gamma}(\bm I-\bm P_{S_{\bm y_0}})$, where $\bm P_{S_{\bm y_0}}:=\bm A_{S_{\bm y_0}}^\top(\bm A_{S_{\bm y_0}}\bm A_{S_{\bm y_0}}^\top)^{-1}\bm A_{S_{\bm y_0}}$ is the orthogonal projection onto the row space of $\bm A_{S_{\bm y_0}}$. It is thus a scaled projection onto the tangent space of the active face, with operator norm at most $1/\gamma$, consistent with Lemma~\ref{lem:oracle_projection}. Three implementation points follow. First, since the expression depends on $\bm A_{S_{\bm y_0}}$ only through its row space, the inverse may be replaced by the Moore--Penrose pseudoinverse when $\bm A_{S_{\bm y_0}}$ is rank deficient; at such points, and at points where strict complementarity fails, \eqref{eq:zstar_gradient} should be read as the derivative of the affine piece selected by the current active set rather than as a Jacobian of $\bm z^{\star}_{\gamma}$. Second, in practice Jacobian--vector products are computed by implicitly differentiating the KKT system of the robust decision problem \ref{ilroopt}, which avoids forming \eqref{eq:zstar_gradient} explicitly. Third, this also accommodates feasible sets described through auxiliary variables: if $\mathcal Z$ is represented by a polyhedral lift in $(\bm z,\bm w)$, the formula applies in the lifted space and the required derivative is read off the $\bm z$-block. The portfolio experiment of Section~\ref{subsec:portfolio_rspo+} uses exactly this device for the $\ell_1$ risk constraint; see Appendices~\ref{app:portfolio_reformulation}
and~\ref{app:gradient_implementation} for more details.

\begin{remark}[Relation to differentiable optimization layers] \label{rem:vs_optnet}
Differentiating through a strongly convex quadratic program via its optimality conditions is standard \citep{amos2017optnet,agrawal2019differentiable}, and quadratic smoothing of a linear program for this purpose is introduced by \cite{wilder2019melding}. We highlight that the difference lies in the modeling aspect: in the existing literature the quadratic term is a training-time smoothing device that is discarded at deployment, where the nominal decision map is restored, so the derivative is taken through a decision rule that is never implemented. In the ILRO framework the same $\gamma$ governs training and deployment, so \eqref{eq:zstar_gradient} is the derivative of the decision rule actually used, and
descending the resulting objective decreases the regret that is actually
incurred.\hfill\Halmos
\end{remark}

\subsubsection*{Algorithm.}
Based on the explicit Jacobian of the robust decision map, we minimize the regularized empirical $\mathcal{RSPO}$ objective
\begin{equation}\label{eq:rspo_refine_obj}
    \hat{L}_N(\theta):=\frac{1}{N}\sum_{i\in[N]}\big[\bm y_i^\top \bm z_\gamma^\star(\bm g_\theta(\bm x_i))-v^\star(\bm y_i)\big]+\lambda\,\Omega(\bm{g}_{\theta})
\end{equation}
over a parametric class $\{\bm g_\theta:\theta\in\Theta\}$ with $\bm g_\theta$ differentiable in $\theta$, where $\Omega(\cdot)$ is the prediction-model regularizer and $\lambda\ge0$ its weight. Since $v^\star(\bm y_i)$ does not depend on $\theta$, the chain rule gives the per-sample gradient $[\nabla_\theta \bm g_\theta(\bm x_i)]^\top \bm q_i$ with $\bm q_i=\nabla \bm z_\gamma^\star(\bm g_\theta(\bm x_i))^\top \bm y_i$. We initialize $\theta^{(0)}$ at the $\mathcal{RSPO}_+$ minimizer, and apply standard stochastic gradient descent coupled with an Armijo test \citep{nocedal2006numerical} to decrease the objective sufficiently. The details of the algorithm are deferred to Appendix~\ref{app:gradient_algorithm}.

In summary, this section provides two complementary routes to the empirical $\mathcal{RSPO}$ learning problem~\eqref{eq:rspo_problem}: the convex $\mathcal{RSPO}_+$ surrogate, which is globally solvable for linear predictors over polyhedral feasible sets and targets the correct population predictor when the robustness parameter $\gamma$ is below the specific threshold; and the gradient method, which descends the $\mathcal{RSPO}$ objective directly for any $\gamma>0$, and can be used to refine the surrogate solution locally. The next section analyzes the statistical guarantees of the two predictors defined by the target and surrogate losses.

\section{Theoretical Guarantees}
\label{sec:statistical_guarantee}
This section is devoted to the finite-sample performance guarantees of the two predictors, namely, the $\mathcal{RSPO}$ predictor and the $\mathcal{RSPO}_+$ predictor. In both cases, we are interested in the population $\mathcal{RSPO}$ risk, since it is the target loss that measures the realized cost of the deployed robust decision.

Let $\mathbb{P}$ denote the population distribution of $(\bm x,\bm y)$ and fix $\gamma>0$. We define the following quantities for a predictor $\bm g$. For a concise exposition, we suppress two types of dependence throughout this section. First, all quantities implicitly depend on $\gamma$ through the losses $\ell_{\mathcal{RSPO}}$ and $\ell_{\mathcal{RSPO}_+}$. Second, sample-based quantities, including empirical risks, estimators, and Rademacher complexities, depend on the sample size $N$. We make these dependencies explicit only when they are relevant to the interpretation. 
The population risks under the target and surrogate losses are 
\begin{equation*}
    R(\bm g):=\mathbb{E}_{(\bm x,\bm y)\sim \mathbb P}\!\left[\ell_{\mathcal{RSPO}}(\bm g(\bm x),\bm y)\right],
    \qquad
    R_+(\bm g):=\mathbb{E}_{(\bm x,\bm y)\sim \mathbb P}\!\left[\ell_{\mathcal{RSPO}_+}(\bm g(\bm x),\bm y)\right],
\end{equation*}
respectively. Given i.i.d. samples $\{(\bm x_i,\bm y_i)\}_{i\in[N]}$, their empirical counterparts are
\begin{equation*}
    \widehat R(\bm g):=\frac1N\sum_{i\in[N]}\ell_{\mathcal{RSPO}}(\bm g(\bm x_i),\bm y_i),
    \qquad
    \widehat R_+(\bm g):=\frac1N\sum_{i\in[N]}\ell_{\mathcal{RSPO}_+}(\bm g(\bm x_i),\bm y_i),
\end{equation*}
respectively. For a class $\mathcal G$ of measurable predictors from $\mathcal X$ to $\mathbb R^d$, we set 
\begin{equation*}
    R^{\star}(\mathcal G):=\inf_{\bm g\in\mathcal G}R(\bm g),
    \qquad R^{\star}_+(\mathcal G):=\inf_{\bm g\in\mathcal G}R_+(\bm g).
\end{equation*}
Recall $\mathcal G_{\rm all}$ denotes the class of all measurable functions from $\mathcal X$ to $\mathbb R^d$, and define the lowest achievable target and surrogate risks by $R^{\star}:=R^{\star}(\mathcal G_{\rm all})$ and $R^{\star}_+:=R^{\star}_+(\mathcal G_{\rm all})$, respectively. Throughout the section we analyze the empirical $\mathcal{RSPO}$ and $\mathcal{RSPO}_+$ risk minimizers without regularization, corresponding to $\lambda=0$ in the formulations \eqref{eq:rspo_problem} and \eqref{eq:rspo+_ermprob}:
\begin{equation*}
    \widehat{\bm g}_{\rm RSPO}\in\argmin_{\bm g\in\mathcal G}\widehat R(\bm g),
    \qquad
    \widehat{\bm g}_{{\rm RSPO}_+}\in\argmin_{\bm g\in\mathcal G}\widehat R_+(\bm g),
\end{equation*}
which we assume to exist.

The two predictors are measured against different benchmarks, and the distinction drives the analysis that follows. The $\mathcal{RSPO}$ predictor minimizes an empirical version of the very risk we wish to control, so the natural benchmark is the best predictor available in the hypothesis class, and the object of interest is the within-class excess risk $R(\widehat{\bm g}_{\rm RSPO})-R^{\star}(\mathcal G)$, which is studied in Section~\ref{subsec:rspo_excess_risk_bounds}. The $\mathcal{RSPO}_+$ predictor, by contrast, minimizes a different risk, so to obtain a target-risk guarantee, we naturally exploit Fisher consistency. The latter is a statement about pointwise minimizers over all measurable predictors as opposed to a restricted function class $\mathcal G$. Its benchmark is therefore $R^{\star}$, and the object of interest is the target excess risk $R(\widehat{\bm g}_{{\rm RSPO}_+})-R^{\star}$, which is studied in Section~\ref{subsec:rspo+_excess_risk_bounds}.

\subsection{Excess Risk Bounds of $\mathcal{RSPO}$ Predictor}
\label{subsec:rspo_excess_risk_bounds}

The excess risk $R(\widehat{\bm g}_{\rm RSPO})-R^{\star}(\mathcal G)$ is controlled by a uniform deviation bound over the hypothesis class, and the analysis therefore reduces to measuring its complexity. We carry this out in two steps. We first bound the deviation by the Rademacher complexity of the function class induced by the loss, and then transfer that bound to the prediction class itself, using the Lipschitz continuity of the decision map.

Given a vector-valued hypothesis class $\mathcal{G}$, its $d$-dimensional multivariate Rademacher complexity, and its empirical counterpart conditional on a sample $\bm x^N=(\bm x_1,\ldots,\bm x_N)$, are
\begin{equation*}
    \mathfrak{R}_d(\mathcal{G}):= \mathbb{E}_{\bm{\sigma},\bm{X}}\left[\sup_{\bm{g}\in\mathcal{G}}\frac{1}{N}\sum_{i\in[N]}\bm{\sigma}_i^\top \bm{g}(\bm{x}_i)\right],
    \qquad
    \widehat{\mathfrak R}_d(\mathcal G;\bm x^N):= \mathbb E_{\bm\sigma}\left[\sup_{\bm g\in\mathcal G}\frac{1}{N}\sum_{i\in[N]}\bm\sigma_i^\top \bm g(\bm x_i)\right],
\end{equation*}
where $\bm{\sigma}_i\in\{+1,-1\}^d$ are independent Rademacher random vectors~\citep{maurer2016vector,bartlett2002rademacher} with $\bm{\sigma}=(\bm{\sigma}_i)_{i\in [N]}$, and $\bm{X}$ is the underlying population for the sample $\bm x^N$, so that $\mathfrak R_d(\mathcal G)=\mathbb E_{\bm X}[\widehat{\mathfrak R}_d(\mathcal G;\bm x^N)]$. We next define the loss-induced function class.

\begin{definition}[Loss-induced function class]
\label{def:loss_induced_class}
    Let $\mathcal G\subseteq \{\bm g:\mathcal X\to \mathbb R^d\}$ and let $\ell:\mathbb R^d\times\mathbb R^d\to \mathbb R$ be a loss function. The loss-induced function class associated with $\ell$ and $\mathcal G$ is
    \begin{equation*}
        \mathcal H_{\ell}
        =
        \left\{
        h_{\bm g}:(\bm x,\bm y)\mapsto \ell(\bm g(\bm x),\bm y)\mid
        \bm g\in\mathcal G
        \right\},
    \end{equation*}
    and its Rademacher complexity is
    \begin{equation*}
        \mathfrak{R}_1(\mathcal{H}_{\ell}):=
        \mathbb{E}_{(\sigma_i)_{i\in[N]},(\bm{X},\bm{Y})}\left[\sup_{\bm{g}\in\mathcal{G}}\frac{1}{N}\sum_{i\in[N]}\sigma_i\,\ell(\bm{g}(\bm{x}_i),\bm{y}_i)\right],
    \end{equation*}
    where $\sigma_i$ are independent Rademacher random variables and $\bm{Y}$ is the underlying population for the sample $(\bm y_1,\ldots,\bm y_N)$.
\end{definition}

To ease the analysis we impose the bounded objective-gap condition
\begin{equation*}
    b=\sup_{\bm{y}\in\mathcal{Y}}\left\{\max_{\bm{z}\in\mathcal{Z}}\bm{y}^\top\bm{z}-\min_{\bm{z}\in\mathcal{Z}}\bm{y}^\top\bm{z}\right\}<\infty,
\end{equation*}
which measures the largest difference in realized objective value between two feasible decisions under the same cost vector. Since $\bm z^{\star}_{\gamma}(\hat{\bm y})\in\mathcal Z$, it implies the uniform bound $0\leq \ell_{\mathcal{RSPO}}(\hat{\bm y},\bm y)\leq b$ for every $\hat{\bm y}\in\mathbb R^d$ and $\bm y\in\mathcal Y$. The first step is then a standard consequence of the symmetrization argument of \citet{bartlett2002rademacher}.

\begin{proposition}\label{prop:rspo_generalization}
    For any $\delta\in(0,1)$, with probability at least $1-\delta$,
    \begin{equation*}
    \sup_{\bm g\in \mathcal G}\left|R(\bm{g})-\widehat{R}(\bm{g})\right|\leq 2\mathfrak{R}_1(\mathcal{H}_{{\ell}_{\mathcal{RSPO}}})+b\sqrt{\frac{\log(2/\delta)}{2N}}.
    \end{equation*}
\end{proposition}

The second step relates the complexity of $\mathcal H_{\ell_{\mathcal{RSPO}}}$ to that of $\mathcal G$. This is possible because the robust decision map is a projection, so the loss is Lipschitz in the prediction. In particular, for every fixed $\bm y$, $\ell_{\mathcal{RSPO}}(\cdot,\bm y)$ is $\left(\|\bm y\|_2/\gamma\right)$-Lipschitz, because $\ell_{\mathcal{RSPO}}(\hat{\bm{y}},\bm{y})=\bm{y}^\top\bm{z}^{\star}_{\gamma}(\hat{\bm{y}})-v^{\star}(\bm{y})$ and $\bm{z}^{\star}_{\gamma}(\cdot)$ is $1/\gamma$-Lipschitz by Lemma~\ref{lem:oracle_projection}. This yields the following result by applying the vector contraction inequality of \cite{maurer2016vector}.

\begin{lemma}\label{lem:vector_contraction_inequality}
    Suppose that $\|\bm y\|_2\leq r$ for all $\bm y\in\mathcal Y$. Then $\mathfrak R_1(\mathcal H_{\ell_{\mathcal{RSPO}}})\leq \frac{\sqrt{2}r}{\gamma}\,\mathfrak R_d(\mathcal G)$.
\end{lemma}

Combining the two steps gives the central result of this subsection.

\begin{theorem}[Meta Generalization Bound]
\label{thm:rspo_meta_generalization}
    For $\gamma>0$, suppose that $\|\bm y\|_2\leq r$ for all $\bm y\in\mathcal Y$. Then, for any $\delta\in(0,1)$, with probability at least $1-\delta$,
    \begin{equation*}
        R(\widehat{\bm g}_{\rm RSPO})
        -
        R^{\star}(\mathcal G)
        \leq
        \frac{4\sqrt{2}\,r}{\gamma}\,
        \mathfrak R_d(\mathcal G)
        +
        b\sqrt{\frac{2\log(2/\delta)}{N}}.
    \end{equation*}
\end{theorem}

\begin{table}[t]
    \centering
    \footnotesize
    \renewcommand{\arraystretch}{1.6}
    \setlength{\tabcolsep}{6pt}
    \caption{Instantiations of the meta generalization bound. The third column reports the complexity bound available for each class and the last column the leading term of the resulting bound on the excess risk $R(\widehat{\bm g}_{\rm RSPO})-R^{\star}(\mathcal G)$.}
    \label{tab:rspo_generalization_methods}
    \begin{tabularx}{\textwidth}{
          >{\raggedright\arraybackslash}p{0.19\textwidth}
          >{\raggedright\arraybackslash}X
          >{\centering\arraybackslash}p{0.20\textwidth}
          >{\centering\arraybackslash}p{0.24\textwidth}}
        \toprule
        \textbf{Hypothesis class} &
        \textbf{Condition imposed} &
        \textbf{\makecell[c]{Complexity bound\\ $\mathfrak R_d(\mathcal G)\leq$}} &
        \textbf{\makecell[c]{Excess risk bound\\ (leading term)}} \\
        \midrule
        Bounded affine\newline (Example~\ref{ex:rspo_gen_affine}) &
        $\bm g_{\bm W}(\bm x)=\bm W\bm\varphi(\bm x)$ with $\|\bm\varphi(\bm x)\|_2\leq \kappa$ and $\|\bm W\|_F\leq B$ &
        $\displaystyle B\kappa\sqrt{\frac{d}{N}}$ &
        $\displaystyle \frac{4\sqrt2\,rB\kappa}{\gamma}\sqrt{\frac{d}{N}}$ \\
        Polynomial discrimination (Example~\ref{ex:rspo_gen_pd}) &
        Order $\nu$ in the sense of Definition~\ref{def:vector_polynomial_discrimination}, with $D_{\mathcal G}(\bm x^N)\leq B_1$ &
        $\displaystyle B_1\sqrt{\frac{2\nu\log(N+1)}{N}}$ &
        $\displaystyle \frac{8rB_2}{\gamma}\sqrt{\frac{\nu\log (N+1)}{N}}$ \\
        Dudley entropy integral (Example~\ref{ex:rspo_gen_dudley}) &
        Entropy integral bounded by $J_{\mathcal G}$ uniformly over samples &
        $\displaystyle C\,\frac{J_{\mathcal G}}{\sqrt N}$ &
        $\displaystyle \frac{4\sqrt2\,Cr}{\gamma}\,\frac{J_{\mathcal G}}{\sqrt N}$ \\
        RKHS ball\newline (Example~\ref{ex:rspo_gen_rkhs}) &
        $\sum_{j}\|g_j\|_{\mathcal H}^2\leq R_{\mathcal H}^2$ with $\sup_{\bm x}K(\bm x,\bm x)\leq\kappa^2$ &
        $\displaystyle \kappa R_{\mathcal H}\sqrt{\frac{d}{N}}$ &
        $\displaystyle \frac{4\sqrt2\,r\kappa R_{\mathcal H}}{\gamma}\sqrt{\frac{d}{N}}$ \\
        \bottomrule
        \addlinespace[3pt]
        \multicolumn{4}{@{}p{\dimexpr\textwidth-2\tabcolsep\relax}@{}}{%
        \emph{Note.} All bounds are stated under the assumptions $\gamma>0$ and $\|\bm y\|_2\leq r$ for every $\bm y\in\mathcal Y$. Each excess risk bound holds with probability at least $1-\delta$ for any $\delta\in(0,1)$ and reports only the leading term; the additive term $b\sqrt{2\log(2/\delta)/N}$, common to all four classes, is omitted. Details and proofs are given in Appendix~\ref{app:gen_examples}.}\\
    \end{tabularx}
\end{table}

Theorem~\ref{thm:rspo_meta_generalization} separates the two sources of statistical difficulty. The robust decision map enters only through the Lipschitz factor $r/\gamma$, and the richness of the prediction class only through $\mathfrak R_d(\mathcal G)$. Deriving a guarantee for a specific hypothesis class therefore reduces to the single task of bounding $\mathfrak R_d(\mathcal G)$. Appendix~\ref{app:gen_examples} carries this out for four representative classes---bounded affine predictors, polynomial-discrimination classes, classes with bounded Dudley entropy integral, and vector-valued RKHS balls. Table~\ref{tab:rspo_generalization_methods} collects the resulting bounds: across the four classes the excess risk decays at the parametric rate $N^{-1/2}$, up to a logarithmic factor in the polynomial-discrimination case, with the robustness parameter entering uniformly through $1/\gamma$.

\begin{remark}[Comparison with $\mathcal{SPO}$ risk bounds]
\label{rem:comparison_spo_generalization}
The argument leading to Theorem~\ref{thm:rspo_meta_generalization} is viable because the robust decision map is smooth. Vector contraction requires the loss to be Lipschitz in the prediction, which is satisfied by $\ell_{\mathcal{RSPO}}(\cdot,\bm y)$. This is not the case for $\ell_{\mathcal{SPO}}(\cdot,\bm y)$, as it is discontinuous. Generalization bounds for $\mathcal{SPO}$ are therefore obtained by controlling the induced loss class directly, through an elegant analysis of margin conditions and combinatorial complexity measures tailored to the feasible region \citep{elbalghiti2023generalization}. However, this comes at the cost that the resulting bounds depend on the complexity and dimensionality of the feasible set $\mathcal Z$.  \hfill\Halmos
\end{remark}

The fact that our bounds depend on the order of $1/\gamma$ highlights the relationship between the amount of data and the robustness requirement: when the formulation is more robust (larger $\gamma$), fewer samples are needed to achieve the same excess risk bound. This suggests that, when data are limited, the decision-maker should seek more robust decisions.

\subsection{Excess Risk Bounds of $\mathcal{RSPO}_+$ Predictor}
\label{subsec:rspo+_excess_risk_bounds}

Now we turn to the excess risk $R(\widehat{\bm g}_{{\rm RSPO}_+})-R^{\star}$ of the surrogate predictor. A natural first attempt is to compare $\widehat{\bm g}_{{\rm RSPO}_+}$ with $\widehat{\bm g}_{\rm RSPO}$ and invoke the bounds of Section~\ref{subsec:rspo_excess_risk_bounds}. However, $\widehat{\bm g}_{\rm RSPO}$ is the minimizer of a nonconvex optimization problem (typically NP-hard), whereas $\widehat{\bm g}_{{\rm RSPO}_+}$ is obtained by minimizing a convex surrogate problem (typically tractable), under finite samples and a restricted function class. Therefore, the computational complexity barrier prevents the two predictors from being close in general. In contrast, Fisher consistency results dictate that, in the population limit and when the function class consists of all measurable functions, the surrogate minimizer is also a target-risk minimizer. Thus, we derive the excess risk bounds through Fisher consistency.

To ease our discussion, we introduce the population ``oracle predictor''
\begin{equation}\label{eq:appE_reference}
    \bm g_0(\bm x):=2\bar{\bm y}(\bm x)=2\mathbb{E}_{\bm{y}\sim\mathbb{P}_{\bm{y}|\bm{x}}}[\bm{y}],
\end{equation}
which by Corollary~\ref{cor:RSPO+a1} is the unique pointwise minimizer of the surrogate risk when $a=1$ under the assumptions of this corollary.

The analysis can be roughly divided into three steps. First, using a standard generalization argument, we bound the excess risk of the surrogate $R_+(\widehat{\bm{g}}_{{\rm RSPO}_+}) - R^{\star}_+$. Next, by the local strong convexity, this implies $\widehat{\bm g}_{{\rm RSPO}_+}$ is close to $\bm g_0$. Finally, applying the Lipschitz property of $R(\cdot)$ and the fact that $R(\bm g_0)=R^{\star}$ due to Fisher consistency, we obtain the bound on target excess risk $R(\widehat{\bm g}_{{\rm RSPO}_+})-R^{\star}$. We begin by providing necessary definitions and conditions.

\subsubsection*{Pointwise risks and the conditions.} Throughout, $\mathcal B(\bm c,r):=\{\bm u\in\mathbb R^d:\|\bm u-\bm c\|_2\le r\}$ denotes the closed Euclidean ball. For a fixed context $\bm x$, define the pointwise target and surrogate risks as
\begin{equation}\label{eq:appE_pointwise_target}
    R(\bm c;\bm x):=\mathbb{E}_{\bm{y}\sim\mathbb{P}_{\bm{y}|\bm{x}}}[\ell_{\mathcal{RSPO}}(\bm c,\bm y)],
    \qquad
    R_+(\bm c;\bm x):=\mathbb{E}_{\bm{y}\sim\mathbb{P}_{\bm{y}|\bm{x}}}[\ell_{\mathcal{RSPO}_+}(\bm c,\bm y)],
\end{equation}
respectively. Since $\mathcal Z$ is compact, $R_{\mathcal Z}:=\max_{\bm z\in\mathcal Z}\|\bm z\|_2<\infty$ and
\begin{equation}\label{eq:appE_loss_envelope}
    |\ell_{\mathcal{RSPO}}(\bm c,\bm y)|\leq2R_{\mathcal Z}\|\bm y\|_2
    \qquad\text{for every }\bm c\in\mathbb R^d ,
\end{equation}
so whenever $\mathbb E\|\bm y\|_2<\infty$ every target risk appearing below is finite and the tower property $R(\bm g)=\mathbb E_{\bm x\sim \bb P_{\bm x}}[R(\bm g(\bm x);\bm x)]$ applies to each. Our analysis leverages the oracle predictor $\bm g_0$, and hence requires the following condition about $\bm g_0$.

\begin{condition}[Oracle Predictor Condition]
    \label{con:reference_condition}
    \begin{enumerate}
        \item[]
        \item[(i)] $\bm g_0$ is measurable and, for almost every $\bm x$, $\bm g_0(\bm x)$ minimizes $R_+(\cdot;\bm x)$ over $\mathbb R^d$;
        \item[(ii)] for $\mathbb{P}_{\bm x}$-almost every $\bm x$, $R_+(\cdot;\bm x)$ is convex on $\mathbb R^d$ and $\mu$-strongly convex on the closed ball $\mathcal B(\bm g_0(\bm x),\rho)$, which is centered at $\bm g_0(\bm x)$ with radius $\rho$. Here, $\mu,\rho>0$ do not depend on $\bm x$;
        \item[(iii)] $\bm g_0(\bm x)\in \mathcal{Y}^{\star}_{\mathcal{RSPO}}(\bm x)$ for $\mathbb P_{\bm x}$-almost every $\bm x$;
        \item[(iv)]  $\|\bm g_0(\bm x)\|_2\leq\bar\beta$ almost surely for some $0<\bar\beta<\infty$.
    \end{enumerate}
\end{condition}

Condition~\ref{con:reference_condition}(i) and condition~\ref{con:reference_condition}(iii) are implied by the results established in this paper and we provide sufficient conditions for (ii) in Lemma~\ref{lem:local_sc}. With $a=1$, Condition~\ref{con:reference_condition}(i) holds under the assumptions of Corollary~\ref{cor:RSPO+a1}. Condition~\ref{con:reference_condition}(ii) is a local curvature requirement. We provide an example where it holds by Lemma~\ref{lem:local_sc} in Appendix~\ref{app:technical_lemmas}. The example, roughly speaking, requires that the distribution of $\bm y$ must place enough mass on a region centered at $\bm g_0(\bm x)$ obtained by rescaling the feasible set. Since Lemma~\ref{lem:local_sc} is stated for a fixed $\bm x$, Condition~\ref{con:reference_condition}(ii) follows with $\mu=\operatorname*{ess\,inf}_{\bm x}p_0(\bm x)/\gamma$ provided $\operatorname*{ess\,inf}_{\bm x}p_0(\bm x)>0$. Condition~\ref{con:reference_condition}(iii) holds by Theorem~\ref{thm:Fisher_consistency} whenever $\mathcal Z$ is a bounded polyhedron with nonempty interior and $\gamma\leq\bar\gamma(\bm x)$ for almost every $\bm x$, which a fixed $\gamma>0$ satisfies as soon as $\gamma\leq\operatorname*{ess\,inf}_{\bm x}\bar\gamma(\bm x)$. When conditions (i)--(iii) hold, $\bm g_0(\bm x)$ is the unique minimizer of $R_+(\cdot;\bm x)$ due to strong convexity, and hence Fisher consistency holds by definition. Condition~(iv) holds trivially when $\bm y$ is uniformly bounded.

\subsubsection*{Bounds on surrogate excess risk.} In this step, we bound surrogate excess risk 
\begin{equation*}
    R_+(\widehat{\bm{g}}_{{\rm RSPO}_+}) - R^{\star}_+=\big[R_+(\widehat{\bm{g}}_{{\rm RSPO}_+}) - R^{\star}_+(\mathcal G)\big]+\big[R^{\star}_+(\mathcal G)-R^{\star}_+\big],
\end{equation*}
using generalization bound for the surrogate empirical risk minimizer. Define the uniform upper bound of the surrogate loss over the prediction image $\mathcal G(\mathcal X):=\{\bm g(\bm x):\bm g\in\mathcal G,\, \bm x \in \mathcal X\}$ and the label set $\mathcal Y$,
\begin{equation*}
    \mathfrak L(\gamma)
    :=
    \sup_{\hat{\bm y}\in\mathcal G(\mathcal X),\,\bm y\in\mathcal Y}
    \ell_{\mathcal{RSPO}_+}(\hat{\bm y},\bm y).
\end{equation*}

\begin{remark}\label{rem:L_gamma_bound}
If $\|\bm y\|_2\leq r$ on $\mathcal Y$ and $\|\hat{\bm y}\|_2\leq G_\infty$ on $\mathcal G(\mathcal X)$, the Cauchy--Schwarz inequality yields $\mathfrak L(\gamma)\leq 2(r+aG_\infty)R_{\mathcal Z}+\frac{a\gamma}{2}R_{\mathcal Z}^{2}$, which implies that $\mathfrak L(\gamma)$ grows at most linearly in $\gamma$.
\end{remark}

\begin{proposition}\label{prop:rspo+_generalization}
Let $\mathcal{G}$ be a hypothesis class from $\mathcal{X}$ to $\mathbb{R}^d$. Suppose that $\mathcal{Z}$ is bounded with diameter $D_{\mathcal{Z}}$, that $\mathcal{G}(\mathcal{X})$ and $\mathcal{Y}$ are bounded so that $\mathfrak{L}(\gamma)<\infty$, and that $\mathfrak{R}_d(\mathcal{G})\leq C_0/\sqrt{N}$ for a constant $C_0$. Then, for any $\delta\in(0,1)$, with probability at least $1 - \delta$,
\begin{equation*}
    R_+(\widehat{\bm{g}}_{{\rm RSPO}_+}) - R^{\star}_+(\mathcal{G}) \leq \frac{4\sqrt{2}a D_{\mathcal{Z}}C_0}{\sqrt{N}}+\mathfrak{L}(\gamma)\sqrt{\frac{2\log(2/\delta)}{N}} .
\end{equation*}
\end{proposition}

Proposition~\ref{prop:rspo+_generalization} controls the generalization error  within function class $\mathcal G$. However, Fisher consistency is only meaningful over the set of all measurable predictors, and we have to account for the expressiveness of $\mathcal G$ itself. Accordingly, for $\delta\in(0,1)$ we write, at $a=1$,
\begin{equation}\label{eq:approx_error}
    B_N:=\frac{4\sqrt{2}D_{\mathcal Z}C_0}{\sqrt{N}}+\mathfrak L(\gamma)\sqrt{\frac{2\log(2/\delta)}{N}},
    \qquad
    A^+_{\mathcal G}:=R^{\star}_+(\mathcal G)-R^{\star}_+\geq0.
\end{equation}
Under the conditions in Proposition~\ref{prop:rspo+_generalization}, we have
\begin{equation}\label{eq:surrogate_excess_risk}
    R_+(\widehat{\bm{g}}_{{\rm RSPO}_+}) - R^{\star}_+\leq B_N+A^+_{\mathcal G}.
\end{equation}
The first term $B_N$ vanishes at the rate $N^{-1/2}$, and the second term $A^+_{\mathcal G}$ is a deterministic property of the pair $(\mathcal G,\mathbb P)$. We briefly discuss when $A^+_{\mathcal G}$ vanishes. Under the conditions of Corollary~\ref{cor:RSPO+a1}, the pointwise minimizer of the surrogate risk is $\bm g_0(\bm x)=2\bar{\bm y}(\bm x)$, so $R^{\star}_+$ is attained by $\bm g_0$. Thus, $\bm g_0\in\mathcal G$ up to $\mathbb P_{\bm x}$-null sets implies $A^+_{\mathcal G}=0$. Conversely, if the infimum defining $R^{\star}_+(\mathcal G)$ is attained by some predictor in $\mathcal G$, $A^+_{\mathcal G}=0$ only if $\bm g_0\in\mathcal G$ up to null sets due to uniqueness of the pointwise minimizer.

\subsubsection*{Bounds on target excess risk.} Based on \eqref{eq:surrogate_excess_risk}, we bound the distance between $\widehat{\bm g}_{{\rm RSPO}_+}$ and $\bm g_0$ by the local strong convexity condition. Then, we derive the bound of excess risk $R(\widehat{\bm g}_{{\rm RSPO}_+})-R^\star$, by applying Lipschitz property of $R(\cdot)$ and the fact that $R(\bm g_0)=R^{\star}$ due to Fisher consistency, which leads to the following theorem. The detailed steps are elaborated in its proof.

\begin{theorem}[Target Excess Risk: General Case]
\label{thm:direct_closure}
Fix $a=1$ and $\gamma>0$, and suppose Condition~\ref{con:reference_condition} and the assumptions of Proposition~\ref{prop:rspo+_generalization} hold. Then, for every $\delta\in(0,1)$, with probability at least $1-\delta$,
\begin{equation}\label{eq:target_bound}
    R(\widehat{\bm g}_{{\rm RSPO}_+})-R^\star
    \leq
    \frac{\bar\beta}{{2}\gamma}
    \left[
        \sqrt{\frac{2\left(B_N+A^+_{\mathcal G}\right)}{\mu}}
        +
        \frac{2\left(B_N+A^+_{\mathcal G}\right)}{\mu\rho}
    \right],
\end{equation}
where $\mu,\rho$ are the curvature constants of Condition~\ref{con:reference_condition}(ii), $\bar\beta$ is the envelope constant of Condition~\ref{con:reference_condition}(iv), and $B_N,A^+_{\mathcal G}$ are defined in \eqref{eq:approx_error}. If, in addition, $\bm g_0\in\mathcal G$ up to $\mathbb P_{\bm x}$-null sets and the structural constants $D_{\mathcal Z},C_0,\mathfrak L(\gamma),\mu,\rho,\bar\beta,\gamma$ do not depend on $N$, then
\begin{equation}\label{eq:appE_vanishing}
    R(\widehat{\bm g}_{{\rm RSPO}_+})-R^\star
    =
     \mathcal O_p(N^{-1/4}).\footnotemark\footnotetext{Here $\mathcal O_p$ denotes stochastic order: $X_N=\mathcal O_p(a_N)$ if, for every $\varepsilon>0$, there exists $M>0$ such that $\mathbb P(|X_N|>Ma_N)<\varepsilon$ for all sufficiently large $N$.}
\end{equation}
\end{theorem}

This is consistent with the excess risk bounds of the $\mathcal{RSPO}$ predictor derived in Section~\ref{subsec:rspo_excess_risk_bounds}: given the structural constants, as the formulation is more robust (larger $\gamma$), fewer samples are needed to achieve the same excess risk bound of the $\mathcal{RSPO}_+$ predictor, by noting that the term $\mathfrak L(\gamma)$ in $B_N$ grows at most linearly in $\gamma$ (Remark~\ref{rem:L_gamma_bound}). If we further impose a quadratic target upper-growth condition, we can improve the bound on target excess risk from order $N^{-1/4}$ to $N^{-1/2}$, as the following corollary shows.

\begin{corollary}[Target Excess Risk: Fast Rate]
\label{cor:direct_fast_rate}
Suppose the assumptions of Theorem~\ref{thm:direct_closure} hold, and assume further the following uniform quadratic target-growth condition: there exist constants $L,r_0>0$, independent of $\bm x$, such that, for almost every $\bm x$,
\begin{equation}\label{eq:target_quadratic}
    R(\bm c;\bm x)-\inf_{\bm c'\in\mathbb R^d}R(\bm c';\bm x)
    \leq
    \frac{L}{2}
    \|\bm c-\bm g_0(\bm x)\|_2^2,
    \quad
    \text{whenever }
    \|\bm c-\bm g_0(\bm x)\|_2\leq r_0 .
\end{equation}
Then, for every $\delta\in(0,1)$, with probability at least $1-\delta$,
\begin{equation}\label{eq:fast_target}
    R(\widehat{\bm g}_{{\rm RSPO}_+})-R^\star
    \leq
    \frac{B_N+A^+_{\mathcal G}}{\min\left\{
        \mu/L,\, \mu\gamma r_0/\bar\beta,\, \mu\gamma \rho/\bar\beta
    \right\}}.
\end{equation}
If, in addition, $\bm g_0\in\mathcal G$ up to $\mathbb P_{\bm x}$-null sets and the structural constants do not depend on $N$, then $R(\widehat{\bm g}_{{\rm RSPO}_+})-R^\star=\mathcal O_p(N^{-1/2})$.
\end{corollary}

Condition~\eqref{eq:target_quadratic} asks that the target risk not grow faster than quadratically near the oracle predictor, complementing from above the lower bound of the growth of surrogate risk from Condition~\ref{con:reference_condition}(ii). This implies the growth rate of the target risk is at most comparable to that of the surrogate risk.

\section{Numerical Experiments}\label{sec:numeric_study}

In this section, we examine whether using the same robust decision map to define the training loss and generate deployed decisions improves out-of-sample decision quality. 
We conduct computational experiments on synthetic instances of two canonical problem classes: capacitated transportation and risk-constrained portfolio optimization. 

\noindent\underline{Synthetic Data-Generation Process.}~Following a synthetic design similar to that of \citet{elmachtoub2022smart}, we use the same basic data-generating protocol for the two problem classes, where $d$ denotes the dimension of the cost and decision vectors and $p$ the dimension of the context vector. 
For each experimental configuration, we draw a ground-truth coefficient matrix $\bm B^\star\in\mathbb R^{d\times p}$ once, with independent $\operatorname{Bernoulli}(0.5)$ entries, and hold it fixed across the $20$ replications. For each observation, the context vector $\bm x_i\in\mathbb R^p$ is drawn from a standard multivariate Gaussian distribution; that is, $\bm x_i\sim\mathcal N(\bm 0,\bm I_p)$. Given $\bm x_i$, the cost vector $\bm y_i\in\mathbb R^d$ is generated according to $\bm y_i=f_{\mathtt{deg}}(\bm B^\star\bm x_i,\bm\epsilon_i)$, where $\bm\epsilon_i$ denotes the noise term and $f_{\mathtt{deg}}$ is polynomial in its first argument. The problem-specific form of $f_{\mathtt{deg}}$ and the noise distribution are stated in Sections~\ref{subsec:transportation} and~\ref{subsec:portfolio_rspo+}. The positive integer parameter, $\mathtt{deg}$, determines the polynomial degree. Importantly, since all methods employ a linear hypothesis class for the predictor, $\mathtt{deg}$ controls the extent of \textit{model misspecification}. Across the numerical experiments, we vary the polynomial degree $\mathtt{deg}$, the noise level, the decision (and cost) dimension $d$, and the training-sample size $N$ to evaluate different aspects of the methods' performance.

\noindent\underline{Methods Compared.}~As summarized in Table~\ref{tab:method_comparison1}, the four empirical methods instantiate the four learning--decision pipelines introduced in Section~\ref{sec:introduction}. All four methods use the same linear prediction class, $\bm g_{\theta}(\bm x)=\bm B\bm x$, but differ in the loss used to train the predictor and the decision map used to generate the deployed decision. 
First, least squares minimizes the squared prediction loss $\ell(\hat{\bm y},\bm y)=\frac{1}{2}\|\hat{\bm y}-\bm y\|_2^2$ and uses the nominal decision map. 
Second, $\mathcal{SPO}_+$ trains the predictor with the nominal $\mathcal{SPO}_+$ surrogate and uses the same nominal decision map for decision making. 
Third, $\mathcal{SPO}_+$ with robust decisions retains the $\mathcal{SPO}_+$ training loss but uses the robust decision map~\ref{ilroopt} to generate deployed decisions. 
Finally, $\mathcal{RSPO}_+$ trains the predictor by solving~\eqref{eq:rspo+_ermprob} and uses the same robust decision map to generate deployed decisions. Thus, among the four methods, $\mathcal{RSPO}_+$ is the only one that combines robust decision-making with learning--decision alignment.

\noindent\underline{Training, Validation, and Testing.}~
For each replication, we independently generate a new training sample and test sample. We apply the same sample-splitting protocol to all methods. 
We train the candidate predictors on $70\%$ of the training sample and select their hyperparameters using the remaining $30\%$. 
Every method selects the regularization parameter $\lambda$ from five logarithmically spaced values in $[10^{-3},10]$. 
For $\mathcal{RSPO}_+$ and $\mathcal{SPO}_+$ with robust decisions, we additionally select the robustness parameter $\gamma$ from five logarithmically spaced values in $[10^{-6},10]$. 
Least squares uses the average prediction loss on the validation data for hyperparameter selection, whereas $\mathcal{SPO}_+$, $\mathcal{SPO}_+$ with robust decisions, and $\mathcal{RSPO}_+$ use the average normalized decision loss. 
We refer to the predictor associated with the selected hyperparameters as the validated predictor. 
Each validated predictor is evaluated on the corresponding test sample of size $n_{\rm test}=100$ in each of the $20$ replications. 

\noindent\underline{Performance Metrics.}~
All quantities below are defined within one experimental replication, with the replication index suppressed. 
For each method $m$, let $\widehat{\bm g}_m$ denote its validated predictor and let $\gamma_m$ denote the robustness parameter of the robust decision map used to generate its deployed decision. 
For notational uniformity, we set $\gamma_m=0$ for least squares and $\mathcal{SPO}_+$ and write $\bm z_0^\star$ for the decision returned by the nominal decision map~\eqref{eq:ilolearningoracle}. 
For $\mathcal{RSPO}_+$ and $\mathcal{SPO}_+$ with robust decisions, $\gamma_m$ is the value of $\gamma$ selected separately for method $m$ by the validation procedure described above. 
Under this convention, every method deploys a decision denoted by $\bm z^\star_{\gamma_m}(\widehat{\bm g}_m(\bm x))$; when $\gamma_m>0$, this decision is returned by the robust decision map~\ref{ilroopt}. 
On the test sample $\{(\bm x_i,\bm y_i)\}_{i=1}^{n_{\rm test}}$, we measure prediction accuracy by
\begin{equation*}
{\rm Relative\ Prediction\ Loss}_m
=
\frac{\sum_{i=1}^{n_{\rm test}}\|\widehat{\bm g}_m(\bm x_i)-\bm y_i\|_2^2}
{\sum_{i=1}^{n_{\rm test}}\|\bm y_i\|_2^2}.
\end{equation*}
Let $v^\star(\bm y_i)=\min_{\bm z\in\mathcal Z}\bm y_i^\top\bm z$ denote the nominal clairvoyant cost under the realized cost vector $\bm y_i$. We measure decision quality by
\begin{equation*}
{\rm Normalized\ Decision\ Loss}_m
=
\frac{\sum_{i=1}^{n_{\rm test}}\{\bm y_i^\top\bm z^\star_{\gamma_m}(\widehat{\bm g}_m(\bm x_i))-v^\star(\bm y_i)\}}
{\sum_{i=1}^{n_{\rm test}}|v^\star(\bm y_i)|}.
\end{equation*}
The normalized decision loss compares the realized cost of each deployed decision with this nominal clairvoyant benchmark. 
We summarize the replication-level ratios over $20$ replications with boxplots; all comparisons below refer to their medians. 
We use normalized decision loss as the primary performance measure and relative prediction loss as a complementary diagnostic of whether improved decisions coincide with more accurate cost predictions. 

\subsection{Capacitated Transportation Problem}
\label{subsec:transportation}

We first consider a capacitated transportation problem with $I$ supply nodes and $J$ demand nodes. 
We index supply nodes by $s\in[I]$ and demand nodes by $t\in[J]$, and let $Q_s$ and $D_t$ denote their respective supply capacities and demands, which satisfy $\sum_{s\in[I]}Q_s\geq\sum_{t\in[J]}D_t$ to ensure feasibility. The variable $z_{st}$ denotes the amount shipped from supply node $s$ to demand node $t$. 
Given a predicted unit-cost vector $\hat{\bm y}=(\hat y_{st})$, the transportation problem is
\begin{equation*}
\bm z^\star_\gamma(\hat{\bm y})\in\argmin_{\bm z}\Bigg\{\hat{\bm y}^{\top}\bm z+\frac{\gamma}{2}\|\bm z\|_2^2:\ z_{st}\geq 0,\ \sum_{t\in[J]}z_{st}\leq Q_s, \ \sum_{s\in[I]}z_{st}=D_t, \ \forall s\in[I], t\in[J]\Bigg\}.
\end{equation*}
The network has $d=IJ$ arcs, and we identify each arc $(s,t)$ with the coordinate $k=(s-1)J+t\in[d]$ when forming the cost and decision vectors. 
For this problem, the unit cost of arc $k$ in observation $i$ is generated as
\begin{equation*}
y_{ik}
=
\left[
\left(
\frac{1}{\sqrt p}(\bm B^\star\bm x_i)_k+3
\right)^{\mathtt{deg}}
+1
\right]\epsilon_{ik},
\end{equation*}
where the $\epsilon_{ik}$ are independently drawn from $\operatorname{Unif}[1-\bar\epsilon,1+\bar\epsilon]$. Hence, $\mathtt{deg}$ controls nonlinearity relative to the  linear prediction class and $\bar\epsilon$ the noise level. 

We organize the transportation results around the four factors considered in our experiments---training-sample size, decision dimension, noise level, and model misspecification---varying one factor at a time while holding the others fixed. Each figure reports normalized decision loss and relative prediction loss to compare downstream decision quality with prediction accuracy.

\subsubsection{Effect of Training-Sample Size}
\begin{figure}[!t]
    \centering
    \includegraphics[width=0.9\linewidth]{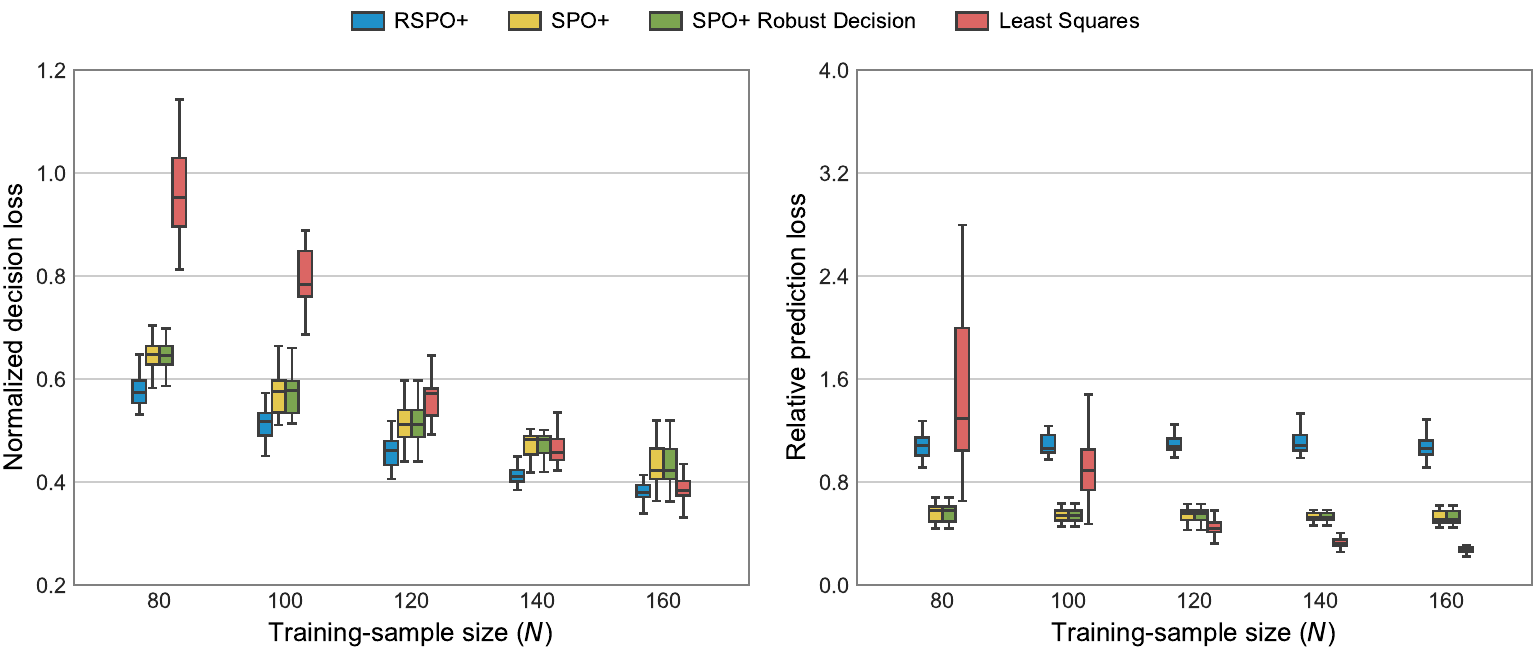}
    \caption{Normalized decision loss (left) and relative prediction loss (right) as the training-sample size $N$ varies on transportation optimization instances, with $p=60$, $d=100$, $\mathtt{deg}=4$, and $\bar\epsilon=0.4$ fixed.}
    \label{fig:transportation_sample_size_combined_metric}\vspace{-4mm}
\end{figure}

We first examine how the training-sample size affects out-of-sample performance. 
We fix $p=60$, $d=100$, $\mathtt{deg}=4$, and $\bar\epsilon=0.4$, and vary $N\in\{80,100,120,140,160\}$. 
Figure~\ref{fig:transportation_sample_size_combined_metric} reports the normalized decision loss and relative prediction loss. 
As $N$ increases, the normalized decision loss of $\mathcal{RSPO}_+$ decreases. It attains the lowest decision loss for $N\leq140$ and is comparable to the lowest value at $N=160$. Its margin over $\mathcal{SPO}_+$ and $\mathcal{SPO}_+$ with robust decisions is largest for the smaller training samples and becomes smaller as $N$ increases. 
The prediction panel separates this decision advantage from prediction accuracy: $\mathcal{SPO}_+$ and $\mathcal{SPO}_+$ with robust decisions attain lower relative prediction loss, while least squares improves substantially as $N$ increases. 

\subsubsection{Effect of Decision Dimension}
\begin{figure}[!t]
    \centering
    \includegraphics[width=0.9\linewidth]{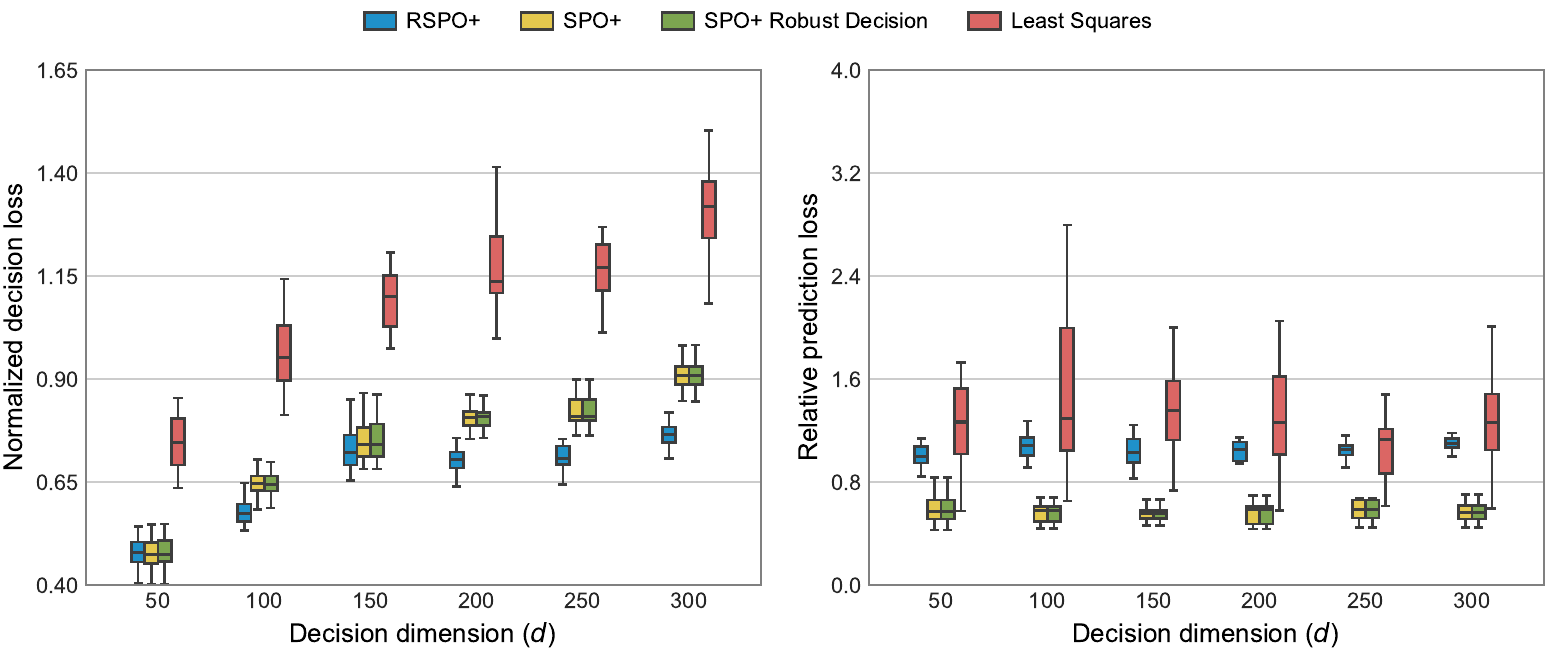}
    
    \caption{Normalized decision loss (left) and relative prediction loss (right) as the decision dimension $d$ varies on transportation optimization instances, with $N=80$, $p=60$, $\mathtt{deg}=4$, and $\bar\epsilon=0.4$ fixed.}
    \label{fig:transportation_decision_dimension_combined}\vspace{-4mm}
\end{figure}

We next examine how the decision dimension affects performance by scaling the transportation network. 
We fix $N=80$, $p=60$, $\mathtt{deg}=4$, and $\bar\epsilon=0.4$, and vary $d\in\{50,100,150,200,250,300\}$. 
Figure~\ref{fig:transportation_decision_dimension_combined} reports the normalized decision loss and relative prediction loss. 
As $d$ increases, the normalized decision loss of $\mathcal{RSPO}_+$ increases overall, while its relative advantage becomes more pronounced. $\mathcal{RSPO}_+$, $\mathcal{SPO}_+$, and $\mathcal{SPO}_+$ with robust decisions perform comparably at $d=50$, and $\mathcal{RSPO}_+$ attains the lowest decision loss for every $d\geq100$, with its largest margins over $\mathcal{SPO}_+$ and $\mathcal{SPO}_+$ with robust decisions at $d\in\{200,250,300\}$. 
Despite having higher relative prediction loss than $\mathcal{SPO}_+$ and $\mathcal{SPO}_+$ with robust decisions, $\mathcal{RSPO}_+$ yields lower decision loss throughout $d\geq100$. 

\subsubsection{Effect of Noise Level}
\begin{figure}[!t]
    \centering
    \includegraphics[width=0.9\linewidth]{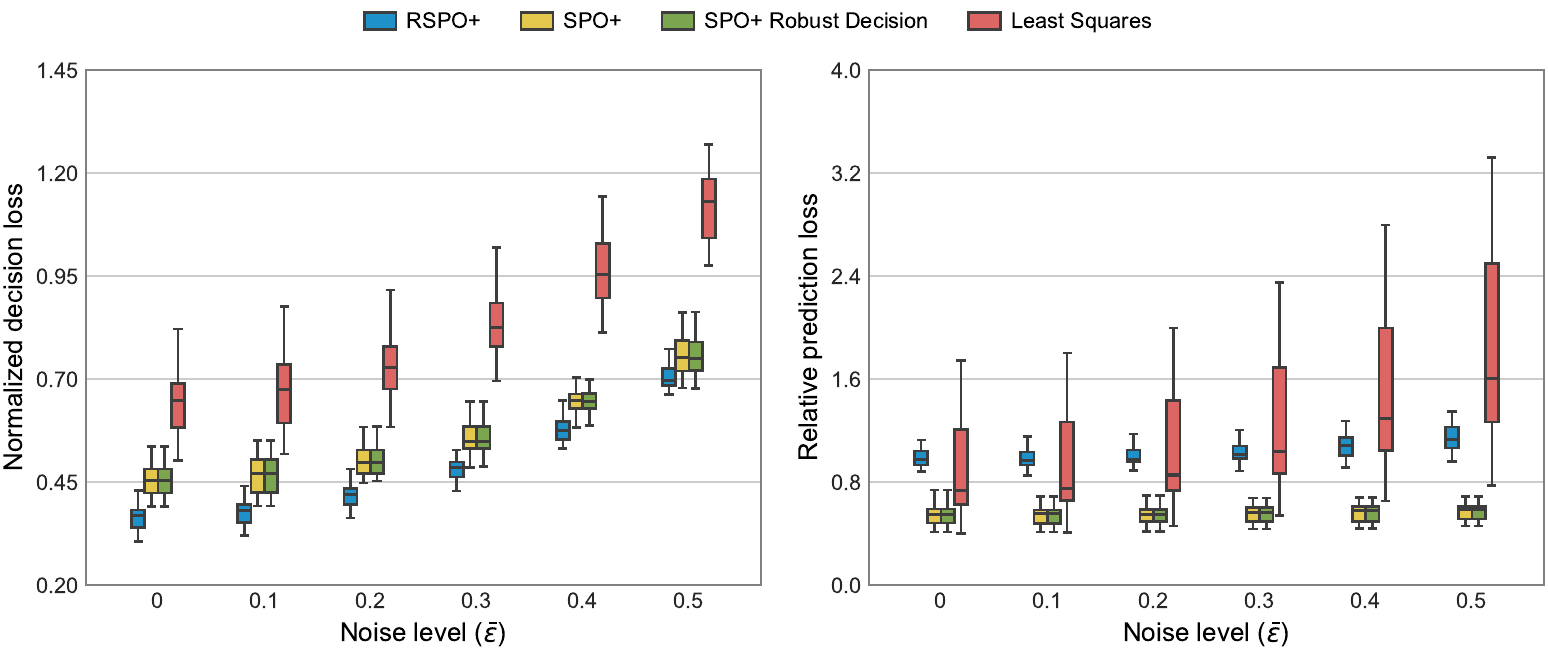}
    
    \caption{Normalized decision loss (left) and relative prediction loss (right) as the noise level $\bar\epsilon$ varies on transportation optimization instances, with $N=80$, $p=60$, $d=100$, and $\mathtt{deg}=4$ fixed.}
    \label{fig:transportation_noise_width_combined}\vspace{-4mm}
\end{figure}

We next examine how the noise level affects performance. 
We fix $N=80$, $p=60$, $d=100$, and $\mathtt{deg}=4$, and vary $\bar\epsilon\in\{0,0.1,0.2,0.3,0.4,0.5\}$. 
Figure~\ref{fig:transportation_noise_width_combined} shows that the normalized decision loss of $\mathcal{RSPO}_+$ increases as the noise level rises. It attains the lowest decision loss throughout the reported range, with its largest margins over $\mathcal{SPO}_+$ and $\mathcal{SPO}_+$ with robust decisions occurring at the lower noise levels. Least squares has the largest decision loss and the widest across-replication variation. 
$\mathcal{SPO}_+$ and $\mathcal{SPO}_+$ with robust decisions attain lower relative prediction loss than $\mathcal{RSPO}_+$, while $\mathcal{RSPO}_+$ attains lower decision loss at every reported noise level. 

\subsubsection{Effect of Model Misspecification}
\begin{figure}[!t]
    \centering
    \includegraphics[width=0.9\linewidth]{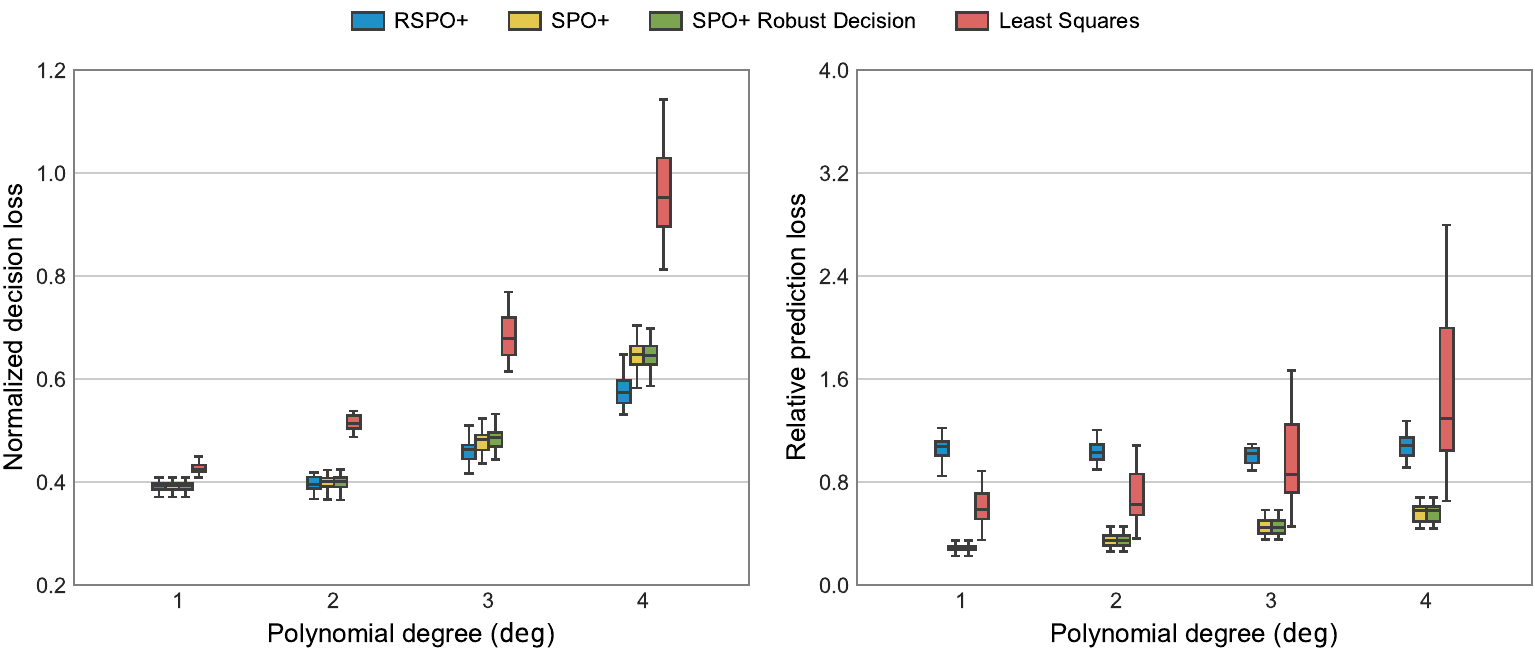}
    
    \caption{Normalized decision loss (left) and relative prediction loss (right) as model misspecification ($\mathtt{deg}$) varies on transportation optimization instances, with $N=80$, $p=60$, $d=100$, and $\bar\epsilon=0.4$ fixed.}
    \label{fig:transportation_polynomial_degree_combined}\vspace{-4mm}
\end{figure}

Finally, we examine how model misspecification affects performance. 
We fix $N=80$, $p=60$, $d=100$, and $\bar\epsilon=0.4$, and vary $\mathtt{deg}\in\{1,2,3,4\}$. 
Because all methods use a linear prediction class, increasing $\mathtt{deg}$ increases the misspecification of the conditional cost function. 
Figure~\ref{fig:transportation_polynomial_degree_combined} shows that the normalized decision loss of $\mathcal{RSPO}_+$ increases with misspecification, but less sharply than the losses of the other methods. $\mathcal{RSPO}_+$, $\mathcal{SPO}_+$, and $\mathcal{SPO}_+$ with robust decisions perform comparably at $\mathtt{deg}=1$; for every nonlinear specification, $\mathcal{RSPO}_+$ attains the lowest decision loss, and its margin over $\mathcal{SPO}_+$ and $\mathcal{SPO}_+$ with robust decisions grows with $\mathtt{deg}$. Least squares exhibits the largest deterioration. 
Despite having higher relative prediction loss than $\mathcal{SPO}_+$ and $\mathcal{SPO}_+$ with robust decisions throughout the sweep, $\mathcal{RSPO}_+$ produces better decisions under the nonlinear specifications.

\subsubsection{Value of Learning--Decision Alignment}

To assess the value of learning--decision alignment, we compare $\mathcal{RSPO}_+$ with $\mathcal{SPO}_+$ with robust decisions. Both methods tune their own $(\lambda,\gamma)$, but only $\mathcal{RSPO}_+$ uses the deployed robust decision map to define its training loss.
Across Figures~\ref{fig:transportation_sample_size_combined_metric}--\ref{fig:transportation_polynomial_degree_combined}, $\mathcal{RSPO}_+$ attains decision loss that is lower than or comparable to that of $\mathcal{SPO}_+$ with robust decisions. The two methods perform comparably at $d=50$ and $\mathtt{deg}=1$, while the difference is more visible at higher decision dimensions and under the nonlinear specifications. These decision gains occur even though $\mathcal{SPO}_+$ with robust decisions attains lower relative prediction loss. Thus, the transportation results support the value of learning--decision alignment.

\begin{figure}[t]
	\centering
	\includegraphics[width=0.85\linewidth]{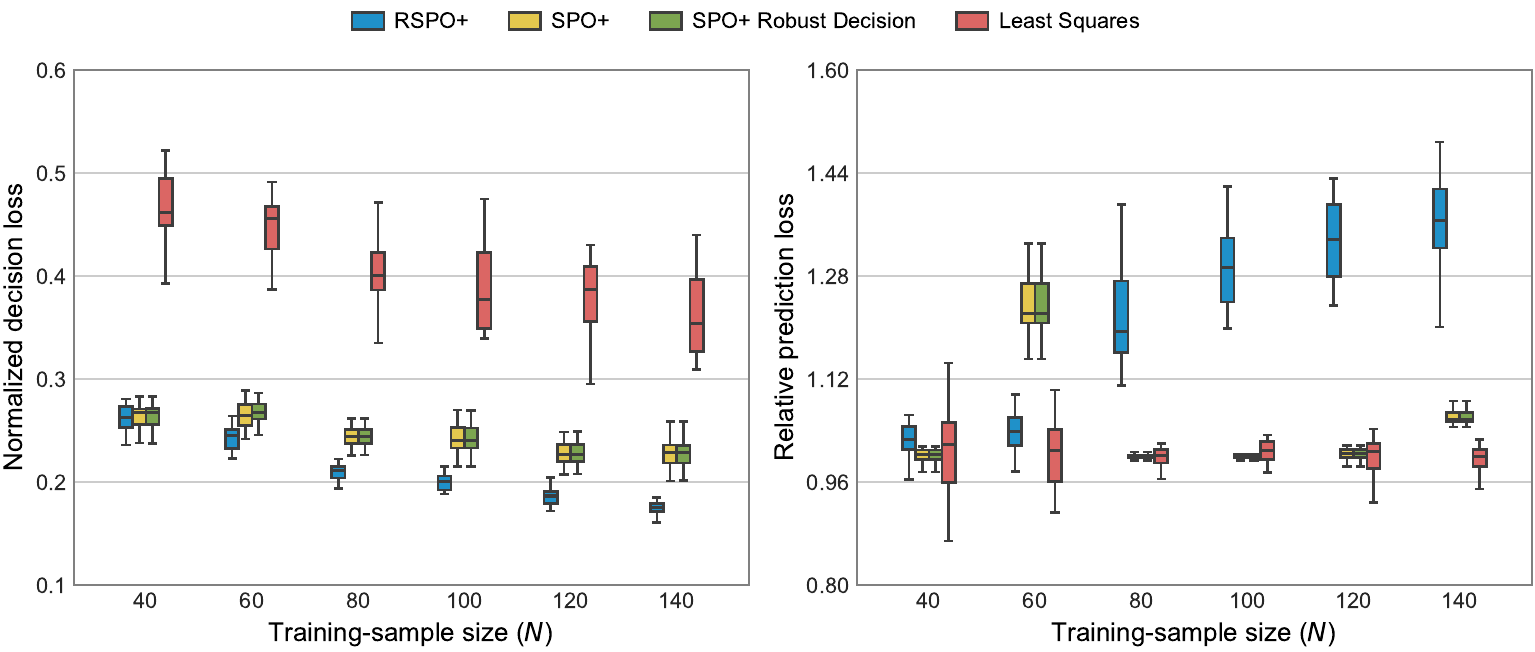}
	
	\includegraphics[width=0.85\linewidth,trim=0 0 0 24bp,clip]{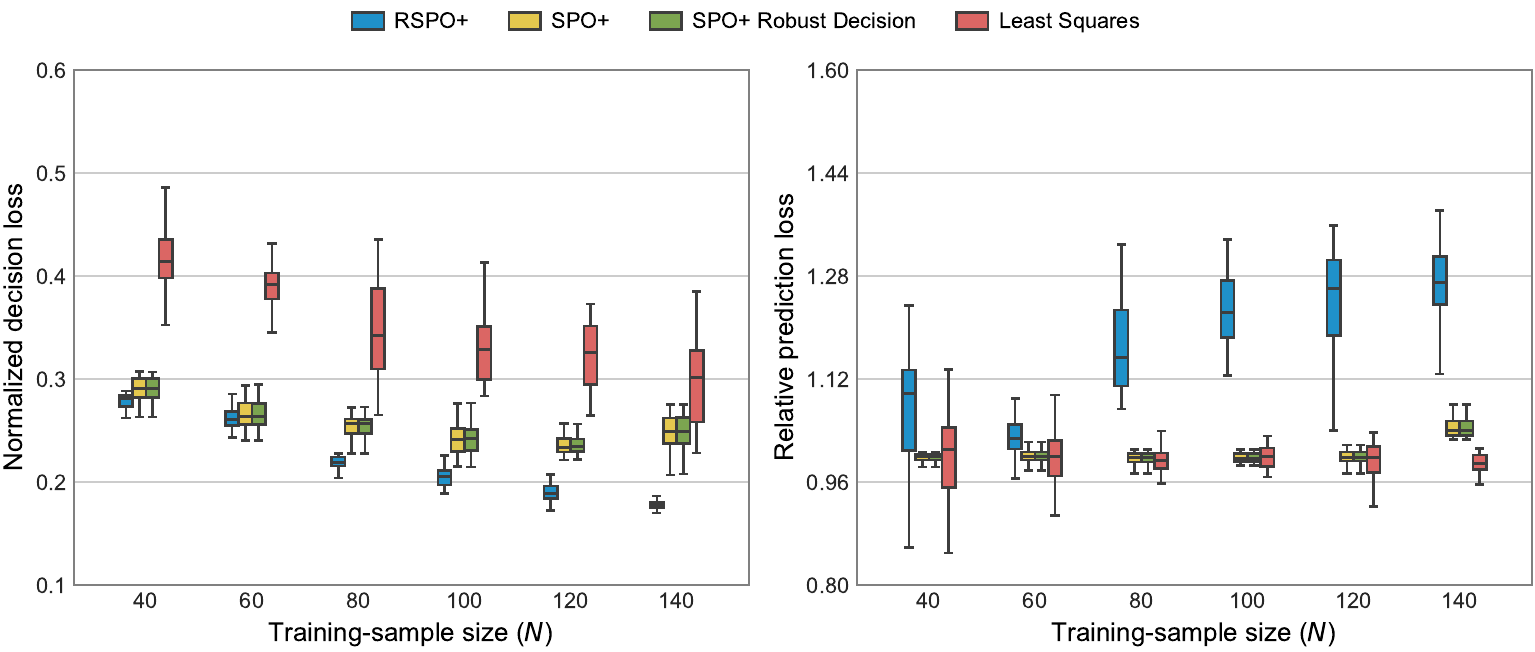}
	
	\includegraphics[width=0.85\linewidth,trim=0 0 0 24bp,clip]{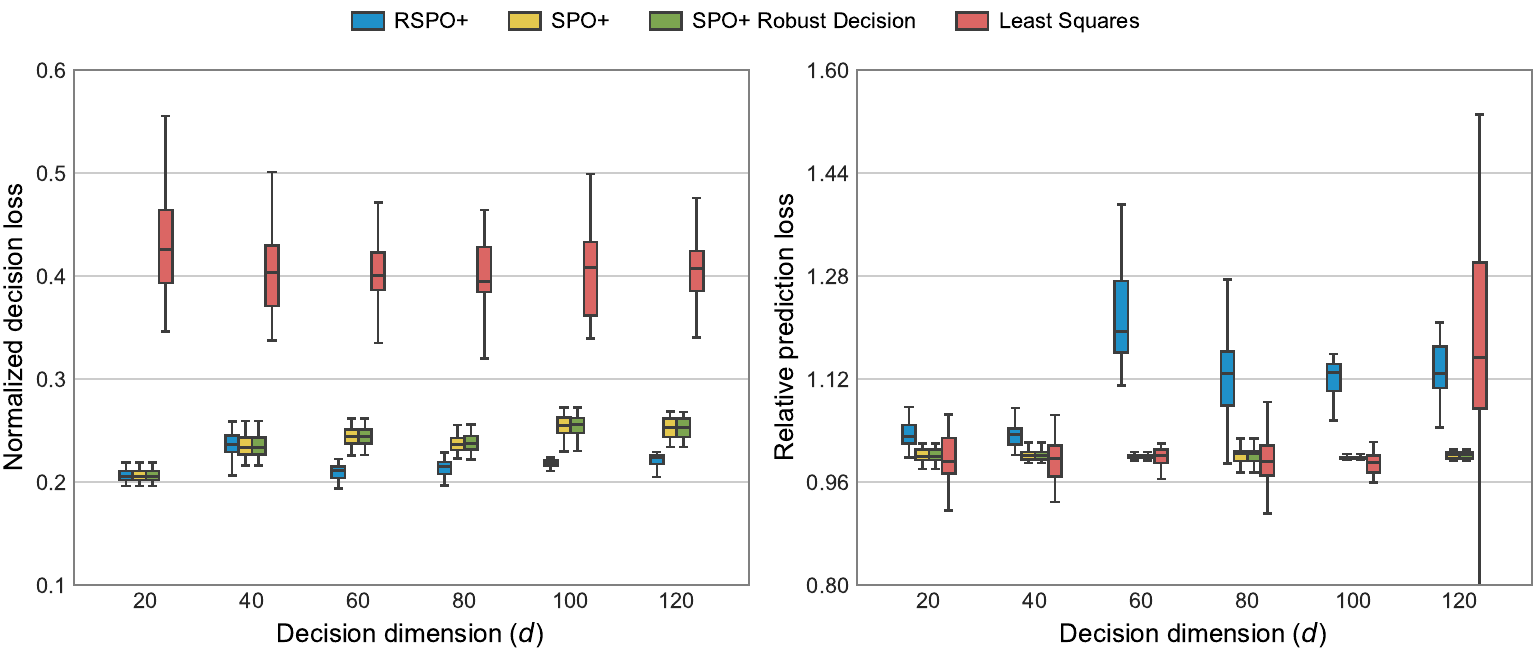}
	\caption{Normalized decision loss (left) and relative prediction loss (right) on portfolio optimization instances. The first and second rows vary the training-sample size $N$ under $\mathtt{deg}=4$ and $\mathtt{deg}=1$, respectively, with $p=80$, $d=60$, and $\tau=1$ fixed. The third row varies the number of assets $d$ with $N=80$, $p=80$, $\mathtt{deg}=4$, and $\tau=1$ fixed.}
	\label{fig:portfolio_combined_effects}\vspace{-6mm}
\end{figure}

\subsection{Portfolio Optimization Problem} \label{subsec:portfolio_rspo+}
We next study a risk-constrained portfolio problem in which $\bm z$ specifies the weights assigned to $d$ assets. 
This experiment examines whether the transportation results extend to a continuous allocation problem with an explicit risk constraint. 
We generate correlated returns by drawing the entries of a four-factor loading matrix $\bm F\in\mathbb R^{d\times4}$ independently from $\operatorname{Unif}[-0.0025\tau,0.0025\tau]$ and setting $\bm\Sigma=\bm F \bm F^\top+(0.01\tau)^2\bm I_d$. For observation $i$, the return vector is generated as
\begin{equation*}
\bm r_i=\left(\frac{0.05}{\sqrt p}\bm B^\star\bm x_i+0.1^{1/\mathtt{deg}}\bm 1\right)^{\mathtt{deg}}+\bm F\bm l_i+0.01\tau\bm\epsilon_i,
\end{equation*}
where $\bm l_i\sim\mathcal N(\bm 0,\bm I_4)$ and $\bm\epsilon_i\sim\mathcal N(\bm 0,\bm I_d)$ are independent, the power is applied component-wise, and the cost vector is $\bm y_i=-\bm r_i$. The terms $\bm F\bm l_i$ and $0.01\tau\bm\epsilon_i$ represent common-factor variation and asset-specific noise, respectively, while $\mathtt{deg}$ controls misspecification of the common linear prediction class. 
Given a predicted cost vector $\hat{\bm y}$, the deployed portfolio solves
\begin{equation*}
\bm z_\gamma^\star(\hat{\bm y})\in\argmin_{\bm z}\left\{\hat{\bm y}^{\top}\bm z+\frac{\gamma}{2}\|\bm z\|_2^2:\ \bm z\geq\bm 0,\ \bm 1^\top\bm z\leq1,\ \|\bm\Sigma\bm z\|_1\leq\beta\right\}, 
\end{equation*}
where $\|\bm\Sigma\bm z\|_1$ measures aggregate risk exposure, and $\beta$ is the risk budget. We use the $\ell_1$ risk constraint because its standard epigraph representation is polyhedral, allowing us to apply Theorem~\ref{thm:rspo_plus_reformulation}. Appendix~\ref{app:portfolio_reformulation} derives the corresponding learning reformulations, and Appendix~\ref{app:portfolio_l2} reports the quadratic-risk counterpart.

\subsubsection{Effects of Sample Size, Model Misspecification, and Decision Dimension}

We next examine how performance changes with training-sample size, model misspecification, and the number of assets. Figure~\ref{fig:portfolio_combined_effects} summarizes the three experiments.  Under $\mathtt{deg}=4$, the normalized decision loss of $\mathcal{RSPO}_+$ decreases steadily with $N$; it is comparable to $\mathcal{SPO}_+$ and $\mathcal{SPO}_+$ with robust decisions at $N=40$, attains the lowest value for every $N\geq60$, and develops a more visible margin as $N$ increases. Under $\mathtt{deg}=1$, $\mathcal{RSPO}_+$ has the lowest normalized decision loss at every reported $N$, with a particularly visible margin from $N=80$ onward. As $d$ varies, $\mathcal{RSPO}_+$ performs comparably to $\mathcal{SPO}_+$ and $\mathcal{SPO}_+$ with robust decisions at $d\in\{20,40\}$ and attains the lowest normalized decision loss for every $d\geq60$. Across the three experiments, $\mathcal{SPO}_+$ and $\mathcal{SPO}_+$ with robust decisions generally attain lower relative prediction loss, whereas $\mathcal{RSPO}_+$ attains lower or comparable normalized decision loss.

\subsubsection{Value of Learning--Decision Alignment}

To assess the value of learning--decision alignment, we compare $\mathcal{RSPO}_+$ with $\mathcal{SPO}_+$ with robust decisions. Both methods tune their own $(\lambda,\gamma)$, but only $\mathcal{RSPO}_+$ uses the deployed robust decision map to define its training loss. In Figure~\ref{fig:portfolio_combined_effects}, $\mathcal{RSPO}_+$ attains lower or comparable decision loss, with a more visible difference as the sample size and portfolio dimension increase. Thus, the portfolio results reinforce the transportation results and support the value of learning--decision alignment in a continuous allocation problem.

\subsubsection{Value of Gradient-Based Refinement}

Finally, we assess the decision improvement obtained by directly minimizing the empirical $\mathcal{RSPO}$ objective~\eqref{eq:rspo_refine_obj}. We first solve the $\mathcal{RSPO}_+$ empirical risk minimization problem~\eqref{eq:rspo+_ermprob} and use the resulting predictor to initialize Algorithm~\ref{alg:gradient-descent}. We apply the same refinement procedure to an $\mathcal{SPO}_+$ predictor to compare the two starting points and the improvements produced from each. Appendix~\ref{app:gradient_implementation} provides the complete tuning, validation, and implementation details.

\begin{figure}[t]
	\vspace{-3mm}
    \centering
    \includegraphics[width=0.85\linewidth]{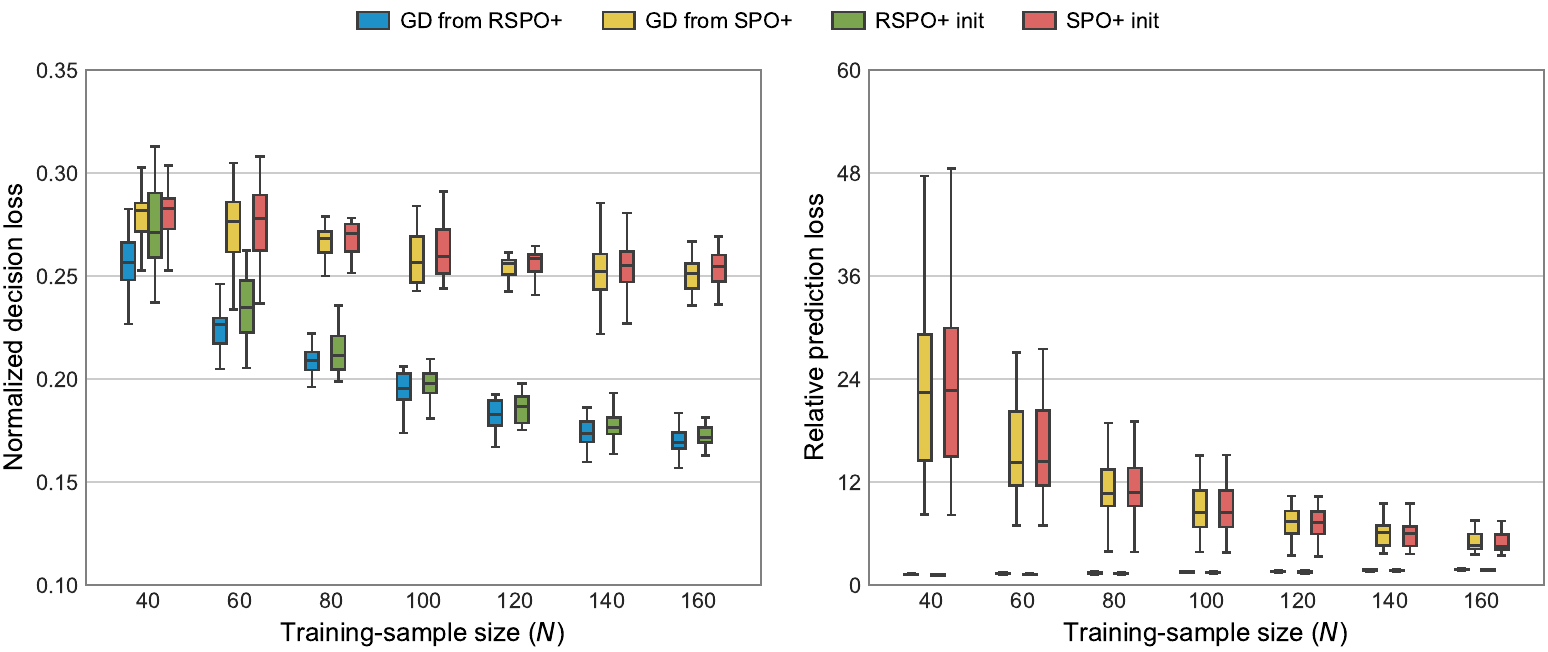}
    
    \caption{Test decision loss (left) and relative prediction loss (right) before and after gradient-based $\mathcal{RSPO}$ refinement as the training-sample size $N$ varies on portfolio optimization instances, with $p=80$, $d=40$, $\mathtt{deg}=4$, and $\tau=1$ fixed.}
    \label{fig:portfolio_rspo_four_methods}\vspace{-6mm}
\end{figure}

\begin{figure}[t]
    \centering
    \includegraphics[width=0.85\linewidth]{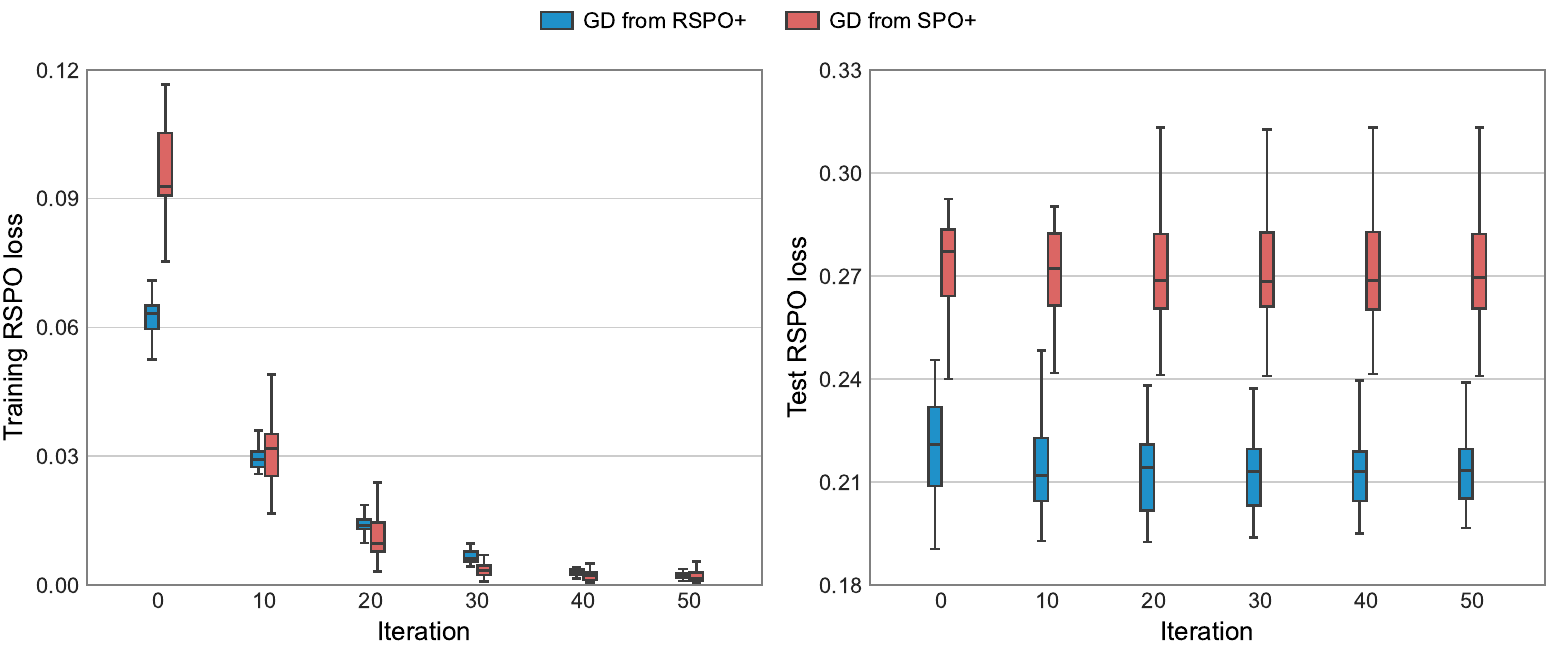}
    
    \caption{Training $\mathcal{RSPO}$ loss (left) and test $\mathcal{RSPO}$ loss (right) across gradient iterations on portfolio optimization instances, with $N=60$, $p=60$, $d=40$, $\mathtt{deg}=4$, and $\tau=1$ fixed.}
    \label{fig:portfolio_sgd_improvement_over_iterations}\vspace{-6mm}
\end{figure}

Figure~\ref{fig:portfolio_rspo_four_methods} shows that refinement from the $\mathcal{RSPO}_+$ predictor lowers test decision loss at every reported $N$, with the largest gains at $N\in\{40,60\}$ and the lowest final loss throughout the sweep. Refinement from the $\mathcal{SPO}_+$ predictor also lowers decision loss, with a higher final loss at every $N$. In both cases, refinement improves decision quality while leaving relative prediction loss broadly unchanged, further separating the decision gain from prediction accuracy. 
Figure~\ref{fig:portfolio_sgd_improvement_over_iterations} shows that the training $\mathcal{RSPO}$ losses decrease rapidly from both initializations and approach nearly the same low level. Most of the test improvement from the $\mathcal{RSPO}_+$ initialization occurs within the first ten iterations and is maintained thereafter. 
Together, these results show that $\mathcal{RSPO}_+$ provides the stronger initialization and that direct optimization of the target $\mathcal{RSPO}$ loss further improves out-of-sample decision quality in the early iterations.

\section{Conclusion}
\label{sec:conclusion}

This paper studies the contextual linear optimization problem with robust decisions. To simultaneously achieve both robustness and learning--decision alignment, we propose an integrated learning and robust optimization (ILRO) framework via $\mathcal{RSPO}$ loss. The resulting learning problem is nonconvex, as is the nominal counterpart ($\mathcal{SPO}$). We address this in two complementary ways: (i) we construct a convex surrogate, the $\mathcal{RSPO}_+$ loss, and we establish its Fisher consistency when the robustness parameter $\gamma$ is below an explicitly computable threshold; (ii) unlike the $\mathcal{SPO}$ loss, the $\mathcal{RSPO}$ loss carries informative first-order information, and hence we develop a gradient descent algorithm to optimize the target loss directly. On the statistical side, we derive finite-sample excess risk bounds for both predictors: (i) a meta generalization bound for the target predictor that separates the effect of the robust decision mapping from the prediction class, and yields $N^{-1/2}$ rates (up to logarithmic factors) for four representative hypothesis classes; and (ii) within the Fisher-consistent regime, target excess risk bounds of order $N^{-1/4}$ for the surrogate predictor, improving to $N^{-1/2}$ under a quadratic growth condition. Numerical experiments on capacitated transportation and risk-constrained portfolio optimization problems compare the framework against predict-then-optimize, decision-focused, and post-hoc robust benchmarks and support these findings, with the clearest gains under limited samples, high-dimensional decisions, and model misspecification. Gradient-based refinement of the $\mathcal{RSPO}_+$ solution improves out-of-sample decision quality further.
 
We outline several potential directions for future research. First, as Proposition~\ref{prop:rspo_piecewise_affine} shows, under certain conditions, the empirical $\mathcal{RSPO}$ objective is continuous and piecewise affine. This structural result invites approximation algorithms built from finitely many cuts, which may offer a globally convergent alternative to the local refinement discussed in Section~\ref{subsec:gradient_refinement}. Second, the selection of the robustness parameter $\gamma$ deserves further investigation. Unlike in the predict-then-robust-optimize framework where $\gamma$ affects only the decision stage, under ILRO it enters the training problem, and hence the decision quality attained at a given $\gamma$ reflects both the robustness it buys and the predictor it induces. Therefore, a judicious choice of $\gamma$ should consider the interaction of the two. Third, we focused on an uncertain objective in this paper. Extending the ILRO framework to the case of both predictive uncertain objective and constraints requires novel modeling methods and analysis techniques, which is a natural next step.

\bibliographystyle{informs2014}
\bibliography{myrefs}

\ECSwitch



\begin{appendices}
\ECHead{\centering Electronic Companion to\\ ``Integrated Learning and Robust Optimization''}

\section{Technical Lemmas}
\label{app:technical_lemmas}

This appendix provides the details of five technical lemmas used throughout our analysis. In particular, Lemma~\ref{lem:oracle_projection} establishes the well-posedness and projection representation of the robust decision map and is invoked in the proofs of Proposition~\ref{prop:optcondition_RSPO+} and Lemma~\ref{lem:rspo_pointwise_excess}; Lemma~\ref{lem:vector_contraction_inequality} is required for Theorem~\ref{thm:rspo_meta_generalization}; Lemma~\ref{lem:path_equivalence} proves the claims made in Remark~\ref{rem:soft-robustness}; and Lemma~\ref{lem:rspo_pointwise_excess} gives the closed-form pointwise target excess risk used in the proof of Theorem~\ref{thm:direct_closure}. Lemma~\ref{lem:local_sc} provides sufficient conditions, in terms of distributional primitives, for the local strong convexity of the surrogate risk invoked in Section~\ref{subsec:rspo+_excess_risk_bounds}. The proofs of the remaining results are provided in Appendix~\ref{app:technical_proofs}.

\bigskip
\noindent\textit{Proof of Lemma \ref{lem:oracle_projection}.}~
For every $\bm y_0\in\mathbb R^d$, completing the square gives
\begin{equation*}
    \bm y_0^\top\bm z+\frac{\gamma}{2}\|\bm z\|_2^2
    =
    \frac{\gamma}{2}
    \left\|\bm z+\frac{\bm y_0}{\gamma}\right\|_2^2
    -\frac{1}{2\gamma}\|\bm y_0\|_2^2.
\end{equation*}
Because $\mathcal Z$ is nonempty and compact and the objective is continuous, a minimizer exists. Since $\gamma>0$, the objective is strongly convex, so the minimizer is unique. The second term in the display is independent of $\bm z$; therefore,
\begin{equation*}
    \bm z_\gamma^\star(\bm y_0)
    =
    \argmin_{\bm z\in\mathcal Z}
    \left\|\bm z+\frac{\bm y_0}{\gamma}\right\|_2^2
    =
    \Pi_{\mathcal Z}\!\left(-\frac{\bm y_0}{\gamma}\right),
\end{equation*}
which proves the projection representation. Thus, for any $\bm y_1,\bm y_2$, the nonexpansiveness of the Euclidean projection gives
\begin{equation*}
\begin{aligned}
    \left\|
        \bm z_{\gamma}^{\star}(\bm y_1)
        -
        \bm z_{\gamma}^{\star}(\bm y_2)
    \right\|_2
    &=
    \left\|
        \Pi_{\mathcal Z}\left(-\frac{\bm y_1}{\gamma}\right)
        -
        \Pi_{\mathcal Z}\left(-\frac{\bm y_2}{\gamma}\right)
    \right\|_2  \\
    &\le
    \left\|
        -\frac{\bm y_1}{\gamma}
        +
        \frac{\bm y_2}{\gamma}
    \right\|_2  \\
    &=
    \frac{1}{\gamma}\|\bm y_1-\bm y_2\|_2.
\end{aligned}
\end{equation*}
Therefore, $\bm z_{\gamma}^{\star}(\cdot)$ is Lipschitz continuous with constant $1/\gamma$. By Rademacher's theorem, every Lipschitz continuous mapping from $\mathbb R^d$ to $\mathbb R^d$ is differentiable almost everywhere. Hence, $\bm z_{\gamma}^{\star}(\cdot)$ is differentiable almost everywhere.
\qed

\bigskip
\noindent\textit{Proof of Lemma \ref{lem:vector_contraction_inequality}.}~
Condition on the sample $\{(\bm x_i,\bm y_i)\}_{i\in[N]}$. For each $i\in[N]$, define $\psi_i:\mathbb R^d\to\mathbb R$ by
\begin{equation*}
    \psi_i(\bm u)
    :=
    \ell_{\mathcal{RSPO}}(\bm u,\bm y_i)
    =
    \bm y_i^\top\bm z^{\star}_{\gamma}(\bm u)-v^{\star}(\bm y_i).
\end{equation*}
By Lemma~\ref{lem:oracle_projection}, each $\psi_i$ is Lipschitz continuous with constant $\|\bm y_i\|_2/\gamma\leq r/\gamma$. By the vector contraction inequality \citep[Corollary~4]{maurer2016vector}, which allows the Lipschitz functions to depend on the index $i$, the Rademacher complexity of a vector-valued function class composed with $L$-Lipschitz scalar functions is bounded, up to a multiplicative factor of $\sqrt{2}L$, by the vector-valued Rademacher complexity of this function class. Since each $\psi_i$ is $(r/\gamma)$-Lipschitz, it follows that
\begin{equation*}
    \mathbb E_{\bm\sigma}
    \left[
    \sup_{\bm g\in\mathcal G}
    \sum_{i\in[N]}
    \sigma_i\,\psi_i(\bm g(\bm x_i))
    \right]
    \leq
    \sqrt{2}\,\frac{r}{\gamma}\,
    \mathbb E_{\bm\sigma}
    \left[
    \sup_{\bm g\in\mathcal G}
    \sum_{i\in[N]}
    \bm\sigma_i^\top\bm g(\bm x_i)
    \right],
\end{equation*}
where on the left-hand side $\sigma_i$ are independent Rademacher random variables and on the right-hand side $\bm\sigma_i\in\{+1,-1\}^d$ are independent Rademacher random vectors. Dividing both sides by $N$ and taking the expectation with respect to the sample yields
\begin{equation*}
    \mathfrak R_1(\mathcal H_{\ell_{\mathcal{RSPO}}})
    \leq
    \frac{\sqrt{2}r}{\gamma}
    \mathfrak R_d(\mathcal G).
\end{equation*}
This completes the proof.
\qed

\bigskip
\begin{lemma}[Solution-path Equivalence]
\label{lem:path_equivalence}
Fix $\hat{\bm y}\in\mathbb R^d$ and let $\mathcal Z\subseteq\mathbb R^d$ be nonempty, compact, and convex. For $\vartheta\geq0$, let $\mathcal S_{\rm ball}(\vartheta):=\argmin_{\bm z\in\mathcal Z} \{\hat{\bm y}^\top\bm z+\vartheta\|\bm z\|_2\}$. For $\gamma>0$, let $\bm z^{\star}_{\gamma}(\hat{\bm y})$ be the unique solution of \ref{ilroopt}. By convention, define $\mathcal Z_0(\hat{\bm y})= \argmin_{\bm z\in \mathcal{Z}}\hat{\bm y}^\top\bm z$ and $\bm z^{\star}_{\infty}(\hat{\bm y}):=\Pi_{\mathcal Z}(\bm 0)=\lim_{\gamma\to\infty}\bm z^{\star}_{\gamma}(\hat{\bm y})$. Then the following hold.
\begin{enumerate}
\item[(i)] For every $\vartheta\geq0$ and every $\bar{\bm z}\in\mathcal S_{\rm ball}(\vartheta)$: if $\bar{\bm z}\neq\bm 0$ and $\vartheta>0$, then $\bar{\bm z}=\bm z^{\star}_{\gamma}(\hat{\bm y})$ with $\gamma=\vartheta/\|\bar{\bm z}\|_2$; if $\vartheta=0$, then $\bar{\bm z}\in\mathcal Z_0(\hat{\bm y})$; if $\bar{\bm z}=\bm 0$, then $\bar{\bm z}=\bm z^{\star}_{\infty}(\hat{\bm y})$.
\item[(ii)] For every $\gamma>0$, $\bm z^{\star}_{\gamma}(\hat{\bm y})\in\mathcal S_{\rm ball}(\vartheta(\gamma))$ with $\vartheta(\gamma):=\gamma\|\bm z^{\star}_{\gamma}(\hat{\bm y})\|_2$.
\item[(iii)] The map $\gamma\mapsto\vartheta(\gamma)$ is nondecreasing on $(0,\infty)$.
\end{enumerate}
\end{lemma}

\noindent\textit{Proof.}~
We first provide an elementary comparison fact. Let $h$ be any function on $\mathcal Z$ and, for $c_1<c_2$, let $\bm x_i$ minimize $\hat{\bm y}^\top\bm x+c_i h(\bm x)$ over $\mathcal Z$. Summing the two optimality inequalities $\hat{\bm y}^\top\bm x_1+c_1h(\bm x_1)\leq\hat{\bm y}^\top\bm x_2+c_1h(\bm x_2)$ and $\hat{\bm y}^\top\bm x_2+c_2h(\bm x_2)\leq\hat{\bm y}^\top\bm x_1+c_2h(\bm x_1)$ yields $(c_2-c_1)(h(\bm x_2)-h(\bm x_1))\leq0$, i.e., $h(\bm x_2)\leq h(\bm x_1)$.

\emph{Part (i).} Let $\bar{\bm z}\in\mathcal S_{\rm ball}(\vartheta)$ with $\bar{\bm z}\neq\bm 0$ and $\vartheta>0$. The objective $\bm z\mapsto\hat{\bm y}^\top\bm z+\vartheta\|\bm z\|_2$ is convex and differentiable at $\bar{\bm z}$, so optimality over the convex set $\mathcal Z$ is equivalent to $\bm 0\in\hat{\bm y}+\vartheta\bar{\bm z}/\|\bar{\bm z}\|_2 +\mathcal N_{\mathcal Z}(\bar{\bm z})$. Setting $\gamma=\vartheta/\|\bar{\bm z}\|_2$, this reads $\bm 0\in\hat{\bm y}+\gamma\bar{\bm z}+\mathcal N_{\mathcal Z}(\bar{\bm z})$, which is the necessary and sufficient optimality condition of \ref{ilroopt}. Since the objective of \ref{ilroopt} is strongly convex for $\gamma>0$, its
minimizer is unique and $\bar{\bm z}=\bm z^{\star}_{\gamma}(\hat{\bm y})$.
If $\vartheta=0$, the ball problem is the nominal problem and $\bar{\bm z}\in\mathcal Z_0(\hat{\bm y})$ by definition. 
Finally, if $\bar{\bm z}=\bm 0\in\mathcal S_{\rm ball}(\vartheta)$, then $\bm 0\in\mathcal Z$ and hence $\Pi_{\mathcal Z}(\bm 0)=\bm 0$. Therefore, $\bar{\bm z}=\bm 0=\Pi_{\mathcal Z}(\bm 0)=\bm z^{\star}_{\infty}(\hat{\bm y})$.

\emph{Part (ii).} Write $\bar{\bm z}=\bm z^{\star}_{\gamma}(\hat{\bm y})$, so that $\bm 0\in\hat{\bm y}+\gamma\bar{\bm z}+\mathcal N_{\mathcal Z}(\bar{\bm z})$. If $\bar{\bm z}\neq\bm 0$, then $\gamma\bar{\bm z}=\vartheta(\gamma)\,\bar{\bm z}/\|\bar{\bm z}\|_2$, so the same inclusion is the (sufficient) optimality condition of the convex ball problem with radius $\vartheta(\gamma)$, and this gives $\bar{\bm z}\in\mathcal S_{\rm ball}(\vartheta(\gamma))$. If $\bar{\bm z}=\bm 0$, the inclusion reads $\bm 0\in\hat{\bm y}+\mathcal N_{\mathcal Z}(\bm 0)$, and this leads to $\bar{\bm z}\in\mathcal S_{\rm ball}(0)$.

\emph{Part (iii).} Let $0<\gamma_1<\gamma_2$ and write $\bm z_i=\bm z^{\star}_{\gamma_i}(\hat{\bm y})$. Applying the comparison fact with $h=\frac12\|\cdot\|_2^2$ gives $\|\bm z_2\|_2\leq\|\bm z_1\|_2$. Suppose, for contradiction, that $\vartheta(\gamma_1)>\vartheta(\gamma_2)$. By part (ii), $\bm z_i\in\mathcal S_{\rm ball}(\vartheta(\gamma_i))$, and the comparison fact with $h=\|\cdot\|_2$ applied to $\vartheta(\gamma_2)<\vartheta(\gamma_1)$ gives $\|\bm z_1\|_2\leq\|\bm z_2\|_2$. Hence $\|\bm z_1\|_2=\|\bm z_2\|_2=:r$, so $\vartheta(\gamma_i)=\gamma_i r$. If $r>0$, then $\vartheta(\gamma_1)=\gamma_1r<\gamma_2r=\vartheta(\gamma_2)$, a contradiction; if $r=0$, then $\vartheta(\gamma_1)=\vartheta(\gamma_2)=0$, again a contradiction. Therefore, $\vartheta(\gamma_1)\leq\vartheta(\gamma_2)$.
\qed

\bigskip
\begin{lemma}\label{lem:rspo_pointwise_excess}
Suppose $\mathcal Z$ is a nonempty compact convex set, $\gamma>0$, $\mathbb E_{\bm{y}\sim\mathbb{P}_{\bm{y}|\bm{x}}}[\|\bm y\|_2]<\infty$, and fix $\bm x$. Define $\Delta_\gamma(\bm c;\bm x):=R(\bm c;\bm x)-\inf_{\bm c'\in\mathbb R^d}R(\bm c';\bm x)$. Recall $\bar{\bm y}=\mathbb E_{\bm{y}\sim\mathbb{P}_{\bm{y}|\bm{x}}}[\bm y]$. Then
\begin{equation*}
\Delta_{\gamma}(\hat{\bm y};\bm x)=\bar{\bm y}^\top\bm z^{\star}_{\gamma}(\hat{\bm y})-v^{\star}(\bar{\bm y}), 
\qquad \text{with } v^{\star}(\bar{\bm y})=\min_{\bm z\in\mathcal Z}\bar{\bm y}^\top\bm z .
\end{equation*}
In particular, $\Delta_{\gamma}(\hat{\bm y};\bm x)\geq0$, with equality if and only if $\bm z^{\star}_{\gamma}(\hat{\bm y})$ is optimal for the nominal problem $\min_{\bm z\in\mathcal Z}\bar{\bm y}^\top\bm z$, that is,
\begin{equation*}
    \Delta_{\gamma}(\hat{\bm y};\bm x)=0 \Leftrightarrow \hat{\bm y}\in\mathcal Y^{\star}_{\mathcal{RSPO}} \Leftrightarrow \bm z^{\star}_{\gamma}(\hat{\bm y})\in \argmin _{\bm z\in\mathcal Z}\bar{\bm y}^\top\bm z.
\end{equation*}
\end{lemma}

\noindent\textit{Proof of Lemma \ref{lem:rspo_pointwise_excess}.}~
Fix $\bm x$ and recall $\ell_{\mathcal{RSPO}}(\hat{\bm y},\bm y)=\bm y^\top\bm z^{\star}_{\gamma}(\hat{\bm y})-v^{\star}(\bm y)$. Taking conditional expectation gives
\begin{equation*}
\mathbb E_{\bm{y}\sim\mathbb{P}_{\bm{y}|\bm{x}}}[\ell_{\mathcal{RSPO}}(\hat{\bm y},\bm y)]=\bar{\bm y}^\top\bm z^{\star}_{\gamma}(\hat{\bm y})-\mathbb E_{\bm{y}\sim\mathbb{P}_{\bm{y}|\bm{x}}}[v^{\star}(\bm y)],
\end{equation*}
since $\bm z^{\star}_{\gamma}(\hat{\bm y})$ does not depend on $\bm y$. The second term is constant in $\hat{\bm y}$, so
\begin{equation*}
\inf_{\bm c\in\mathbb R^d}\mathbb E_{\bm{y}\sim\mathbb{P}_{\bm{y}|\bm{x}}}[\ell_{\mathcal{RSPO}}(\bm c,\bm y)]
=\inf_{\bm c\in\mathbb R^d}\bar{\bm y}^\top\bm z^{\star}_{\gamma}(\bm c)-\mathbb E_{\bm{y}\sim\mathbb{P}_{\bm{y}|\bm{x}}}[v^{\star}(\bm y)].
\end{equation*}
By Lemma~\ref{lem:oracle_projection}, $\bm z^{\star}_{\gamma}(\bm c)=\Pi_{\mathcal Z}(-\bm c/\gamma)$, and since the Euclidean projection maps $\mathbb R^d$ onto $\mathcal Z$ (every $\bm z_0\in\mathcal Z$ satisfies $\Pi_{\mathcal Z}(\bm z_0)=\bm z_0$), the image $\{\bm z^{\star}_{\gamma}(\bm c):\bm c\in\mathbb R^d\}$ equals $\mathcal Z$. Hence $\inf_{\bm c}\bar{\bm y}^\top\bm z^{\star}_{\gamma}(\bm c)=\min_{\bm z\in\mathcal Z}\bar{\bm y}^\top\bm z=v^{\star}(\bar{\bm y})$. Subtracting, the constant term cancels and
\begin{equation*}
\Delta_{\gamma}(\hat{\bm y};\bm x)=\bar{\bm y}^\top\bm z^{\star}_{\gamma}(\hat{\bm y})-v^{\star}(\bar{\bm y}).
\end{equation*}
This is nonnegative because $\bm z^{\star}_{\gamma}(\hat{\bm y})\in\mathcal Z$, and equals zero if and only if $\bm z^{\star}_{\gamma}(\hat{\bm y})$ attains $\min_{\bm z\in\mathcal Z}\bar{\bm y}^\top\bm z$. By the definition, $\Delta_{\gamma}(\hat{\bm y};\bm x)=0$ holds exactly when $\hat{\bm y}$ minimizes the pointwise $\mathcal{RSPO}$ risk, i.e., $\hat{\bm y}\in\mathcal Y^{\star}_{\mathcal{RSPO}}$. Therefore,
\begin{equation*}
    \Delta_{\gamma}(\hat{\bm y};\bm x)=0 \Leftrightarrow \hat{\bm y}\in\mathcal Y^{\star}_{\mathcal{RSPO}} \Leftrightarrow \bm z^{\star}_{\gamma}(\hat{\bm y})\in \argmin _{\bm z\in\mathcal Z}\bar{\bm y}^\top\bm z.
\end{equation*}
This completes the proof.
\qed

\begin{lemma}[Sufficient Conditions for Local Strong Convexity of Surrogate Risk]\label{lem:local_sc}
Suppose $\mathcal Z$ is a nonempty compact convex set, $\gamma>0$, and $\mathbb E_{\bm{y}\sim\mathbb{P}_{\bm{y}|\bm{x}}}[\|\bm y\|_2]<\infty$. Set $a=1$ and fix $\bm x$. Let $\rho>0$ be such that $\mathcal Z$ contains a Euclidean ball of radius strictly larger than $\rho/\gamma$, and define
\begin{equation*}
K_{\rho}:=\left\{\bm w\in\mathbb R^d: \mathcal B(\bm w,\rho/\gamma)\subseteq\mathcal Z\right\},
\qquad
p_0:=\mathbb P_{\bm{y}|\bm{x}}\big(\bm y\in 2\bar{\bm y}+\gamma K_{\rho}\big).
\end{equation*}
Then $R_{+}(\cdot;\bm x)$ is $(p_0/\gamma)$-strongly convex on $\mathcal B(2\bar{\bm y},\rho)$. In particular, $p_0>0$ whenever $\mathbb P_{\bm y|\bm x}$ admits a density that is positive on an open set containing $2\bar{\bm y}+\gamma K_{\rho}$.
\end{lemma}

\noindent\textit{Proof of Lemma \ref{lem:local_sc}.}~
We prove that the gradient of $R_{+}(\cdot;\bm x)$ is $(p_0/\gamma)$-strongly monotone on $\mathcal B(2\bar{\bm y},\rho)$, which is equivalent to the claimed strong convexity.

\emph{Step 1: gradient formula.} By the proof of Proposition~\ref{prop:optcondition_RSPO+} (with $a=1$), $R_{+}(\cdot;\bm x)$ is finite, convex, and differentiable on $\mathbb R^d$ with
\begin{equation*}
\nabla_{\bm y} R_{+}(\hat{\bm y};\bm x)=\mathbb E_{\bm{y}\sim\mathbb{P}_{\bm{y}|\bm{x}}}\left[\bm z^{\star}_{\gamma}(\bm y)\right]-\mathbb E_{\bm{y}\sim\mathbb{P}_{\bm{y}|\bm{x}}}\left[\bm z^{\star}_{\gamma}(\hat{\bm y}-\bm y)\right],
\qquad
\bm z^{\star}_{\gamma}(\bm c)=\Pi_{\mathcal Z}\!\left(-\frac{\bm c}{\gamma}\right).
\end{equation*}
Note that the first expectation does not depend on $\hat{\bm y}$ and will cancel in gradient differences.

\emph{Step 2: the monotonicity pairing as an expectation.} Fix $\hat{\bm y}_1,\hat{\bm y}_2\in\mathcal B(2\bar{\bm y},\rho)$ and set $\bm w_i(\bm y):=(\bm y-\hat{\bm y}_i)/\gamma$, so that $\bm z^{\star}_{\gamma}(\hat{\bm y}_i-\bm y)=\Pi_{\mathcal Z}(\bm w_i(\bm y))$ and $\hat{\bm y}_1-\hat{\bm y}_2=\gamma(\bm w_2(\bm y)-\bm w_1(\bm y))$ for every $\bm y$. By Step 1,
\begin{equation*}
\left\langle\nabla_{\bm y} R_{+}(\hat{\bm y}_1;\bm x)-\nabla_{\bm y} R_{+}(\hat{\bm y}_2;\bm x),\,\hat{\bm y}_1-\hat{\bm y}_2\right\rangle
=\gamma\,\mathbb E_{\bm{y}\sim\mathbb{P}_{\bm{y}|\bm{x}}}\left[\big\langle\Pi_{\mathcal Z}(\bm w_2)-\Pi_{\mathcal Z}(\bm w_1),\,\bm w_2-\bm w_1\big\rangle\right].
\end{equation*}

\emph{Step 3: lower bound on the expectation.} Applying firm nonexpansiveness of the Euclidean projection, $\langle\Pi_{\mathcal Z}(\bm u)-\Pi_{\mathcal Z}(\bm v),\bm u-\bm v\rangle\geq\|\Pi_{\mathcal Z}(\bm u)-\Pi_{\mathcal Z}(\bm v)\|_2^2\geq0$. Hence the expectation in Step 2 is nonnegative. Moreover, on the event $E:=\{\bm y:\bm w_1(\bm y)\in\mathcal Z\text{ and }\bm w_2(\bm y)\in\mathcal Z\}$, the projections act as the identity, so the integrand over $E$ equals $\|\bm w_2-\bm w_1\|_2^2=\|\hat{\bm y}_1-\hat{\bm y}_2\|_2^2/\gamma^2$. Hence,
\begin{equation*}
\left\langle\nabla R_{+}(\hat{\bm y}_1)-\nabla R_{+}(\hat{\bm y}_2),\,\hat{\bm y}_1-\hat{\bm y}_2\right\rangle
\geq\frac{\mathbb P(E)}{\gamma}\,\|\hat{\bm y}_1-\hat{\bm y}_2\|_2^2 .
\end{equation*}

\emph{Step 4: the event $E$ contains $\{\bm y\in2\bar{\bm y}+\gamma K_{\rho}\}$.} For any $\bm y\in2\bar{\bm y}+\gamma K_{\rho}$, we have $(\bm y-2\bar{\bm y})/\gamma\in K_{\rho}$, and since $\|\hat{\bm y}_i-2\bar{\bm y}\|_2\leq\rho$,
\begin{equation*}
\bm w_i(\bm y)=\frac{\bm y-2\bar{\bm y}}{\gamma}+\frac{2\bar{\bm y}-\hat{\bm y}_i}{\gamma}\in \mathcal B\!\left(\frac{\bm y-2\bar{\bm y}}{\gamma},\frac{\rho}{\gamma}\right)\subseteq\mathcal Z,\ i=1,2
\end{equation*}
by the definition of $K_{\rho}$. Hence, $E\supseteq\{\bm y\in2\bar{\bm y}+\gamma K_{\rho}\}$ and $\mathbb P(E)\geq p_0$, so that
\begin{equation*}
\left\langle\nabla R_{+}(\hat{\bm y}_1)-\nabla R_{+}(\hat{\bm y}_2),\,\hat{\bm y}_1-\hat{\bm y}_2\right\rangle
\geq\frac{p_0}{\gamma}\,\|\hat{\bm y}_1-\hat{\bm y}_2\|_2^2
\qquad\forall\,\hat{\bm y}_1,\hat{\bm y}_2\in\mathcal B(2\bar{\bm y},\rho).
\end{equation*}

\emph{Step 5: strong convexity and positivity of $p_0$.} For a differentiable convex function, $(p_0/\gamma)$-strong monotonicity of the gradient on the convex set $\mathcal B(2\bar{\bm y},\rho)$ is equivalent to $(p_0/\gamma)$-strong convexity of $R_{+}$ on that set, which proves the main claim. Finally, if $\mathcal Z$ contains a ball $\mathcal B(\bm z_0,\rho/\gamma+\eta)$ with $\eta>0$, then $\mathcal B(\bm z_0,\eta)\subseteq K_{\rho}$, so $2\bar{\bm y}+\gamma K_{\rho}$ has nonempty interior. If in addition the conditional density of $\bm y$ is positive on an open set containing $2\bar{\bm y}+\gamma K_{\rho}$, then $p_0>0$.
\qed

\newpage
\section{Technical Proofs}
\label{app:technical_proofs}

This appendix presents the proofs of the main results stated in the paper, following the same order as their appearance in the main text.

\noindent\textit{Proof of Proposition \ref{prop:rspo_piecewise_affine}.}~
By Lemma~\ref{lem:oracle_projection}, $\bm z^{\star}_{\gamma}(\hat{\bm y})=\Pi_{\mathcal Z}(-\hat{\bm y}/\gamma)$ is single-valued and $(1/\gamma)$-Lipschitz continuous, so the unregularized empirical $\mathcal{RSPO}$ objective is continuous. It remains to establish the piecewise affine structure. Call $S\subseteq[m]$ \emph{regular} if the rows $\{\bm a_j\}_{j\in S}$ of $\bm A$ are linearly independent (this includes $S=\emptyset$). For a regular $S$, define the affine maps
\begin{equation*}
    \bm z_S(\hat{\bm y})=\frac{1}{\gamma}\big(\bm A_S^\top(\bm A_S\bm A_S^\top)^{-1}\bm A_S-\bm I\big)\hat{\bm y}+\bm A_S^\top(\bm A_S\bm A_S^\top)^{-1}\bm b_S,
    \qquad
    \bm\lambda_S(\hat{\bm y})=(\bm A_S\bm A_S^\top)^{-1}\big(\gamma\bm b_S+\bm A_S\hat{\bm y}\big),
\end{equation*}
with the convention $\bm z_{\emptyset}(\hat{\bm y})=-\hat{\bm y}/\gamma$, and the closed polyhedron
\begin{equation*}
    \mathcal Q_S:=\big\{\hat{\bm y}\in\mathbb R^d:\ \bm a_j^\top\bm z_S(\hat{\bm y})\geq b_j\ \ (j\notin S),\ \ \bm\lambda_S(\hat{\bm y})\geq\bm 0\big\},
\end{equation*}
where all defining conditions are affine in $\hat{\bm y}$. A direct computation gives $\bm A_S\bm z_S(\hat{\bm y})=\bm b_S$ and the stationarity identity $\gamma\bm z_S(\hat{\bm y})+\hat{\bm y}=\bm A_S^\top\bm\lambda_S(\hat{\bm y})$ for every $\hat{\bm y}$. Hence, for $\hat{\bm y}\in\mathcal Q_S$, the pair $(\bm z_S(\hat{\bm y}),\bm\lambda_S(\hat{\bm y}))$ (with multipliers extended by zero off $S$) satisfies primal feasibility, dual feasibility, complementary slackness, and stationarity for the strongly convex problem $\min_{\bm A\bm z\geq\bm b}\hat{\bm y}^\top\bm z+\frac{\gamma}{2}\|\bm z\|_2^2$. Since the constraints are affine, the KKT conditions are sufficient, and by strong convexity the minimizer is unique. Therefore, $\bm z^{\star}_{\gamma}(\hat{\bm y})=\bm z_S(\hat{\bm y})$ when $\hat{\bm y}\in\mathcal Q_S$.

We next show that the sets $\{\mathcal Q_S: S\ \text{is regular}\}$ cover $\mathbb R^d$. Fix $\hat{\bm y}$, write $\bm z^{\star}=\bm z^{\star}_{\gamma} (\hat{\bm y})$, and let $S_{\hat{\bm y}}$ be its active set. By the KKT conditions (which are also necessary here, the constraints being affine), the multiplier set $\{\bm\lambda\in\mathbb R^m_+:\bm A^\top\bm\lambda=\gamma\bm z^{\star}+\hat{\bm y},\ \lambda_j=0\ (j\notin S_{\hat{\bm y}})\}$ is nonempty. It is closed, convex, and contains no line, and hence possesses an extreme point $\hat{\bm\lambda}$. By the standard characterization of extreme points of polyhedra in standard form, the rows $\{\bm a_j\}_{j\in S}$ with $S:=\operatorname{supp}(\hat{\bm\lambda})\subseteq S_{\hat{\bm y}}$ are linearly independent, so $S$ is regular. Since $\bm A_S\bm z^{\star}=\bm b_S$ and $\bm A_S^\top\hat{\bm\lambda}_S=\gamma\bm z^{\star}+\hat{\bm y}$, left-multiplying the latter by $(\bm A_S\bm A_S^\top)^{-1}\bm A_S$ yields $\hat{\bm\lambda}_S=\bm\lambda_S(\hat{\bm y})\geq\bm0$, and substituting back gives $\bm z^{\star}=\bm z_S(\hat{\bm y})$. Feasibility of $\bm z^{\star}$ then shows $\hat{\bm y}\in\mathcal Q_S$. As $S$ ranges over the finitely many regular subsets of $[m]$, the closed polyhedra $\{\mathcal Q_S\}$ therefore cover $\mathbb R^d$, and $\bm z^{\star}_{\gamma}$ is continuous and piecewise affine.

Now fix $i\in[N]$. The map $\bm B\mapsto\bm B\bm x_i$ is linear, and the preimage of a polyhedron under a linear map is a polyhedron, so $\bm B\mapsto\bm y_i^\top\bm z^{\star}_{\gamma}(\bm B\bm x_i)$ is continuous and is affine on each member of the finite polyhedral family $\{\bm B:\bm B\bm x_i\in\mathcal Q_S\}_{S\ \text{is regular}}$. Taking the common refinement of these $N$ families yields a finite polyhedral subdivision of $\mathbb R^{d\times p}$ on each cell of which every summand, and hence their average, is affine. Subtracting the constant $\frac1N\sum_{i\in[N]}v^{\star}(\bm y_i)$ preserves this. Discarding the cells with empty interior leaves a finite family of full-dimensional polyhedra whose union is still $\mathbb R^{d\times p}$, because the discarded cells are contained in the union of the finitely many hyperplanes defining the subdivision. The gradient on the interior of each cell is the constant linear coefficient of the corresponding affine piece.
\qed

\subsection{Proofs of Section \ref{sec:optimization_schemes}}
\noindent\textit{Proof of Proposition \ref{prop:dual_representation}.}~
For fixed $\bm z\in\mathcal Z$, the function
\begin{equation*}
    a\mapsto \bm y^\top\bm z-a\left(\hat{\bm y}^\top\bm z+\frac{\gamma}{2}\|\bm z\|_2^2-v_\gamma^\star(\hat{\bm y})\right)
\end{equation*}
is affine. Hence $q(\cdot)$, as the pointwise maximum of affine functions, is convex on $(0,\infty)$. Let $\bm z_a$ be the unique maximizer in the definition of $q(a)$ for $a>0$. By Danskin's theorem, the derivative of $q$ at $a$ is
\begin{equation*}
    q'(a)=v_\gamma^\star(\hat{\bm y})-\hat{\bm y}^\top\bm z_a-\frac{\gamma}{2}\|\bm z_a\|_2^2.
\end{equation*}
Since $v_\gamma^\star(\hat{\bm y})$ is the minimum of $\hat{\bm y}^\top\bm z+\frac{\gamma}{2}\|\bm z\|_2^2$ over $\mathcal Z$, we have $q'(a)\leq0$. Thus, $q$ is nonincreasing.

It remains to show the limiting representation. Write $h(\bm z):=\hat{\bm y}^\top\bm z+\frac{\gamma}{2}\|\bm z\|_2^2-v_\gamma^\star(\hat{\bm y})\geq0$, where equality holds if and only if $\bm z=\bm z_\gamma^\star(\hat{\bm y})$, because $\gamma>0$ makes the minimizer unique. Then, $q(a)=\max_{\bm z\in\mathcal Z}\{\bm y^\top\bm z-a h(\bm z)\}$. Let the diameter $D_{\mathcal Z}<\infty$ and $M:=\sup_{\bm z\in\mathcal Z}|\bm y^\top\bm z|<\infty$. For any $\varepsilon>0$, compactness and uniqueness imply that
\begin{equation*}
    \eta_\varepsilon:=\inf_{\bm z\in\mathcal Z:\|\bm z-\bm z_\gamma^\star(\hat{\bm y})\|_2\geq\varepsilon}h(\bm z)>0.
\end{equation*}
For $a>2M/\eta_\varepsilon$, any maximizer of $q(a)$ must lie within $\varepsilon$ of $\bm z_\gamma^\star(\hat{\bm y})$; otherwise its objective value is at most $M-a\eta_\varepsilon<-M$, whereas evaluating at $\bm z_\gamma^\star(\hat{\bm y})$ gives $\bm y^\top\bm z_\gamma^\star(\hat{\bm y})\geq -M$. Therefore, every sequence of maximizers converges to $\bm z_\gamma^\star(\hat{\bm y})$ as $a\to\infty$, i.e., $\bm z_a\to \bm z_\gamma^\star(\hat{\bm y})$. Note that $\bm y^\top\bm z_\gamma^\star(\hat{\bm y})\le q(a) \le \bm y^\top\bm z_a$ and by continuity, we have
\begin{equation*}
    \lim_{a\to+\infty} q(a)=\bm y^\top\bm z_\gamma^\star(\hat{\bm y}).
\end{equation*}
Subtracting $v^\star(\bm y)$ gives the stated representation of $\ell_{\mathcal{RSPO}}(\hat{\bm y},\bm y)$.
\qed

\bigskip
\noindent\textit{Proof of Proposition \ref{prop:rspo+_loss_properties}.}~
For part (i), the inequality $\ell_{\mathcal{RSPO}}\leq\ell_{\mathcal{RSPO}_+}$ is the construction preceding the definition in the main text. Nonnegativity holds because $\bm z^{\star}_{\gamma}(\hat{\bm y})\in\mathcal Z$, so that 
\begin{equation*}
    \ell_{\mathcal{RSPO}}(\hat{\bm y},\bm y)=\bm y^\top\bm z^{\star}_{\gamma}(\hat{\bm y})-v^{\star}(\bm y)\geq\min_{\bm z\in\mathcal Z}\bm y^\top\bm z-v^{\star}(\bm y)=0.
\end{equation*}

For part (ii), for each fixed $\bm z\in\mathcal Z$, the term $\bm y^\top\bm z-a\hat{\bm y}^\top\bm z-\frac{a\gamma}{2}\|\bm z\|_2^2$ is affine in $\hat{\bm y}$. Taking the maximum over $\bm z\in\mathcal Z$ preserves convexity, and the remaining terms in $\ell_{\mathcal{RSPO}_+}(\hat{\bm y},\bm y)$ are affine or constant in $\hat{\bm y}$. Hence $\ell_{\mathcal{RSPO}_+}(\cdot,\bm y)$ is convex.

For part (iii), define
\begin{equation*}
    \psi(\hat{\bm y})
    :=
    \max_{\bm z\in\mathcal Z}
    \left\{
    \bm y^\top\bm z-a\hat{\bm y}^\top\bm z-\frac{a\gamma}{2}\|\bm z\|_2^2
    \right\}.
\end{equation*}
The maximizer in $\psi(\hat{\bm y})$ is unique because the objective is strongly concave in $\bm z$, and it is
\begin{equation*}
    \argmax_{\bm z\in\mathcal Z}
    \left\{
    (\bm y-a\hat{\bm y})^\top\bm z-\frac{a\gamma}{2}\|\bm z\|_2^2
    \right\}
    =
    \bm z_\gamma^\star\left(\hat{\bm y}-\frac{1}{a}\bm y\right).
\end{equation*}
By Danskin's theorem,
\begin{equation*}
    \nabla \psi(\hat{\bm y})
    =
    -a\bm z_\gamma^\star\left(\hat{\bm y}-\frac{1}{a}\bm y\right).
\end{equation*}
Adding the gradient of the affine term $a\bm z_\gamma^\star(\bm y)^\top\hat{\bm y}$ gives
\begin{equation*}
    \nabla_{\hat{\bm y}}\ell_{\mathcal{RSPO}_+}(\hat{\bm y},\bm{y})
    =
    a \left(\bm{z}^{\star}_{\gamma}(\bm{y})- \bm{z}^{\star}_{\gamma}\left(\hat{\bm{y}}-\frac{1}{a}\bm{y}\right) \right).
\end{equation*}
The uniqueness of the maximizer and Danskin's theorem establish differentiability at every prediction, so the displayed vector is the gradient everywhere.

For part (iv), by (ii)--(iii) the map $\ell_{\mathcal{RSPO}_+}(\cdot,\bm y)$ is convex and differentiable, and its gradient is $a$ times the difference of two points of $\mathcal Z$, hence uniformly bounded in norm by $aD_{\mathcal Z}$. The mean value theorem then yields the claimed $aD_{\mathcal Z}$-Lipschitz continuity.
\qed

\bigskip
\noindent\textit{Proof of Theorem \ref{thm:rspo_plus_reformulation}.}~
It suffices to reformulate the maximization term in each summand of $\ell_{\mathcal{RSPO}_+}$. Fix a sample $i$ and write $\bm c_i:=\bm y_i-a\bm g_\theta(\bm x_i)$. The optimization term in the surrogate loss is
\begin{equation*}
    \max_{\bm z\in\mathcal Z}
    \left\{
    \bm c_i^\top\bm z-\frac{a\gamma}{2}\|\bm z\|_2^2
    \right\}
    \quad
    \text{with }\mathcal Z=\{\bm z:\bm A\bm z\geq\bm b\}.
\end{equation*}
This is a concave quadratic maximization problem over a nonempty bounded polyhedron. Introduce a multiplier $\bm p_i\geq\bm0$ for the constraint $\bm b-\bm A\bm z\leq\bm0$. The Lagrangian for the maximization problem is
\begin{equation*}
    L_i(\bm z,\bm p_i)
    =
    \bm c_i^\top\bm z-\frac{a\gamma}{2}\|\bm z\|_2^2
    -\bm p_i^\top(\bm b-\bm A\bm z)
    =
    -\bm b^\top\bm p_i
    +(\bm c_i+\bm A^\top\bm p_i)^\top\bm z
    -\frac{a\gamma}{2}\|\bm z\|_2^2 .
\end{equation*}
For fixed $\bm p_i$, maximizing over $\bm z\in\mathbb R^d$ gives
\begin{equation*}
    \sup_{\bm z\in\mathbb R^d}L_i(\bm z,\bm p_i)
    =
    -\bm b^\top\bm p_i
    +
    \frac{1}{2a\gamma}\|\bm c_i+\bm A^\top\bm p_i\|_2^2,
\end{equation*}
where the supremum is attained at $\bm z=\frac{1}{a\gamma}(\bm c_i+\bm A^\top\bm p_i)$. Strong duality and dual attainment hold because the primal is a concave quadratic maximization over a nonempty bounded polyhedron. Indeed, for convex quadratic programs with affine constraints, strong duality requires no Slater condition \citep[see, e.g.,][Section~5.2.3]{boyd2004convex}. Therefore,

\begin{equation*}
    \max_{\bm z\in\mathcal Z}
    \left\{
    \bm c_i^\top\bm z-\frac{a\gamma}{2}\|\bm z\|_2^2
    \right\}
    =
    \min_{\bm p_i\geq\bm0}
    \left\{
    -\bm b^\top\bm p_i
    +
    \frac{1}{2a\gamma}\|\bm y_i-a\bm g_\theta(\bm x_i)+\bm A^\top\bm p_i\|_2^2
    \right\}.
\end{equation*}
Substituting this identity into each $\mathcal{RSPO}_+$ summand and retaining the remaining terms
\begin{equation*}
    a\bm z_\gamma^\star(\bm y_i)^\top\bm g_\theta(\bm x_i)
    +\frac{a\gamma}{2}\|\bm z_\gamma^\star(\bm y_i)\|_2^2
    -v^\star(\bm y_i)
\end{equation*}
yields the stated finite-dimensional reformulation. Since the inner minimum over $\bm p_i\in\mathbb R^m_+$ is attained for every $\theta$, minimizing the reformulation jointly over $(\theta,\{\bm p_i\})$ is equivalent to minimizing the partially minimized objective (the $\mathcal{RSPO}_+$ objective) over $\theta$ alone. The two problems therefore share the same optimal value, and $\theta^{\star}$ is optimal for Problem~\eqref{eq:rspo+_ermprob} if and only if $(\theta^{\star},\{\bm p_i^{\star}\})$ is optimal for the reformulation, with $\{\bm p_i^{\star}\}$ the attaining multipliers. The regularizer $\lambda\Omega(\bm g_\theta)$ and the constraint $\theta\in\Theta$ are unchanged. \qed

\bigskip
\noindent\textit{Proof of Proposition \ref{prop:minRSPO}.}~
Throughout this proof, $\bar{\bm z}:=\bm z^{\star}(\bar{\bm y})$ denotes the fixed vector given by the unique nominal solution at the conditional mean $\bar{\bm y}$.

\emph{Step 1: pointwise risk and lower bound.} For any $\hat{\bm y}$, the pointwise $\mathcal{RSPO}$ risk at the given $\bm x$ satisfies
\begin{equation*}
    \begin{aligned}
        \mathbb{E}_{\bm{y}\sim\mathbb{P}_{\bm{y}|\bm{x}}}\left[ \ell_{\mathcal{RSPO}}(\hat{\bm{y}},\bm{y})\right]
        =&\mathbb{E}_{\bm{y}\sim\mathbb{P}_{\bm{y}|\bm{x}}}\left[ \bm{y}^\top\bm{z}^{\star}_{\gamma}(\hat{\bm{y}})-v^{\star}(\bm{y})\right]\\
        =&\bar{\bm{y}}^\top\bm{z}^{\star}_{\gamma}(\hat{\bm{y}})-\mathbb{E}_{\bm{y}\sim\mathbb{P}_{\bm{y}|\bm{x}}}\left[ v^{\star}(\bm{y})\right]\\
        \geq& \bar{\bm{y}}^\top\bar{\bm z}-\mathbb{E}_{\bm{y}\sim\mathbb{P}_{\bm{y}|\bm{x}}}\left[ v^{\star}(\bm{y})\right],
    \end{aligned}
\end{equation*}
where the second equality holds because $\bm z^{\star}_{\gamma}(\hat{\bm y})$ is independent of $\bm y$, and the inequality holds because $\bm z^{\star}_{\gamma}(\hat{\bm y})\in\mathcal Z$ while $\bar{\bm z}$ minimizes $\bar{\bm y}^\top\bm z$ over $\mathcal Z$.

\emph{Step 2: the lower bound is attained.} Consider the candidate prediction $\hat{\bm y}_0:=-\gamma\bar{\bm z}$. Completing the square, the robust problem at $\hat{\bm y}_0$ becomes
\begin{equation*}
    \min_{\bm z\in\mathcal Z}\left\{\hat{\bm y}_0^\top\bm z+\frac{\gamma}{2}\|\bm z\|_2^2\right\}
    =\min_{\bm z\in\mathcal Z}\frac{\gamma}{2}\left\|\bm z-\bar{\bm z}\right\|_2^2-\frac{\gamma}{2}\|\bar{\bm z}\|_2^2,
\end{equation*}
i.e., it finds the point of $\mathcal Z$ closest to $\bar{\bm z}$. Since $\bar{\bm z}\in\mathcal Z$, the unique minimizer is $\bar{\bm z}$ itself, so $\bm z^{\star}_{\gamma}(\hat{\bm y}_0)=\bar{\bm z}$ and the risk at $\hat{\bm y}_0$ equals the lower bound of Step 1. Hence the infimum of the pointwise risk coincides with that bound, and because $\bar{\bm z}$ is the unique minimizer of $\bar{\bm y}^\top\bm z$ over $\mathcal Z$, equality in Step 1 holds precisely when $\bm z^{\star}_{\gamma}(\hat{\bm y})=\bar{\bm z}$. Therefore,
\begin{equation*}
    \hat{\bm y}\in\mathcal{Y}^{\star}_{\mathcal{RSPO}}
    \quad\Longleftrightarrow\quad
    \bm z^{\star}_{\gamma}(\hat{\bm y})=\bar{\bm z}.
\end{equation*}

\emph{Step 3: normal-cone characterization.} The condition $\bm z^{\star}_{\gamma}(\hat{\bm y})=\bar{\bm z}$ states that $\bar{\bm z}$ minimizes the convex differentiable function $\bm z\mapsto\hat{\bm y}^\top\bm z+\frac{\gamma}{2}\|\bm z\|_2^2$ over the closed convex set $\mathcal Z$. By the first-order optimality condition for convex problems, this holds if and only if
\begin{equation*}
    -\big(\hat{\bm y}+\gamma\bar{\bm z}\big)\in {\mathcal{N}}_{\mathcal{Z}}(\bar{\bm z})
    \quad\Longleftrightarrow\quad
    \hat{\bm y}\in -{\mathcal{N}}_{\mathcal{Z}}(\bar{\bm z})-\gamma\bar{\bm z}.
\end{equation*}
Combined with the equivalence of Step 2, this proves the general representation \eqref{eq:RSPO_normalcone}:
\begin{equation*}
    \mathcal{Y}^{\star}_{\mathcal{RSPO}}
    =-{\mathcal{N}}_{\mathcal{Z}}\big(\bm z^{\star}(\bar{\bm y})\big)-\gamma\,\bm z^{\star}(\bar{\bm y}).
\end{equation*}

\emph{Step 4: polyhedral specialization.} Suppose now that $\mathcal Z=\{\bm z\in\mathbb R^d:\bm A\bm z\geq\bm b\}$ is a bounded polyhedron. By Step 3, it suffices to verify $-{\mathcal{N}}_{\mathcal{Z}}(\bar{\bm z})=\{\bm A^\top\bm\lambda:\bm\lambda\in\bm\Lambda\}$, which we do via the KKT conditions. The condition $\bm z^{\star}_{\gamma}(\hat{\bm y})=\bar{\bm z}$ holds if and only if the KKT system
\begin{equation*}
    \begin{aligned}
        \hat{\bm{y}}+\gamma\bar{\bm z}-\bm{A}^\top\bm{\lambda}&=\bm{0},\\
        \bm A\bar{\bm z}&\geq \bm b,\qquad
        \bm{\lambda}\geq\bm{0},\\
        \lambda_i(\bm a_i^\top\bar{\bm z}-b_i)&=0,\quad i=1,\ldots,m,
    \end{aligned}
\end{equation*}
admits a multiplier $\bm\lambda$. Feasibility $\bm A\bar{\bm z}\geq\bm b$ holds automatically. Complementary slackness forces $\lambda_i=0$ for every constraint inactive at $\bar{\bm z}$, i.e., $\bm\lambda\in\bm\Lambda$; conversely, for any $\bm\lambda\in\bm\Lambda$ complementary slackness holds by the definition of $\bm\Lambda$. Hence the KKT system is solvable if and only if the stationarity equation can be met with some $\bm\lambda\in\bm\Lambda$, that is,
\begin{equation*}
    \bm z^{\star}_{\gamma}(\hat{\bm y})=\bar{\bm z}
    \quad\Longleftrightarrow\quad
    \hat{\bm y}=\bm A^\top\bm\lambda-\gamma\bar{\bm z}\ \text{ for some }\bm\lambda\in\bm\Lambda .
\end{equation*}
(Note that $\hat{\bm y}_0$ of Step 2 corresponds to the choice $\bm\lambda=\bm0\in\bm\Lambda$.)

Combining the two displayed equivalences,
\begin{equation*}
    \mathcal{Y}^{\star}_{\mathcal{RSPO}}
    =
    \left\{\bm{A}^\top\bm{\lambda}-\gamma\bar{\bm z}:\bm{\lambda}\in\bm{\Lambda} \right\},
\end{equation*}
and comparing with Step 3 shows $-{\mathcal{N}}_{\mathcal{Z}}(\bar{\bm z})=\{\bm A^\top\bm\lambda:\bm\lambda\in\bm\Lambda\}$, which is precisely the polyhedral description asserted in the proposition with $\bar{\bm z}=\bm z^{\star}(\bar{\bm y})$.
\qed

\bigskip
\noindent\textit{Proof of Proposition \ref{prop:optcondition_RSPO+}.}~
\emph{Step 1: finiteness.} Recall $R_{\mathcal Z}=\max_{\bm z\in\mathcal Z}\|\bm z\|_2<\infty$. For any fixed $\hat{\bm y}$ and any $\bm y$, every inner product appearing in $\ell_{\mathcal{RSPO}_+}(\hat{\bm y},\bm y)$ is bounded via Cauchy--Schwarz, giving
\begin{equation*}
0\leq\ell_{\mathcal{RSPO}_+}(\hat{\bm y},\bm y)\leq \left(2\|\bm y\|_2+2a\|\hat{\bm y}\|_2\right)R_{\mathcal Z}+\frac{a\gamma}{2}R_{\mathcal Z}^{2},
\end{equation*}
where nonnegativity follows from $\ell_{\mathcal{RSPO}_+}\geq\ell_{\mathcal{RSPO}}\geq0$ by Proposition~\ref{prop:rspo+_loss_properties}(i). Since $\mathbb E_{\bm y\sim\mathbb P_{\bm y|\bm x}}[\|\bm y\|_2]<\infty$, the pointwise surrogate risk $R_{+}(\hat{\bm y};\bm x):=\mathbb E_{\bm y\sim\mathbb P_{\bm y|\bm x}}[\ell_{\mathcal{RSPO}_+}(\hat{\bm y},\bm y)]$ is finite for every $\hat{\bm y}$.

\emph{Step 2: differentiation under the expectation.} By Proposition~\ref{prop:rspo+_loss_properties}, $\ell_{\mathcal{RSPO}_+}(\cdot,\bm y)$ is differentiable everywhere with $\nabla_{\hat{\bm y}}\ell_{\mathcal{RSPO}_+}(\hat{\bm y},\bm y)=a(\bm z^{\star}_{\gamma}(\bm y)-\bm z^{\star}_{\gamma}(\hat{\bm y}-\frac1a\bm y))$. The gradient is continuous in $\bm y$ because $\bm z^{\star}_{\gamma}(\cdot)=\Pi_{\mathcal Z}(-\cdot/\gamma)$ is $1/\gamma$-Lipschitz by Lemma~\ref{lem:oracle_projection}, and uniformly bounded in norm by $aD_{\mathcal Z}$. Since the difference quotients of $\ell_{\mathcal{RSPO}_+}(\cdot,\bm y)$ are bounded by $aD_{\mathcal Z}$ uniformly in $\bm y$, the dominated convergence theorem yields that $R_{+}$ is differentiable with
\begin{equation*}
\nabla_{\bm y} R_{+}(\hat{\bm y};\bm x)=\mathbb E_{\bm y\sim\mathbb P_{\bm y|\bm x} }\left[\nabla_{\hat{\bm y}}\,\ell_{\mathcal{RSPO}_+}(\hat{\bm y},\bm y)\right]
=a\,\mathbb E_{\bm y\sim\mathbb P_{\bm y|\bm x}}\left[\bm z^{\star}_{\gamma}(\bm y)\right]-a\,\mathbb E_{\bm y\sim\mathbb P_{\bm y|\bm x}}\left[\bm z^{\star}_{\gamma}\Big(\hat{\bm y}-\frac1a\bm y\Big)\right].
\end{equation*}

\emph{Step 3: optimality condition.} $R_{+}$ is convex (Proposition~\ref{prop:rspo+_loss_properties}) and differentiable, so $\hat{\bm y}\in\mathcal Y^{\star}_{\mathcal{RSPO}_+}$ if and only if $\nabla_{\bm y} R_{+}(\hat{\bm y};\bm x)=\bm 0$, i.e.
\begin{equation*}
\mathbb{E}_{\bm{y}\sim \mathbb P_{\bm y|\bm x}}\left[\bm{z}^{\star}_{\gamma}(\bm{y})\right]=\mathbb{E}_{\bm{y}\sim \mathbb P_{\bm y|\bm x}}\left[\bm{z}^{\star}_{\gamma}\Big(\hat{\bm{y}}-\frac{1}{a}\bm{y}\Big)\right].
\end{equation*}
The proof is complete.
\qed

\bigskip
\noindent\textit{Proof of Corollary \ref{cor:RSPO+a1}.}~
All expectations below are finite because $\mathcal Z$ is compact and $\mathbb E[\|\bm y\|_2]<\infty$; see Step~1 of the proof of Proposition~\ref{prop:optcondition_RSPO+}. Consider any vector $\Delta\neq\bm{0}$. We compute the difference $\mathbb{E}_{\bm{y}\sim \mathbb P_{\bm y|\bm x}}[\ell_{\mathcal{RSPO}_+}(2\bar{\bm{y}}+\Delta,\bm{y})]-\mathbb{E}_{\bm{y}\sim \mathbb P_{\bm y|\bm x}}[\ell_{\mathcal{RSPO}_+}(2\bar{\bm{y}},\bm{y})]$ as follows:

\begin{equation*}
    \begin{aligned}
        &\mathbb{E}_{\bm{y}\sim \mathbb P_{\bm y|\bm x}}[\ell_{\mathcal{RSPO}_+}(2\bar{\bm{y}}+\Delta,\bm{y})]-\mathbb{E}_{\bm{y}\sim \mathbb P_{\bm y|\bm x}}[\ell_{\mathcal{RSPO}_+}(2\bar{\bm{y}},\bm{y})]\\
        =&\mathbb{E}_{\bm{y}\sim \mathbb P_{\bm y|\bm x}}\left[\max
        _{\bm{z}\in\mathcal{Z}}\left\{\bm{y}^\top\bm{z}-(2\bar{\bm{y}}+\Delta)^\top\bm{z}
        -\frac{\gamma}{2}\|\bm{z}\|^2_2\right\}+\{\bm{z}^{\star}_{\gamma}(\bm{y})^\top(2\bar{\bm{y}}+\Delta)+\frac{\gamma}{2}\|\bm{z}^{\star}_{\gamma}(\bm{y})\|^2_2\}-v^{\star}(\bm{y})\right]\\
        &-\mathbb{E}_{\bm{y}\sim \mathbb P_{\bm y|\bm x}}\left[\max
        _{\bm{z}\in\mathcal{Z}}\left\{\bm{y}^\top\bm{z}-2\bar{\bm{y}}^\top\bm{z}
        -\frac{\gamma}{2}\|\bm{z}\|^2_2\right\}+\{2\bm{z}^{\star}_{\gamma}(\bm{y})^\top\bar{\bm{y}}+\frac{\gamma}{2}\|\bm{z}^{\star}_{\gamma}(\bm{y})\|^2_2\}-v^{\star}(\bm{y})\right]\\
        =&\mathbb{E}_{\bm{y}\sim \mathbb P_{\bm y|\bm x}}\left[\max
        _{\bm{z}\in\mathcal{Z}}\left\{\bm{y}^\top\bm{z}-(2\bar{\bm{y}}+\Delta)^\top\bm{z}
        -\frac{\gamma}{2}\|\bm{z}\|^2_2\right\}-\max
        _{\bm{z}\in\mathcal{Z}}\left\{\bm{y}^\top\bm{z}-2\bar{\bm{y}}^\top\bm{z}-\frac{\gamma}{2}\|\bm{z}\|^2_2\right\}+\bm{z}^{\star}_{\gamma}(\bm{y})^\top\Delta\right]\\
        =&\mathbb{E}_{\bm{y}\sim \mathbb P_{\bm y|\bm x}}\left[-v^{\star}_{\gamma}(2\bar{\bm{y}}+\Delta-\bm{y})-(\bm{y}-2\bar{\bm{y}})^\top \bm{z}^{\star}_{\gamma}(2\bar{\bm{y}}-\bm{y})+\frac{\gamma}{2}\|\bm{z}^{\star}_{\gamma}(2\bar{\bm{y}}-\bm{y})\|^2_2+\bm{z}^{\star}_{\gamma}(\bm{y})^\top\Delta\right]\\
        =&\mathbb{E}_{\bm{y}\sim \mathbb P_{\bm y|\bm x}}\left[-v^{\star}_{\gamma}(2\bar{\bm{y}}+\Delta-\bm{y})-(\bm{y}-2\bar{\bm{y}}-\Delta)^\top \bm{z}^{\star}_{\gamma}(2\bar{\bm{y}}-\bm{y})+\frac{\gamma}{2}\|\bm{z}^{\star}_{\gamma}(2\bar{\bm{y}}-\bm{y})\|^2_2\right]\\
        =&\mathbb{E}_{\bm{y}\sim \mathbb P_{\bm y|\bm x}}\left[-v^{\star}_{\gamma}(2\bar{\bm{y}}+\Delta-\bm{y})+(2\bar{\bm{y}}+\Delta-\bm{y})^\top \bm{z}^{\star}_{\gamma}(2\bar{\bm{y}}-\bm{y})+\frac{\gamma}{2}\|\bm{z}^{\star}_{\gamma}(2\bar{\bm{y}}-\bm{y})\|^2_2\right],
    \end{aligned}
\end{equation*}
where the fourth equality follows from the fact that the conditional distribution of $\bm{y}$ is symmetric about its conditional mean, which implies that $\mathbb{E}_{\bm{y}\sim \mathbb P_{\bm y|\bm x}}[\bm{z}^{\star}_{\gamma}(2\bar{\bm{y}}-\bm{y})]=\mathbb{E}_{\bm{y}\sim \mathbb P_{\bm y|\bm x}}[\bm{z}^{\star}_{\gamma}(\bm{y})]$.

Since $v^{\star}_{\gamma}(2\bar{\bm{y}}+\Delta-\bm{y})$ is the minimum value for the optimization problem $\min_{\bm{z}\in\mathcal{Z}}(2\bar{\bm{y}}+\Delta-\bm{y})^\top\bm{z}+\frac{\gamma}{2}\|\bm{z}\|^2_2$, we have
\begin{equation*}
    -v^{\star}_{\gamma}(2\bar{\bm{y}}+\Delta-\bm{y})+(2\bar{\bm{y}}+\Delta-\bm{y})^\top \bm{z}^{\star}_{\gamma}(2\bar{\bm{y}}-\bm{y})+\frac{\gamma}{2}\|\bm{z}^{\star}_{\gamma}(2\bar{\bm{y}}-\bm{y})\|^2_2\geq 0.
\end{equation*}
This proves that $2\bar{\bm y}\in\mathcal Y^\star_{\mathcal{RSPO}_+}$. Suppose now that $\operatorname{int}\mathcal Z\neq\emptyset$.
We now show that the expectation of the nonnegative integrand is strictly positive, by exhibiting an explicit open event on which the integrand is strictly positive. Define the open set
\begin{equation*}
    A:=2\bar{\bm{y}}+\gamma\,\mathrm{int}\,\mathcal{Z}
    =\left\{\bm{y}\in\mathbb{R}^d:\ -\frac{2\bar{\bm{y}}-\bm{y}}{\gamma}\in\mathrm{int}\,\mathcal{Z}\right\}.
\end{equation*}
We claim that
\begin{equation*}
    \bm{z}^{\star}_{\gamma}(2\bar{\bm{y}}-\bm{y})\neq\bm{z}^{\star}_{\gamma}(2\bar{\bm{y}}-\bm{y}+\Delta),\qquad \forall\bm{y}\in A.
\end{equation*}
Indeed, write $\bm{w}:=-(2\bar{\bm{y}}-\bm{y})/\gamma$, so that $\bm{w}\in\mathrm{int}\,\mathcal{Z}$ for $\bm{y}\in A$ and, by the projection representation, $\bm{z}^{\star}_{\gamma}(2\bar{\bm{y}}-\bm{y})=\Pi_{\mathcal{Z}}(\bm{w})=\bm{w}$, while $\bm{z}^{\star}_{\gamma}(2\bar{\bm{y}}-\bm{y}+\Delta)=\Pi_{\mathcal{Z}}(\bm{w}-\Delta/\gamma)$. Suppose for contradiction that $\Pi_{\mathcal{Z}}(\bm{w}-\Delta/\gamma)=\bm{w}$. If $\bm{w}-\Delta/\gamma\in\mathcal{Z}$, then $\Pi_{\mathcal{Z}}(\bm{w}-\Delta/\gamma)=\bm{w}-\Delta/\gamma\neq\bm{w}$ since $\Delta\neq\bm{0}$, a contradiction. If $\bm{w}-\Delta/\gamma\notin\mathcal{Z}$, then its projection lies on the boundary of $\mathcal{Z}$, whereas $\bm{w}\in\mathrm{int}\,\mathcal{Z}$, again a contradiction. This proves the claim.

For $\bm{y}\in A$, the point $\bm{z}^{\star}_{\gamma}(2\bar{\bm{y}}-\bm{y})$ is therefore not the unique minimizer (uniqueness holds since $\gamma>0$) of $\min_{\bm{z}\in\mathcal{Z}}(2\bar{\bm{y}}+\Delta-\bm{y})^\top\bm{z}+\frac{\gamma}{2}\|\bm{z}\|^2_2$, so the integrand is strictly positive on $A$. Since $\mathcal{Z}$ has nonempty interior, $A$ is a nonempty open set, and since the density of $\bm{y}$ is positive on an open set containing $2\bar{\bm{y}}+\gamma\mathcal{Z}\supseteq A$, we have $\mathbb{P}_{\bm y|\bm x}(\bm{y}\in A)>0$.
Therefore,
\begin{equation*}
    \mathbb{E}_{\bm{y}\sim \mathbb P_{\bm y|\bm x}}\left[-v^{\star}_{\gamma}(2\bar{\bm{y}}+\Delta-\bm{y})+(2\bar{\bm{y}}+\Delta-\bm{y})^\top \bm{z}^{\star}_{\gamma}(2\bar{\bm{y}}-\bm{y})+\frac{\gamma}{2}\|\bm{z}^{\star}_{\gamma}(2\bar{\bm{y}}-\bm{y})\|^2_2\right]>0.
\end{equation*}
Thus $2\bar{\bm{y}}+\Delta$ is not a minimizer of the $\mathcal{RSPO}_+$ risk, and $2\bar{\bm{y}}$ must be the unique minimizer.
\qed

\bigskip
\noindent\textit{Proof of Theorem \ref{thm:Fisher_consistency}.}~
Throughout, write $\bar{\bm z}:=\bm z^{\star}(\bar{\bm y})$ and $\mathcal C:=-{\mathcal{N}}_{\mathcal{Z}}(\bar{\bm z})$, which is a closed convex cone because $\mathcal Z$ is closed and convex. We may suppress the dependence of $\bm x$ for concise exposition. Define the feasibility set
\begin{equation*}
    \Gamma:=\left\{\gamma\geq0:
    2\bar{\bm y}+\gamma\bar{\bm z}
    \in \mathcal C\right\},
\end{equation*}
so that, by the general representation \eqref{eq:RSPO_normalcone} of Proposition~\ref{prop:minRSPO}, for every $\gamma>0$, membership $\gamma\in\Gamma$ is equivalent to $2\bar{\bm y}\in\mathcal Y^\star_{\mathcal{RSPO}}$, and $\bar{\gamma}(\bm x)=\sup\Gamma$ coincides with \eqref{eq:gamma}.

\emph{Part (i).} Since $\mathcal C$ is convex and the mapping $\gamma\mapsto 2\bar{\bm{y}}+\gamma\bar{\bm z}$ is affine, $\Gamma$ is convex as the preimage of a convex set under an affine mapping, and therefore is an interval. Since $\mathcal C$ is closed and the mapping is continuous, $\Gamma$ is closed. Moreover, by the first-order optimal condition, the optimality of $\bar{\bm z}$ for the nominal problem $\min_{\bm z\in\mathcal Z}\bar{\bm y}^\top\bm z$ means $-\bar{\bm y}\in {\mathcal{N}}_{\mathcal{Z}}(\bar{\bm z})$, i.e., $\bar{\bm y}\in \mathcal C$, and since $\mathcal C$ is a cone, $2\bar{\bm y}\in \mathcal C$. Thus, $0\in\Gamma$. Consequently, $\bar{\gamma}=\sup\Gamma\in[0,+\infty]$ is attained whenever it is finite. If $\bar{\gamma} <\infty$, then $\Gamma=[0,\bar{\gamma}]$; otherwise, we have $\Gamma=[0,+\infty)$. It follows that for $0<\gamma\leq\bar{\gamma}$, we have $2\bar{\bm{y}}\in\mathcal{Y}^{\star}_{\mathcal{RSPO}}$. Combining this with $\mathcal{Y}^{\star}_{\mathcal{RSPO}_+}=\{2\bar{\bm{y}}\}$ from Corollary~\ref{cor:RSPO+a1}, whose uniqueness part uses $\operatorname{int}\mathcal Z\neq\emptyset$, we have $\mathcal{Y}^{\star}_{\mathcal{RSPO}_+}(\bm x)\subseteq \mathcal{Y}^{\star}_{\mathcal{RSPO}}(\bm x)$ for fixed $\bm x$. If $0<\gamma\leq\bar{\gamma}(\bm x)$ for $\mathbb P_{\bm x}$-almost every $\bm x$, the $\mathcal{RSPO}_+$ loss is Fisher consistent with the $\mathcal{RSPO}$ loss by definition. When $\mathcal Z=\{\bm z:\bm A\bm z\geq\bm b\}$ is a bounded polyhedron, Proposition~\ref{prop:minRSPO} gives $\mathcal C=\{\bm A^\top\bm\lambda:\bm\lambda\in\bm\Lambda\}$, so $\bar\gamma$ is the value of the linear program
\begin{equation*}
    \sup\Big\{\gamma\geq0:\ 2\bar{\bm{y}}=\bm{A}^\top\bm{\lambda}-\gamma\bm{z}^{\star}(\bar{\bm{y}}),\ \bm{\lambda}\in\bm{\Lambda}\Big\}.
\end{equation*}

\emph{Part (ii).} Suppose that $\mathcal Z$ is a bounded polyhedron. Since $\operatorname{int}\mathcal Z\neq\emptyset$ by assumption, $\mathcal Z$ is full-dimensional and hence has at least two vertices. If $\bar{\bm z}=\bm 0$, then $2\bar{\bm y}+\gamma\bar{\bm z}=2\bar{\bm y}\in \mathcal C$ for every $\gamma\geq0$ by Part (i), so $\Gamma=[0,+\infty)$ and $\bar\gamma=+\infty$, as stipulated. Assume henceforth that $\bar{\bm z}\neq\bm 0$.

\emph{Step 1: uniqueness places $\bar{\bm y}$ in the interior of $\mathcal C$.} Let $V(\mathcal Z)$ denote the finite vertex set of the bounded polyhedron $\mathcal Z$. Since a linear objective attains its minimum over $\mathcal Z$ at a vertex, uniqueness of $\bar{\bm z}$ implies $\bar{\bm z}\in V(\mathcal Z)$ and
\begin{equation*}
    \epsilon_0:=\min_{\bm v\in V(\mathcal Z)\setminus\{\bar{\bm z}\}}\bar{\bm y}^\top(\bm v-\bar{\bm z})>0 .
\end{equation*}
Recall that $D_{\mathcal Z}$ is the diameter of $\mathcal Z$ and set $\epsilon:=\epsilon_0/D_{\mathcal Z}$. For any $\bm c$ with $\|\bm c-\bar{\bm y}\|_2<\epsilon$ and any vertex $\bm v\neq\bar{\bm z}$, the Cauchy--Schwarz inequality gives
\begin{equation*}
    \bm c^\top(\bm v-\bar{\bm z})
    =\bar{\bm y}^\top(\bm v-\bar{\bm z})+(\bm c-\bar{\bm y})^\top(\bm v-\bar{\bm z})
    \geq \epsilon_0-\|\bm c-\bar{\bm y}\|_2\,D_{\mathcal Z}>0 ,
\end{equation*}
so $\bar{\bm z}$ is optimal (indeed uniquely optimal) for the cost vector $\bm c$, i.e., $\bm c\in \mathcal C$. Hence the open ball $\mathcal B_{\rm o}(\bar{\bm y},\epsilon)\subseteq \mathcal C$, so $\bar{\bm y}\in\operatorname{int} \mathcal C$ and $\delta:=\dist(\bar{\bm y},\partial \mathcal C)\geq\epsilon>0$. Moreover, since $\mathcal C$ is closed and convex, the closed ball $\mathcal B(\bar{\bm y},\delta)\subseteq \mathcal C$.

\emph{Step 2: membership under the translation.} Let $0\leq\gamma\leq 2\delta/\|\bar{\bm z}\|_2$. Then $\|(\gamma/2)\bar{\bm z}\|_2\leq\delta$, so $\bar{\bm y}+(\gamma/2)\bar{\bm z}\in \mathcal B(\bar{\bm y},\delta)\subseteq \mathcal C$, and since $\mathcal C$ is a cone,
\begin{equation*}
    2\bar{\bm y}+\gamma\bar{\bm z}=2\Big(\bar{\bm y}+\frac{\gamma}{2}\bar{\bm z}\Big)\in \mathcal C .
\end{equation*}
Hence $\gamma\in\Gamma$, and therefore $\bar{\gamma}(\bm x)\geq 2\delta/\|\bar{\bm z}\|_2=2\dist(\bar{\bm y},\partial \mathcal C)/\|\bm z^{\star}(\bar{\bm y})\|_2>0$. In particular the regime $(0,\bar\gamma(\bm x)]$ in Part (i) is nonempty.
\qed

\bigskip
\noindent\textit{Proof of Example \ref{ex:fisher_consistency}.}~
The nominal problem at $\bar{\bm y}=(1,3/2)^\top$ has the unique solution $\bm z^\star(\bar{\bm y})=(-1,-2)^\top$. At this solution, the active inequalities in the representation $\bm A\bm z\geq\bm b$ are $-z_1+z_2\geq-1$ and $z_1\geq-1$, whose normal vectors are $(-1,1)^\top$ and $(1,0)^\top$, respectively. Hence, by Proposition~\ref{prop:minRSPO}, the condition $2\bar{\bm y}\in\mathcal Y^\star_{\mathcal{RSPO}}$ is equivalent to the existence of $\lambda_1,\lambda_2\geq0$ such that
\begin{equation*}
    (2,3)^\top
    =
    \lambda_1(-1,1)^\top+\lambda_2(1,0)^\top-\gamma(-1,-2)^\top.
\end{equation*}
Equating coordinates gives
\begin{equation*}
    \lambda_1=3-2\gamma,
    \qquad
    \lambda_2=5-3\gamma.
\end{equation*}
Both multipliers are nonnegative if and only if $0\leq\gamma\leq3/2$, so the threshold in \eqref{eq:gamma} is $\bar\gamma=3/2$. In particular, $\gamma=1$ lies in the Fisher-consistent regime, whereas $\gamma=2$ does not, as illustrated in Figure~\ref{fig:example_for_RSPO+}.
\qed

\bigskip
\noindent\textit{Proof of Corollary \ref{cor:spo_plus_a1}.}~
Fix $\bm x$ and write $\bar{\bm z}:=\bm z^{\star}(\bar{\bm y})$, $\bm Z^{\star}(\bm c):=\argmin_{\bm z\in\mathcal Z}\bm c^\top\bm z$, and let $V(\mathcal Z)$ denote the finite vertex set of the bounded polyhedron $\mathcal Z$. All expectations below are finite because $\mathcal Z$ is compact and $\mathbb E[\|\bm y\|_2]<\infty$. Since $\bm y$ admits a density and the set of cost vectors $\bm c$ for which $\bm Z^{\star}(\bm c)$ is not a singleton is contained in the union of the boundaries of the finitely many cones $\{-{\mathcal{N}}_{\mathcal{Z}}(\bm v):\bm v\in V(\mathcal Z)\}$, a Lebesgue-null set, the nominal decision map $\bm z^{\star}(\bm y)$ is single-valued almost surely; the same applies to $\bm z^{\star}(2\bar{\bm y}-\bm y)$, because $2\bar{\bm y}-\bm y$ also admits a density. If $\mathcal Z$ is a singleton the statement is trivial, so assume $|V(\mathcal Z)|\geq2$.

\emph{Step 1: the oracle predictor $2\bar{\bm y}$ is a target-risk minimizer.} For any $\hat{\bm y}$ and any fixed $\bm z_0\in\bm Z^{\star}(\hat{\bm y})$,
\begin{equation*}
    \begin{aligned}
        \mathbb{E}_{\bm{y}\sim \mathbb P_{\bm y|\bm x}}\big[\ell_{\mathcal{SPO}}(\hat{\bm y},\bm y)\big]
        =&\mathbb{E}_{\bm{y}\sim \mathbb P_{\bm y|\bm x}}\Big[\max_{\bm z\in\bm Z^{\star}(\hat{\bm y})}\bm y^\top\bm z\Big]-\mathbb{E}_{\bm{y}\sim \mathbb P_{\bm y|\bm x}}[v^{\star}(\bm y)]\\
        \geq&\bar{\bm y}^\top\bm z_0-\mathbb{E}_{\bm{y}\sim \mathbb P_{\bm y|\bm x}}[v^{\star}(\bm y)]
        \geq v^{\star}(\bar{\bm y})-\mathbb{E}_{\bm{y}\sim \mathbb P_{\bm y|\bm x}}[v^{\star}(\bm y)].
    \end{aligned}
\end{equation*}
Since the nominal argmin is invariant under positive scaling of the cost vector, uniqueness of $\bar{\bm z}$ for $\bar{\bm y}$ gives $\bm Z^{\star}(2\bar{\bm y})=\{\bar{\bm z}\}$, so the pointwise $\mathcal{SPO}$ risk at $2\bar{\bm y}$ equals $\bar{\bm y}^\top\bar{\bm z}-\mathbb{E}_{\bm{y}\sim \mathbb P_{\bm y|\bm x}}[v^{\star}(\bm y)]=v^{\star}(\bar{\bm y})-\mathbb{E}_{\bm{y}\sim \mathbb P_{\bm y|\bm x}}[v^{\star}(\bm y)]$, which attains the lower bound. Hence $2\bar{\bm y}\in\mathcal Y^{\star}_{\mathcal{SPO}}$.

\emph{Step 2: the oracle predictor minimizes the surrogate risk.} With $a=1$, using $\max_{\bm z\in\mathcal Z}(\bm y-\hat{\bm y})^\top\bm z=-v^{\star}(\hat{\bm y}-\bm y)$, the $\mathcal{SPO}_+$ loss reads $\ell_{\mathcal{SPO}_+}(\hat{\bm y},\bm y)=-v^{\star}(\hat{\bm y}-\bm y)+\hat{\bm y}^\top\bm z^{\star}(\bm y)-v^{\star}(\bm y)$ almost surely. Therefore, for any $\Delta\in\mathbb R^d$,
\begin{equation*}
    \begin{aligned}
        &\mathbb{E}_{\bm{y}\sim \mathbb P_{\bm y|\bm x}}\big[\ell_{\mathcal{SPO}_+}(2\bar{\bm y}+\Delta,\bm y)\big]
        -\mathbb{E}_{\bm{y}\sim \mathbb P_{\bm y|\bm x}}\big[\ell_{\mathcal{SPO}_+}(2\bar{\bm y},\bm y)\big] \\
        =&\mathbb{E}_{\bm{y}\sim \mathbb P_{\bm y|\bm x}}\big[-v^{\star}(2\bar{\bm y}+\Delta-\bm y)+v^{\star}(2\bar{\bm y}-\bm y)+\Delta^\top\bm z^{\star}(\bm y)\big].
    \end{aligned}
\end{equation*}
By central symmetry, $\bm w:=2\bar{\bm y}-\bm y$ has the same distribution as $\bm y$, so $\mathbb{E}_{\bm{y}\sim \mathbb P_{\bm y|\bm x}}[\Delta^\top\bm z^{\star}(\bm y)]=\mathbb{E}_{\bm{y}\sim \mathbb P_{\bm y|\bm x}}[\Delta^\top\bm z^{\star}(\bm w)]$, and the difference becomes
\begin{equation}\label{eq:spo_plus_a1_diff}
    \mathbb{E}_{\bm{y}\sim \mathbb P_{\bm y|\bm x}}\big[v^{\star}(\bm w)+\Delta^\top\bm z^{\star}(\bm w)-v^{\star}(\bm w+\Delta)\big].
\end{equation}
Since $v^{\star}(\bm w+\Delta)\leq(\bm w+\Delta)^\top\bm z^{\star}(\bm w)=v^{\star}(\bm w)+\Delta^\top\bm z^{\star}(\bm w)$, the integrand is nonnegative pointwise, so the difference \eqref{eq:spo_plus_a1_diff} is nonnegative for every $\Delta$, i.e., $2\bar{\bm y}$ minimizes the pointwise $\mathcal{SPO}_+$ risk with $a=1$.

\emph{Step 3: uniqueness of the surrogate minimizer.} Let $\Delta\neq\bm 0$. Since $\operatorname{int}\mathcal Z\neq\emptyset$, the linear function $\bm z\mapsto\Delta^\top\bm z$ is nonconstant on $\mathcal Z$, so there exist vertices $\bm z_1\in\argmax_{\bm z\in V(\mathcal Z)}\Delta^\top\bm z$ and $\bm z_2\in V(\mathcal Z)$ with $\Delta^\top\bm z_2<\Delta^\top\bm z_1$. Because $\mathcal Z$ is full-dimensional, the cone $\mathcal C_1:=\operatorname{int}\big(-{\mathcal{N}}_{\mathcal{Z}}(\bm z_1)\big)$ is a nonempty open cone with apex $\bm 0$, on which $\bm Z^{\star}(\bm c)=\{\bm z_1\}$ (by the same vertex-perturbation argument as in Step 1 of the proof of Theorem~\ref{thm:Fisher_consistency}(ii)). Fix $\bm c_0\in \mathcal C_1$ and $\varepsilon>0$ with $\mathcal B(\bm c_0,\varepsilon)\subseteq \mathcal C_1$. Since $\mathcal C_1$ is a cone, $\mathcal B(t\bm c_0,t\varepsilon)\subseteq \mathcal C_1$ for every $t>0$. For $\bm w\in \mathcal B(t\bm c_0,t\varepsilon)$ we have $\bm z^{\star}(\bm w)=\bm z_1$ and $v^{\star}(\bm w)=\bm w^\top\bm z_1$, so, bounding $v^{\star}(\bm w+\Delta)\leq(\bm w+\Delta)^\top\bm z_2$,
\begin{equation*}
    v^{\star}(\bm w)+\Delta^\top\bm z^{\star}(\bm w)-v^{\star}(\bm w+\Delta)\geq \bm w^\top(\bm z_1-\bm z_2)+\Delta^\top(\bm z_1-\bm z_2)
    \geq \Delta^\top(\bm z_1-\bm z_2)-t\big(\|\bm c_0\|_2+\varepsilon\big)D_{\mathcal Z},
\end{equation*}
where $D_{\mathcal Z}$ is the diameter of $\mathcal Z$. Since $\Delta^\top(\bm z_1-\bm z_2)>0$ is fixed, choosing $t>0$ small enough makes $v^{\star}(\bm w)+\Delta^\top\bm z^{\star}(\bm w)-v^{\star}(\bm w+\Delta)\geq\Delta^\top(\bm z_1-\bm z_2)/2>0$ on $\mathcal B(t\bm c_0,t\varepsilon)$. Finally, $\bm y$ has a density positive on an open set $O\ni2\bar{\bm y}$, so $\bm w=2\bar{\bm y}-\bm y$ has a density positive on the open set $2\bar{\bm y}-O\ni\bm 0$. For $t$ small enough, $\mathcal B(t\bm c_0,t\varepsilon)\subseteq2\bar{\bm y}-O$, with $\mathbb P_{\bm y|\bm x}(\bm w\in \mathcal B(t\bm c_0,t\varepsilon))>0$ and the difference \eqref{eq:spo_plus_a1_diff} is strictly positive. Therefore, $2\bar{\bm y}+\Delta$ is not a surrogate minimizer, and hence $\mathcal Y^{\star}_{\mathcal{SPO}_+}=\{2\bar{\bm y}\}$ for $a=1$.

\emph{Step 4: conclusion.} By Steps 1-3, for every $\bm x$ the unique minimizer of the pointwise $\mathcal{SPO}_+$ risk with $a=1$ belongs to $\mathcal Y^{\star}_{\mathcal{SPO}}$. Integrating over $\mathbb P_{\bm x}$ via the interchangeability principle, as in Section~\ref{subsec:fisher_consistency}, yields Fisher consistency of the $\mathcal{SPO}_+$ loss with $a=1$ with respect to the $\mathcal{SPO}$ loss.
\qed

\bigskip
\noindent\textit{Proof of Proposition \ref{prop:zstar_gradient_polyhedral}.}~
Let $S=S_{\bm y_0}$ for notational simplicity and write $\bm z^0=\bm z_\gamma^\star(\bm y_0)$. Since the objective is strongly convex and the constraints are affine, the KKT conditions are necessary and sufficient: $\bm z\in\mathcal Z$ solves the problem at $\bm y$ if and only if there exists $\bm\lambda\geq\bm 0$ such that
\begin{equation*}
    \begin{aligned}
        \bm{y}+\gamma \bm{z}-\bm{A}^\top \bm{\lambda}&=\bm 0,\\
        \bm{A}\bm z&\geq \bm b,\\
        \lambda_i(\bm a_i^\top \bm z-b_i)&=0,\quad \forall i.
    \end{aligned}
\end{equation*}
At $\bm y_0$, LICQ by assumption (i) guarantees that the multiplier $\bm\lambda^0$ is unique, and strict complementarity by assumption (ii) gives $\lambda^0_i>0$ for all $i\in S$ and, by definition of the active set, $\bm a_i^\top\bm z^0>b_i$ for all $i\notin S$.

We now construct a candidate primal--dual pair for $\bm y$ near $\bm y_0$ and verify that it satisfies the KKT conditions with active set $S$. Guided by the stationarity condition restricted to $S$ (\textit{i.e.}, $\bm\lambda_{S^c}=\bm 0$ and $\bm A_S\bm z=\bm b_S$), consider the linear system
\begin{equation*}
    \gamma \bm{z}-\bm{A}_{S}^\top \bm{\lambda}_{S}=-\bm{y},\qquad \bm{A}_{S}\bm{z}=\bm{b}_S.
\end{equation*}
From the first equation, $\bm z=\frac{1}{\gamma}(\bm A_S^\top \bm \lambda_S-\bm y)$. Substituting into $\bm A_S\bm z=\bm b_S$ and using the invertibility of $\bm A_S\bm A_S^\top$ (which follows from LICQ) yields the unique solution
\begin{equation*}
    \bm{\lambda}_{S}(\bm y)
    =
    (\bm{A}_{S} \bm{A}_{S}^\top)^{-1}
    (\gamma \bm{b}_{S} + \bm{A}_{S} \bm{y}),
    \qquad
    \bm{z}(\bm{y})
    =
    \frac{1}{\gamma}\Big(\bm{A}_{S}^\top (\bm{A}_{S} \bm{A}_{S}^\top)^{-1}\bm{A}_{S}-\bm{I}\Big)\bm{y}+ \bm{A}_{S}^\top (\bm{A}_{S} \bm{A}_{S}^\top)^{-1}\bm{b}_{S}.
\end{equation*}
Both mappings are affine, hence continuous, in $\bm y$, and by construction $\bm z(\bm y_0)=\bm z^0$ and $\bm\lambda_S(\bm y_0)=\bm\lambda^0_S$ (uniqueness of the multiplier under LICQ). By continuity and strict complementarity, there exists a neighborhood $\mathcal N$ of $\bm y_0$ such that, for every $\bm y\in\mathcal N$,
\begin{equation*}
    \bm\lambda_S(\bm y)>\bm 0
    \qquad\text{and}\qquad
    \bm a_i^\top\bm z(\bm y)>b_i,\quad \forall i\notin S.
\end{equation*}
Consequently, for every $\bm y\in\mathcal N$, the pair $\big(\bm z(\bm y),(\bm\lambda_S(\bm y),\bm 0_{S^c})\big)$ satisfies stationarity, primal feasibility ($\bm A_S\bm z(\bm y)=\bm b_S$ and the inactive constraints hold strictly), dual feasibility, and complementary slackness. By sufficiency of the KKT conditions, $\bm z_\gamma^\star(\bm y)=\bm z(\bm y)$ for all $\bm y\in\mathcal N$, and the active set of $\bm z_\gamma^\star(\bm y)$ is exactly $S$ on $\mathcal N$.

Thus, $\bm z_\gamma^\star(\cdot)$ is affine on $\mathcal N$. In particular, it is differentiable at $\bm y_0$ with Jacobian
\begin{equation*}
     \nabla\bm{z}^{\star}_{\gamma}(\bm{y}_0)=\frac{1}{\gamma}\big(\bm{A}_{S_{\bm{y}_0}}^{\top}(\bm{A}_{S_{\bm{y}_0}}\bm{A}_{S_{\bm{y}_0}}^\top)^{-1}\bm{A}_{S_{\bm{y}_0}}-\bm{I}\big).
\end{equation*}
This completes the proof.
\qed

\subsection{Proofs of Section \ref{sec:statistical_guarantee}}

\noindent\textit{Proof of Proposition \ref{prop:rspo_generalization}.}~
As noted in Subsection~\ref{subsec:rspo_excess_risk_bounds}, since $\bm z^{\star}_{\gamma}(\hat{\bm y})\in\mathcal Z$, the definition of $b$ implies
\begin{equation*}
    0\leq \ell_{\mathcal{RSPO}}(\hat{\bm y},\bm y)= \bm y^\top\bm z^{\star}_{\gamma}(\hat{\bm y})-v^{\star}(\bm y)\leq b,
    \qquad \forall \hat{\bm y}\in\mathbb R^d,\ \bm y\in\mathcal Y.
\end{equation*}
Let $S=\{(\bm x_i,\bm y_i)\}_{i\in[N]}$ denote the sample and define
\begin{equation*}
    \Phi(S)
    :=
    \sup_{\bm g\in\mathcal G}
    \left(
    R(\bm g)
    -
    \widehat R(\bm g)
    \right).
\end{equation*}
If one sample point $(\bm x_i,\bm y_i)$ is replaced by an independent copy $(\bm x_i',\bm y_i')$, then for every $\bm g\in\mathcal G$ the empirical risk changes by at most $b/N$, and hence $\Phi$ changes by at most $b/N$ as well. By McDiarmid's bounded difference inequality, with probability at least $1-\delta/2$,
\begin{equation*}
    \Phi(S)
    \leq
    \mathbb E[\Phi(S)]
    +
    b\sqrt{\frac{\log(2/\delta)}{2N}}.
\end{equation*}
Next, by the standard symmetrization argument \citep{bartlett2002rademacher}, letting $S'=\{(\bm x_i',\bm y_i')\}_{i\in[N]}$ be an independent ghost sample and $\sigma_1,\ldots,\sigma_N$ be independent Rademacher random variables,
\begin{equation*}
    \begin{aligned}
        \mathbb E[\Phi(S)]
        &=
        \mathbb E_S\left[\sup_{\bm g\in\mathcal G}
        \mathbb E_{S'}\left[
        \frac{1}{N}\sum_{i\in[N]}
        \left(
        \ell_{\mathcal{RSPO}}(\bm g(\bm x_i'),\bm y_i')
        -
        \ell_{\mathcal{RSPO}}(\bm g(\bm x_i),\bm y_i)
        \right)
        \right]\right]\\
        &\leq
        \mathbb E_{S,S',\bm\sigma}\left[\sup_{\bm g\in\mathcal G}
        \frac{1}{N}\sum_{i\in[N]}
        \sigma_i
        \left(
        \ell_{\mathcal{RSPO}}(\bm g(\bm x_i'),\bm y_i')
        -
        \ell_{\mathcal{RSPO}}(\bm g(\bm x_i),\bm y_i)
        \right)\right]\\
        &\leq
        2\mathfrak R_1(\mathcal H_{\ell_{\mathcal{RSPO}}}).
    \end{aligned}
\end{equation*}
The same two steps applied to
$\Phi'(S):=\sup_{\bm g\in\mathcal G}(\widehat R(\bm g)-R(\bm g))$
give the identical bound with probability at least $1-\delta/2$. A union bound over the two events yields
\begin{equation*}
    \sup_{\bm g\in \mathcal G}\left|R(\bm{g})-\widehat{R}(\bm{g})\right|\leq 2\mathfrak{R}_1(\mathcal{H}_{{\ell}_{\mathcal{RSPO}}})+b\sqrt{\frac{\log(2/\delta)}{2N}}
\end{equation*}
with probability at least $1-\delta$.
\qed

\bigskip
\noindent\textit{Proof of Theorem \ref{thm:rspo_meta_generalization}.}~
Let $\mathcal E$ denote the event
\begin{equation*}
    \sup_{\bm g\in \mathcal G}\left|R(\bm{g})-\widehat{R}(\bm{g})\right|\leq 2\mathfrak{R}_1(\mathcal{H}_{{\ell}_{\mathcal{RSPO}}})+b\sqrt{\frac{\log(2/\delta)}{2N}},
\end{equation*}
so that $\mathbb P(\mathcal E)\geq 1-\delta$ by Proposition~\ref{prop:rspo_generalization}. Fix an arbitrary $\epsilon>0$ and choose $\bm g_\epsilon\in\mathcal G$ such that
$R(\bm g_\epsilon)\leq R^{\star}(\mathcal G)+\epsilon$ by the definition of the infimum. On the event $\mathcal E$, we have
\begin{equation*}
    \begin{aligned}
        R(\widehat{\bm g}_{\rm RSPO})-R^{\star}(\mathcal G)
        =&
        \left[R(\widehat{\bm g}_{\rm RSPO})-\widehat R(\widehat{\bm g}_{\rm RSPO})\right]
        +
        \left[\widehat R(\widehat{\bm g}_{\rm RSPO})-\widehat R(\bm g_\epsilon)\right]\\
        &+
        \left[\widehat R(\bm g_\epsilon)-R(\bm g_\epsilon)\right]
        +
        \left[R(\bm g_\epsilon)-R^{\star}(\mathcal G)\right]\\
        \leq&
        2\sup_{\bm g\in \mathcal G}\left|R(\bm{g})-\widehat{R}(\bm{g})\right|
        +0+\epsilon\\
        \leq&
        4\mathfrak{R}_1(\mathcal{H}_{{\ell}_{\mathcal{RSPO}}})+b\sqrt{\frac{2\log(2/\delta)}{N}}+\epsilon,
    \end{aligned}
\end{equation*}
where the second inequality uses $\widehat R(\widehat{\bm g}_{\rm RSPO})\leq \widehat R(\bm g_\epsilon)$, which holds because $\widehat{\bm g}_{\rm RSPO}$ minimizes the empirical risk, and the last inequality uses the inequality defining event $\mathcal E$. Since the event $\mathcal E$ does not depend on $\epsilon$, letting $\epsilon\downarrow0$ gives, on $\mathcal E$,
\begin{equation*}
    R(\widehat{\bm g}_{\rm RSPO})-R^{\star}(\mathcal G)
    \leq
    4\mathfrak{R}_1(\mathcal{H}_{{\ell}_{\mathcal{RSPO}}})+b\sqrt{\frac{2\log(2/\delta)}{N}}.
\end{equation*}
Finally, applying Lemma~\ref{lem:vector_contraction_inequality} to bound $4\mathfrak{R}_1(\mathcal{H}_{{\ell}_{\mathcal{RSPO}}})\leq \frac{4\sqrt2 r}{\gamma}\mathfrak R_d(\mathcal G)$ completes the proof.
\qed

\bigskip
\noindent\textit{Proof of Proposition \ref{prop:rspo+_generalization}.}~
Following a similar argument to that in Proposition~\ref{prop:rspo_generalization}, for any $\delta>0$, with probability at least $1-\delta$, for all $\bm{g}\in\mathcal{G}$ it holds that
\begin{equation*}
    |R_+(\bm{g})-\widehat R_+(\bm{g})|\leq 2\mathfrak{R}_1(\mathcal H_{\ell_{\mathcal{RSPO}_+}})+\mathfrak{L}(\gamma)\sqrt{\frac{\log(2/\delta)}{2N}}.
\end{equation*}
This result applies because the loss takes values in $[0,\mathfrak L(\gamma)]$ on $\mathcal G \times \mathcal Y$. Indeed, the upper bound holds by the definition of $\mathfrak L(\gamma)$, and nonnegativity follows from $\ell_{\mathcal{RSPO}_+} \geq \ell_{\mathcal{RSPO}} \geq 0$ (Proposition~\ref{prop:rspo+_loss_properties}(i)). 

Moreover, by Proposition~\ref{prop:rspo+_loss_properties}(iv), $\ell_{\mathcal{RSPO}_+}(\cdot,\bm y)$ is $aD_{\mathcal Z}$-Lipschitz continuous for every $\bm y\in\mathcal Y$. By the vector contraction inequality \citep[Corollary~4]{maurer2016vector}, $\mathfrak{R}_1(\mathcal H_{\ell_{\mathcal{RSPO}_+}})\leq\sqrt{2}aD_{\mathcal{Z}}\mathfrak{R}_d(\mathcal{G})$. Thus,
\begin{equation*}
    |R_+(\bm{g})-\widehat R_+(\bm{g})|\leq 2\sqrt{2}aD_{\mathcal{Z}}\mathfrak{R}_d(\mathcal{G})+\mathfrak{L}(\gamma)\sqrt{\frac{\log(2/\delta)}{2N}}.
\end{equation*}

Proceeding as in Theorem~\ref{thm:rspo_meta_generalization} and using $\mathfrak{R}_d(\mathcal{G})\leq C_0/\sqrt{N}$, we obtain that, for all $\delta\in(0,1)$, with probability $1-\delta$,
\begin{equation*}
    R_+(\widehat{\bm{g}}_{{\rm RSPO}_+}) - R^{\star}_+(\mathcal{G})\leq \frac{4\sqrt{2}aD_{\mathcal{Z}}C_0}{\sqrt{N}}+\mathfrak{L}(\gamma)\sqrt{\frac{2\log(2/\delta)}{N}}.
\end{equation*}
The proof is complete.
\qed

\bigskip
\noindent\textit{Proof of Theorem \ref{thm:direct_closure}.}~
We complete this proof in six steps. \textit{Step~1} bounds the surrogate excess risk of $\widehat{\bm g}_{{\rm RSPO}_+}$ relative to the oracle predictor $\bm g_0$ by $B_N+A^+_{\mathcal G}$.  \textit{Step~2} converts the local strong convexity of Condition~\ref{con:reference_condition}(ii) into a growth bound valid on all of $\mathbb R^d$.  \textit{Step~3} integrates that bound to control $\mathbb E_{\bm x\sim \bb P_{\bm x}}\|\widehat{\bm g}_{{\rm RSPO}_+}(\bm x)-\bm g_0(\bm x)\|_2$. \textit{Step~4} transfers this control to the target risk through the Lipschitz continuity of the decision map.  \textit{Step~5} identifies the oracle predictor with the target optimum, which turns the resulting estimate into a bound on the target excess risk.  \textit{Step~6} reads off the rate. To simplify the exposition, we define
\begin{equation*}
    \Delta_\gamma(\bm c;\bm x)
    :=
    R(\bm c;\bm x)-\inf_{\bm c'\in\mathbb R^d}R(\bm c';\bm x),
    \qquad
    \Delta^+_{\gamma}(\bm c;\bm x)
    :=
    R_+(\bm c;\bm x)-\inf_{\bm c'\in\mathbb R^d}R_+(\bm c';\bm x).
\end{equation*}

\emph{Step 1.}
Condition~\ref{con:reference_condition}(i) and the tower property imply $R_+(\bm g_0)\leq R_+(\bm g)$ for every measurable predictor $\bm g$ with finite surrogate risk.  Hence
\begin{equation}
    R_+(\bm g_0)
    =
    R_+^\star(\mathcal G_{\rm all}).
\label{eq:reference_point_surrogate_optimal}
\end{equation}
Proposition~\ref{prop:rspo+_generalization} implies that, with probability at least $1-\delta$,
\begin{equation*}
    R_+(\widehat{\bm g}_{{\rm RSPO}_+})
    -
    R_+^\star(\mathcal G)
    \leq
    \frac{4\sqrt 2\,D_{\mathcal Z}C_0}{\sqrt N}
    +
    \mathfrak L(\gamma)
    \sqrt{\frac{2\log(2/\delta)}{N}}.
\end{equation*}
Therefore, on this event,
\begin{align}
    R_+(\widehat{\bm g}_{{\rm RSPO}_+})-R_+(\bm g_0)
    &=
    \bigl[
        R_+(\widehat{\bm g}_{{\rm RSPO}_+})
        -
        R_+^\star(\mathcal G)
    \bigr]
    +
    \bigl[
        R_+^\star(\mathcal G)
        -
        R_+^\star(\mathcal G_{\rm all})
    \bigr]\notag\\
    &\leq
    \frac{4\sqrt 2\,D_{\mathcal Z}C_0}{\sqrt N}
    +
    \mathfrak L(\gamma)
    \sqrt{\frac{2\log(2/\delta)}{N}}
    +
    \left[
        R_+^\star(\mathcal G)
        -
        R_+^\star(\mathcal G_{\rm all})
    \right]\notag\\
    &=
    B_N+A^+_{\mathcal G}.
\label{eq:appE_surrogate_excess}
\end{align}
In particular, $B_N+A^+_{\mathcal G}\geq0$, so the square root appearing in \eqref{eq:target_bound} is well-defined.

\emph{Step 2.}
We claim that, for $\bb P_{\bm x}$-almost every $\bm x$ and every $\bm c\in\mathbb R^d$,
\begin{equation}\label{eq:appE_two_regime}
    \Delta^+_{\gamma}(\bm c;\bm x)
    \geq
    \frac{\mu}{2}
    \min\left\{
        \|\bm c-\bm g_0(\bm x)\|_2^2,
        \rho\|\bm c-\bm g_0(\bm x)\|_2
    \right\}.
\end{equation}
The exceptional null set is the union of those in Conditions~\ref{con:reference_condition}(i) and~(ii) and does not depend on $\bm c$, so \eqref{eq:appE_two_regime} may be applied at a sample-dependent argument.  Fix such an $\bm x$ and write $t:=\|\bm c-\bm g_0(\bm x)\|_2$.  Since $\bm g_0(\bm x)$ is a global minimizer of the convex function $R_+(\cdot;\bm x)$, we have $\bm 0\in\partial R_+(\bm g_0(\bm x);\bm x)$. Under Condition~\ref{con:reference_condition}(i) the pointwise surrogate excess risk takes the form
\begin{equation*}
    \Delta^+_{\gamma}(\bm c;\bm x)
    =
    R_+(\bm c;\bm x)-R_+(\bm g_0(\bm x);\bm x),
    \qquad
    \bm c\in\mathbb R^d.
\end{equation*}
If $t\leq\rho$, strong convexity on the closed ball $\mathcal B(\bm g_0(\bm x),\rho)$ gives
\begin{equation*}
    \Delta^+_{\gamma}(\bm c;\bm x)
    \geq
    \frac{\mu}{2}t^2.
\end{equation*}
If $t>\rho$, set $\bm c_\rho:=\bm g_0(\bm x)+\frac{\rho}{t}(\bm c-\bm g_0(\bm x))$, so that $\|\bm c_\rho-\bm g_0(\bm x)\|_2=\rho$.  Convexity and $\bm c_\rho=(1-\rho/t)\bm g_0(\bm x)+(\rho/t)\bm c$ give
\begin{equation*}
    \Delta^+_{\gamma}(\bm c_\rho;\bm x)
    \leq
    (1-\frac{\rho}{t})\Delta^+_{\gamma}(\bm g_0(\bm x);\bm x)+\frac{\rho}{t}\,\Delta^+_{\gamma}(\bm c;\bm x)=\frac{\rho}{t}\,\Delta^+_{\gamma}(\bm c;\bm x) .
\end{equation*}
The first case applies to $\bm c_\rho$, so
\begin{equation*}
    \Delta^+_{\gamma}(\bm c;\bm x)
    \geq
    \frac{t}{\rho}\,\Delta^+_{\gamma}(\bm c_\rho;\bm x)
    \geq
    \frac{t}{\rho}\cdot\frac{\mu}{2}\rho^2
    =
    \frac{\mu\rho}{2}t .
\end{equation*}
Since $\min\{t^2,\rho t\}$ equals $t^2$ when $t\leq\rho$ and $\rho t$ otherwise, the two cases prove \eqref{eq:appE_two_regime}.

\emph{Step 3.}
Taking expectations in \eqref{eq:appE_two_regime} at $\bm c=\widehat{\bm g}_{{\rm RSPO}_+}(\bm x)$, and using $\mathbb E_{\bm x\sim \bb P_{\bm x}}[\Delta^+_{\gamma}(\widehat{\bm g}_{{\rm RSPO}_+}(\bm x);\bm x)]=R_+(\widehat{\bm g}_{{\rm RSPO}_+})-R_+(\bm g_0)$ together with \eqref{eq:appE_surrogate_excess}, yields
\begin{equation}\label{eq:appE_truncated_moment}
    \mathbb E_{\bm x\sim \bb P_{\bm x}}
    \left[
        \min\left\{
            \|\widehat{\bm g}_{{\rm RSPO}_+}(\bm x)-\bm g_0(\bm x)\|_2^2,
            \rho
            \|\widehat{\bm g}_{{\rm RSPO}_+}(\bm x)-\bm g_0(\bm x)\|_2
        \right\}
    \right]
    \leq
    \frac{2(B_N+A^+_{\mathcal G})}{\mu}.
\end{equation}
For every $t\geq0$,
\begin{equation}\label{eq:appE_elementary}
    t
    \leq
    \sqrt{\min\{t^2,\rho t\}}
    +
    \frac{1}{\rho}\min\{t^2,\rho t\}.
\end{equation}
Indeed, if $t\leq\rho$ the first term on the right equals $t$, and if $t>\rho$ the second term equals $t$; in either case the remaining term is nonnegative.  Applying \eqref{eq:appE_elementary} with $t=\|\widehat{\bm g}_{{\rm RSPO}_+}(\bm x)-\bm g_0(\bm x)\|_2$, taking expectations, and bounding the first resulting term by Jensen's inequality (the square root is concave) gives
\begin{align}
    &\mathbb E_{\bm x\sim \bb P_{\bm x}}
    \left[
        \|\widehat{\bm g}_{{\rm RSPO}_+}(\bm x)-\bm g_0(\bm x)\|_2
    \right]\notag\\
    &\leq
    \left(
        \mathbb E_{\bm x\sim \bb P_{\bm x}}
        \left[
            \min\left\{
                \|\widehat{\bm g}_{{\rm RSPO}_+}(\bm x)-\bm g_0(\bm x)\|_2^2,
                \rho
                \|\widehat{\bm g}_{{\rm RSPO}_+}(\bm x)-\bm g_0(\bm x)\|_2
            \right\}
        \right]
    \right)^{1/2}\notag\\
    &\quad+
    \frac{1}{\rho}
    \mathbb E_{\bm x\sim \bb P_{\bm x}}
    \left[
        \min\left\{
            \|\widehat{\bm g}_{{\rm RSPO}_+}(\bm x)-\bm g_0(\bm x)\|_2^2,
            \rho
            \|\widehat{\bm g}_{{\rm RSPO}_+}(\bm x)-\bm g_0(\bm x)\|_2
        \right\}
    \right]\notag\\
    &\leq
    \sqrt{\frac{2(B_N+A^+_{\mathcal G})}{\mu}}
    +
    \frac{2(B_N+A^+_{\mathcal G})}{\mu\rho}.
\label{eq:appE_first_moment}
\end{align}

\emph{Step 4.}
By the definition of $\ell_{\mathcal{RSPO}}$ and the tower property, every measurable $\bm g$ satisfies
\begin{equation}\label{eq:appE_target_representation}
    R(\bm g)
    =
    \mathbb E_{\bm x\sim \bb P_{\bm x}}
    \left[
        \bar{\bm y}(\bm x)^\top\bm z_\gamma^\star(\bm g(\bm x))
    \right]
    -
    \mathbb E_{\bm y\sim\bb P_{\bm y}}[v^\star(\bm y)] ,
\end{equation}
all terms being finite by \eqref{eq:appE_loss_envelope}.  Subtracting \eqref{eq:appE_target_representation} at $\bm g$ and at $\bm g_0$ gives
\begin{equation}\label{eq:appE_target_reference_identity}
    R(\bm g)-R(\bm g_0)
    =
    \mathbb E_{\bm x\sim \bb P_{\bm x}}
    \left[
        \bar{\bm y}(\bm x)^\top
        \left(
            \bm z_\gamma^\star(\bm g(\bm x))
            -
            \bm z_\gamma^\star(\bm g_0(\bm x))
        \right)
    \right].
\end{equation}
By Lemma~\ref{lem:oracle_projection}, $\bm z_\gamma^\star(\cdot)$ is $1/\gamma$-Lipschitz on $\mathbb R^d$.  Hence Condition~\ref{con:reference_condition}(iv), Cauchy--Schwarz, and \eqref{eq:appE_first_moment} imply
\begin{align}
    \left|
        R(\widehat{\bm g}_{{\rm RSPO}_+})-R(\bm g_0)
    \right|
    &\leq
    \frac{\bar\beta}{2\gamma}
    \mathbb E_{\bm x\sim \bb P_{\bm x}}
    \left[
        \|\widehat{\bm g}_{{\rm RSPO}_+}(\bm x)-\bm g_0(\bm x)\|_2
    \right]\notag\\
    &\leq
    \frac{\bar\beta}{2\gamma}
    \left[
        \sqrt{\frac{2(B_N+A^+_{\mathcal G})}{\mu}}
        +
        \frac{2(B_N+A^+_{\mathcal G})}{\mu\rho}
    \right].
\label{eq:reference_point_target_bound}
\end{align}

\emph{Step 5.}
By Condition~\ref{con:reference_condition}(iii) and Lemma~\ref{lem:rspo_pointwise_excess}, $\Delta_\gamma(\bm g_0(\bm x);\bm x)=0$ for almost every $\bm x$, that is, $\bm g_0(\bm x)$ minimizes $R(\cdot;\bm x)$ over $\mathbb R^d$. Integrating the pointwise inequality $R(\bm g_0(\bm x);\bm x)\leq R(\bm g(\bm x);\bm x)$ and using the tower property gives $R(\bm g_0)\leq R(\bm g)$ for every measurable $\bm g$. Since $R^\star=R^\star(\mathcal G_{\rm all})$, this yields
\begin{equation}\label{eq:appE_fisher_global}
    R(\bm g_0)=R^\star.
\end{equation}
Combining \eqref{eq:appE_fisher_global} with \eqref{eq:reference_point_target_bound} proves \eqref{eq:target_bound}.

\emph{Step 6.}
Since $\bm g_0\in\mathcal G$, we have $R_+^\star(\mathcal G)=R_+^\star(\mathcal G_{\rm all})$ by Condition~\ref{con:reference_condition}(i). The stated constants do not depend on $N$.  Fix $\eta\in(0,1)$ and take $\delta=\eta$ in \eqref{eq:target_bound}. By \eqref{eq:appE_surrogate_excess},
\begin{equation}\label{eq:appE_scaled_budget}
    \bar c(\eta)
    :=
    \sqrt N\,(B_N+A^+_{\mathcal G})
    =
    4\sqrt 2\,D_{\mathcal Z}C_0
    +
    \mathfrak L(\gamma)\sqrt{2\log(2/\eta)}
    <
    \infty ,
\end{equation}
which does not depend on $N$. Hence, on the event of probability at least $1-\eta$ on which \eqref{eq:target_bound} holds, and for every $N\geq1$,
\begin{equation*}
    N^{1/4}
    \left[
        R(\widehat{\bm g}_{{\rm RSPO}_+})-R^\star
    \right]
    \leq
    \frac{\bar\beta}{2\gamma}
    \left[
        \sqrt{\frac{2\bar c(\eta)}{\mu}}
        +
        \frac{2\bar c(\eta)}{\mu\rho}N^{-1/4}
    \right]
    \leq
    \frac{\bar\beta}{2\gamma}
    \left[
        \sqrt{\frac{2\bar c(\eta)}{\mu}}
        +
        \frac{2\bar c(\eta)}{\mu\rho}
    \right]
    =:
    M(\eta).
\end{equation*}
Since $M(\eta)$ depends only on $\eta$ and the structural constants,
\begin{equation*}
    \mathbb P
    \left(
        N^{1/4}
        \left[
            R(\widehat{\bm g}_{{\rm RSPO}_+})-R^\star
        \right]
        >
        M(\eta)
    \right)
    \leq
    \eta,
    \qquad
    N\geq1 ,
\end{equation*}
which is precisely the statement that $R(\widehat{\bm g}_{{\rm RSPO}_+})-R^\star=\mathcal O_p(N^{-1/4})$.
\qed

\bigskip
\noindent\textit{Proof of Corollary \ref{cor:direct_fast_rate}.}~

Fix $\bm x$ outside the exceptional null sets and $\bm c\in\mathbb R^d$. For concise exposition, write $t:=\|\bm c-\bm g_0(\bm x)\|_2$, $\tilde r:=\min\{r_0,\rho\}$, and $\kappa_{\rm dir}:=\min\{\mu/L,\,\mu\gamma\tilde r/\bar\beta\}$.

If $0\leq t\leq\tilde r$, then $t\leq r_0$, so \eqref{eq:target_quadratic} applies, and $t\leq\rho$, so \eqref{eq:appE_two_regime} is in its quadratic regime. The two give, respectively,
\begin{equation*}
    \Delta_\gamma(\bm c;\bm x)
    \leq
    \frac{L}{2}t^2 ,\quad
     \Delta^+_{\gamma}(\bm c;\bm x)
    \geq
    \frac{\mu}{2}t^2.
\end{equation*}
Consequently,
\begin{equation*}
    \Delta_\gamma(\bm c;\bm x)
    \leq
    \frac{L}{\mu}\,\Delta^+_{\gamma}(\bm c;\bm x)
    \leq
    \frac{1}{\kappa_{\rm dir}}\,\Delta^+_{\gamma}(\bm c;\bm x) ,
\end{equation*}
where the second inequality uses $\kappa_{\rm dir}\leq\mu/L$.

If $t>\tilde r$, then Lemma~\ref{lem:rspo_pointwise_excess}, Condition~\ref{con:reference_condition}(iii)--(iv) and the $1/\gamma$-Lipschitz property of $\bm z_\gamma^\star(\cdot)$ give
\begin{equation*}
    \Delta_\gamma(\bm c;\bm x)
    =
    \bar{\bm y}(\bm x)^\top
    \left(
        \bm z_\gamma^\star(\bm c)
        -
        \bm z_\gamma^\star(\bm g_0(\bm x))
    \right)
    \leq
    \frac{\bar\beta}{2\gamma}t .
\end{equation*}
Since $\tilde r=\min\{r_0,\rho\}$, the condition $t>\tilde r$ implies
\begin{equation*}
    \min\{t^2,\rho t\}\geq\tilde r\,t.
\end{equation*}
Indeed, if $t\leq\rho$ then $t^2>\tilde r\,t$, and if $t>\rho$ then $\rho t\geq\tilde r\,t$.  Thus \eqref{eq:appE_two_regime} gives
\begin{equation*}
    \Delta^+_{\gamma}(\bm c;\bm x)
    \geq
    \frac{\mu\tilde r}{2}t ,
\end{equation*}
and it follows that
\begin{equation*}
    \Delta_\gamma(\bm c;\bm x)
    \leq
    \frac{\bar\beta}{\mu\gamma\tilde r}\,\Delta^+_{\gamma}(\bm c;\bm x)
    \leq
    \frac{1}{\kappa_{\rm dir}}\,\Delta^+_{\gamma}(\bm c;\bm x) ,
\end{equation*}
where the last inequality uses $\kappa_{\rm dir}\leq\mu\gamma\tilde r/\bar\beta$. The two cases establish the global comparison
\begin{equation}
    \Delta_\gamma(\bm c;\bm x)
    \leq
    \frac{1}{\kappa_{\rm dir}}\,\Delta^+_{\gamma}(\bm c;\bm x),
    \qquad
    \bm c\in\mathbb R^d .
\label{eq:appE_pointwise_fast}
\end{equation}

Applying \eqref{eq:appE_pointwise_fast} at $\bm c=\widehat{\bm g}_{{\rm RSPO}_+}(\bm x)$ and integrating, we have
\begin{align*}
    R(\widehat{\bm g}_{{\rm RSPO}_+})-R^\star
    &= \mathbb E_{\bm x\sim \bb P_{\bm x}}
    \left[
        R(\widehat{\bm g}_{{\rm RSPO}_+}(\bm x);\bm x)-R(\bm g_0(\bm x);\bm x)
    \right]\\
    &=\mathbb E_{\bm x\sim \bb P_{\bm x}}
    \left[
        \Delta_\gamma(
            \widehat{\bm g}_{{\rm RSPO}_+}(\bm x);
            \bm x
        )
    \right]\\
    &\leq\frac{1}{\kappa_{\rm dir}}\mathbb E_{\bm x\sim \bb P_{\bm x}}
    \left[
        \Delta_\gamma^+(
            \widehat{\bm g}_{{\rm RSPO}_+}(\bm x);
            \bm x
        )
    \right]\\
    &=\frac{1}{\kappa_{\rm dir}} \mathbb E_{\bm x\sim \bb P_{\bm x}}
    \left[
        R_+(\widehat{\bm g}_{{\rm RSPO}_+}(\bm x);\bm x)-R_+(\bm g_0(\bm x);\bm x)
    \right]\\
    &=
    \frac{1}{\kappa_{\rm dir}}
    \bigl[
        R_+(\widehat{\bm g}_{{\rm RSPO}_+})
        -
        R_+(\bm g_0)
    \bigr]\leq
    \frac{B_N+A^+_{\mathcal G}}{\kappa_{\rm dir}},
\end{align*}
which proves \eqref{eq:fast_target}.

For the rate, fix $\eta\in(0,1)$ and take $\delta=\eta$. Since $\bm g_0\in\mathcal G$, we have $R_+^\star(\mathcal G)=R_+^\star(\mathcal G_{\rm all})$ by Condition~\ref{con:reference_condition}(i). Thus, $\bar c(\eta)=\sqrt N\,B_N$ of \eqref{eq:appE_scaled_budget} does not depend on $N$, so on an event of probability at least $1-\eta$,
\begin{equation*}
    N^{1/2}
    \left[
        R(\widehat{\bm g}_{{\rm RSPO}_+})-R^\star
    \right]
    \leq
    \frac{\bar c(\eta)}{\kappa_{\rm dir}},
    \qquad
    N\geq1 ,
\end{equation*}
and $\kappa_{\rm dir}>0$ does not depend on $N$, which gives the stated $\mathcal O_p(N^{-1/2})$ rate.
\qed

\newpage
\section{The Stochastic Gradient-Descent Algorithm}
\label{app:gradient_algorithm}

The procedure of the stochastic gradient descent scheme with Armijo Backtracking for Problem~\eqref{eq:rspo_refine_obj} is summarized in Algorithm~\ref{alg:gradient-descent}, in which
\begin{equation*}
    \hat{L}_{\mathcal{B}_t}(\theta)
    =\frac{1}{B_{\mathrm s}}\sum_{i\in\mathcal{B}_t}
    \big[\bm y_i^\top \bm z_\gamma^\star(\bm g_\theta(\bm x_i))-v^\star(\bm y_i)\big]
    +\lambda\Omega(\bm{g}_{\theta})
\end{equation*}
denotes the batch counterpart of~\eqref{eq:rspo_refine_obj}. Taking $B_{\mathrm s}=N$ recovers the full-batch method. Per iteration, the gradient phase solves the decision problem \ref{ilroopt} once per sample and forms one $|S_i^{(t)}|\times|S_i^{(t)}|$ linear system per sample for $\bm q_i^{(t)}$; each line-search trial $k=0,1,\dots$ costs another $B_{\mathrm s}$ decision map evaluations. The total is $(k^{\star}+2)B_{\mathrm s}$ decision map evaluations when the step is accepted at trial $k^{\star}$, and at most $(K+2)B_{\mathrm s}$ otherwise. The cost is thus dominated by the decision map, which favors a small number of refinement steps over running the scheme to convergence.

For the full-batch version the method is a genuine descent method, and the line search terminates after at most $K+1$ trials by construction. The proposition below shows that a strictly decreasing step is accepted whenever $\hat L_N$ is differentiable at the current iterate with nonzero gradient, provided $K$ is large enough.

\begin{proposition}[Monotone Descent] \label{prop:descent}
Let $B_{\mathrm s}=N$, let $\mathcal Z$ be a nonempty bounded polyhedron, and let $\hat L_N$ be given by~\eqref{eq:rspo_refine_obj} with $\theta\to\Omega(\bm{g}_{\theta})$ differentiable and bounded below. 
Then $\{\hat L_N(\theta^{(t)})\}_{t\geq0}$ is nonincreasing and bounded below, hence convergent. 
Moreover, if $\hat L_N$ is differentiable at $\theta^{(t)}$ and $\bm h^{(t)}=\nabla\hat L_N(\theta^{(t)})\neq\bm 0$, then the Armijo test is satisfied for all sufficiently small stepsizes, so for $K$ large enough the step is accepted and the resulting decrease is at least $c_{\rm ls}\,\rho^{k^{\star}}_{\rm ls}\eta_t\|\bm h^{(t)}\|_2^2>0$.
\end{proposition}

\noindent\textit{Proof of Proposition \ref{prop:descent}.}~
Monotonicity is immediate from the acceptance rule: either the test succeeds, in which case
\begin{equation*}
    \hat L_N(\theta^{(t+1)})
    \leq
    \hat L_N(\theta^{(t)})-c_{\rm ls}\,\rho^{k^{\star}}_{\rm ls}\eta_t\|\bm h^{(t)}\|_2^2
    \leq
    \hat L_N(\theta^{(t)}),
\end{equation*}
or it fails and $\theta^{(t+1)}=\theta^{(t)}$. For the lower bound, $\ell_{\mathcal{RSPO}}\geq0$ by Proposition~\ref{prop:rspo+_loss_properties}(i) and $\Omega$ is bounded below by assumption, so $\hat L_N\geq\lambda\inf_{\theta\in\Theta}\Omega(\bm{g}_{\theta})>-\infty$. A nonincreasing sequence that is bounded below converges.
 
For the second claim, suppose $\hat L_N$ is differentiable at $\theta^{(t)}$ with $\bm h^{(t)}=\nabla\hat L_N(\theta^{(t)})\neq\bm 0$. Differentiability gives
\begin{equation*}
    \hat L_N\big(\theta^{(t)}-s\bm h^{(t)}\big)
    =
    \hat L_N(\theta^{(t)})-s\|\bm h^{(t)}\|_2^2+o(s),
    \qquad 
    s\downarrow0,
\end{equation*}
so that
\begin{equation*}
    \hat L_N\big(\theta^{(t)}-s\bm h^{(t)}\big)-\hat L_N(\theta^{(t)})+c_{\rm ls}\,s\|\bm h^{(t)}\|_2^2
    =
    -(1-c_{\rm ls})\,s\|\bm h^{(t)}\|_2^2+o(s)<0
\end{equation*}
for all sufficiently small $s>0$, because $c_{\rm ls}<1$ and $\|\bm h^{(t)}\|_2>0$. Since $\rho_{\rm ls}\in(0,1)$, the trial stepsizes $\rho^{k}_{\rm ls}\eta_t$ decrease to zero as $k$ grows, so the Armijo test is satisfied for all $k$ large enough. If $K$ is at least the smallest such index, the step is accepted, and the stated decrease is the acceptance inequality itself.
\qed

\begin{algorithm}[t]
\caption{Stochastic Gradient Descent with Armijo Backtracking for Problem~\eqref{eq:rspo_refine_obj}}
\label{alg:gradient-descent}
\KwIn{sample $\{(\bm{x}_i,\bm{y}_i)\}_{i=1}^N$; polyhedral set $\mathcal Z=\{\bm{z}:\bm{A}\bm{z}\geq \bm{b}\}$; $\gamma>0$; $\Omega(\cdot)$, $\lambda\geq0$; stepsize $\eta_0>0$; batch size $B_{\mathrm s}$ ($B_{\mathrm s}=N$ gives the full-batch method); iteration budget $T$; backtracking parameters $\rho_{\rm ls}\in(0,1)$, $K\in\mathbb{N}$, $c_{\rm ls}\in(0,1)$.}
\textbf{Initialize:} $\theta^{(0)}\leftarrow$ solution of the $\mathcal{RSPO}_+$ problem~\eqref{eq:rspo+_ermprob}\;
\For{$t=0,1,\dots,T-1$}{
  Sample $\mathcal{B}_t\subseteq[N]$, $|\mathcal{B}_t|=B_{\mathrm s}$, uniformly without replacement\;
  \tcp{batch gradient}
  \For{$i\in\mathcal{B}_t$}{
    $\hat{\bm y}_i^{(t)}\leftarrow \bm g_{\theta^{(t)}}(\bm x_i)$,\quad
    $\bm z_i^{(t)}\leftarrow \bm z^{\star}_{\gamma}(\hat{\bm y}_i^{(t)})$,\quad
    $\bm q_i^{(t)}\leftarrow \nabla\bm z^{\star}_{\gamma}(\hat{\bm y}_i^{(t)})^{\top}\bm y_i$ via \eqref{eq:zstar_gradient} with $S_i^{(t)}=\{j:\bm a_j^{\top}\bm z_i^{(t)}=b_j\}$\;
  }
  $\bm h^{(t)}\leftarrow \frac{1}{B_{\mathrm s}}\sum_{i\in\mathcal{B}_t}\big[\nabla_{\theta}\bm g_{\theta^{(t)}}(\bm x_i)\big]^{\top}\bm q_i^{(t)}+\lambda\nabla_{\theta}\Omega(\bm g_{\theta^{(t)}})$\;
  $\eta_t\leftarrow \eta_0$ if $B_{\mathrm s}=N$, and $\eta_t\leftarrow\eta_0/\sqrt{t+1}$ otherwise\;
  \tcp{Armijo backtracking}
  \eIf{$\exists\,k\in\{0,\dots,K\}:\ \hat{L}_{\mathcal{B}_t}\big(\theta^{(t)}-\rho^{k}_{\rm ls}\eta_t\,\bm h^{(t)}\big)\leq \hat{L}_{\mathcal{B}_t}(\theta^{(t)})-c_{\rm ls}\,\rho^{k}_{\rm ls}\eta_t\|\bm h^{(t)}\|^2_2$}{
    $\theta^{(t+1)}\leftarrow \theta^{(t)}-\rho^{k^{\star}}_{\rm ls}\eta_t\,\bm h^{(t)}$, where $k^{\star}$ is the smallest such $k$\;
  }{
    $\theta^{(t+1)}\leftarrow \theta^{(t)}$\;
  }
}
\KwOut{$\theta^{(T)}$.}
\end{algorithm}

\newpage
\section{Concrete Excess Risk Bounds of the $\mathcal{RSPO}$ Predictor}
\label{app:gen_examples}

In this appendix, we instantiate the meta generalization bound of Theorem~\ref{thm:rspo_meta_generalization} for four representative classes of predictors: (i) the bounded affine class; (ii) the polynomial discrimination class; (iii) the Dudley-entropy-integral bounded class; and (iv) the vector-valued RKHS class.
Throughout, we assume $\gamma>0$ and $\|\bm y\|_2\leq r$ for all $\bm y\in\mathcal Y$, as stated in Theorem~\ref{thm:rspo_meta_generalization}. The scalar antecedents of the complexity bounds below are classical; the content of this appendix is their adaptation to the multivariate Rademacher complexity $\mathfrak R_d$ of \citet{maurer2016vector}, with the explicit constants required by Table~\ref{tab:rspo_generalization_methods}. In each example, the excess risk bound follows from the complexity bound by a direct application of Theorem~\ref{thm:rspo_meta_generalization}, so only the complexity bounds require proof.

\subsection{Bounded Affine Class} We begin with affine predictors, which cover the linear models used in predict-then-optimize pipelines and yield an explicit dimension-dependent bound. The scalar analogue of the bound below is classical \citep{bartlett2002rademacher,kakade2008complexity}; see also \citet{maurer2006rademacher} for complexity bounds on classes of linear transformations. We include a short proof because the multivariate form with the explicit factor $\sqrt d$ does not appear verbatim in these references.

\begin{assumption}\label{ass:bounded_affine_class}
There exists a context map $\bm\varphi:\mathcal X\to\mathbb R^m$ with $\|\bm\varphi(\bm x)\|_2\leq \kappa$ for all $\bm x\in\mathcal X$, and the hypothesis class is
\begin{equation*}
    \mathcal G_{\rm aff}=\left\{\bm g_{\bm W}(\bm x)=\bm W\bm\varphi(\bm x):\bm W\in\mathbb R^{d\times m},\ \|\bm W\|_F\leq B\right\}.
\end{equation*}
\end{assumption}

\begin{example}[Bounded affine class]
\label{ex:rspo_gen_affine}
Under Assumption~\ref{ass:bounded_affine_class}, the multivariate Rademacher complexity satisfies
\begin{equation}\label{eq:rad_affine}
    \mathfrak R_d(\mathcal G_{\rm aff})
    \leq
    B\kappa\sqrt{\frac{d}{N}},
\end{equation}
and Theorem~\ref{thm:rspo_meta_generalization} therefore yields that, for any $\delta\in(0,1)$, with probability at least $1-\delta$,
\begin{equation*}
    R(\widehat{\bm{g}}_{\rm RSPO})-R^{\star}(\mathcal G_{\rm aff})
    \leq
    \frac{4\sqrt{2}\,rB\kappa}{\gamma}\sqrt{\frac{d}{N}}
    +
    b\sqrt{\frac{2\log(2/\delta)}{N}} .
\end{equation*}
The bound is explicit in the decision dimension: the two radii $B$ and $\kappa$ enter multiplicatively, and the factor $\sqrt d$ reflects that a $d$-dimensional cost vector must be predicted from the same sample.
\end{example}
\noindent\textit{Proof of Example \ref{ex:rspo_gen_affine}.}~
By the definition of multivariate Rademacher complexity,
\begin{equation*}
    \mathfrak R_d(\mathcal G_{\rm aff})
    =
    \mathbb E_{\bm\sigma,\bm X}
    \left[
    \sup_{\|\bm W\|_F\leq B}
    \frac{1}{N}
    \sum_{i=1}^N
    \bm\sigma_i^\top \bm W\bm\varphi(\bm x_i)
    \right],
\end{equation*}
where $\bm\sigma_i\in\{+1,-1\}^d$ are independent Rademacher random vectors. Using the Frobenius inner product, we have
\begin{equation*}
    \sum_{i=1}^N
    \bm\sigma_i^\top \bm W\bm\varphi(\bm x_i)
    =
    \left\langle
    \bm W,
    \sum_{i=1}^N
    \bm\sigma_i\bm\varphi(\bm x_i)^\top
    \right\rangle_F .
\end{equation*}
Therefore, by the Cauchy--Schwarz inequality,
\begin{equation*}
    \mathfrak R_d(\mathcal G_{\rm aff})
    \leq
    \frac{B}{N}
    \mathbb E_{\bm\sigma,\bm X}
    \left[
    \left\|
    \sum_{i=1}^N
    \bm\sigma_i\bm\varphi(\bm x_i)^\top
    \right\|_F
    \right].
\end{equation*}
By Jensen's inequality,
\begin{equation*}
    \mathbb E_{\bm\sigma,\bm X}
    \left[
    \left\|
    \sum_{i=1}^N
    \bm\sigma_i\bm\varphi(\bm x_i)^\top
    \right\|_F
    \right]
    \leq
    \left(
    \mathbb E_{\bm\sigma,\bm X}
    \left[
    \left\|
    \sum_{i=1}^N
    \bm\sigma_i\bm\varphi(\bm x_i)^\top
    \right\|_F^2
    \right]
    \right)^{1/2}.
\end{equation*}
Since the Rademacher vectors are independent and centered, the cross terms vanish, and thus
\begin{equation*}
    \mathbb E_{\bm\sigma,\bm X}
    \left[
    \left\|
    \sum_{i=1}^N
    \bm\sigma_i\bm\varphi(\bm x_i)^\top
    \right\|_F^2
    \right]
    =
    \sum_{i=1}^N
    \mathbb E_{\bm\sigma,\bm X}
    \left[
    \|\bm\sigma_i\|_2^2
    \|\bm\varphi(\bm x_i)\|_2^2
    \right]
    \leq
    Nd\kappa^2,
\end{equation*}
where we used $\|\bm\sigma_i\|_2^2=d$ and $\|\bm\varphi(\bm x_i)\|_2\leq \kappa$. Hence,
\begin{equation*}
    \mathfrak R_d(\mathcal G_{\rm aff})
    \leq
    \frac{B}{N}\sqrt{Nd\kappa^2}
    =
    B\kappa\sqrt{\frac{d}{N}}.
\end{equation*}
The proof is complete.
\qed

\subsection{Polynomial Discrimination Class} Affine classes do not cover discontinuous or finitely generated prediction rules. We next consider classes whose number of distinct prediction patterns on any finite sample grows only polynomially in the sample size. The scalar version of the resulting bound is standard \citep[see, e.g.,][Chapter~4]{wainwright2019high}; the vector-valued case reduces to it by identifying $(\mathbb R^d)^N$ with $\mathbb R^{Nd}$.

\begin{definition}\label{def:vector_polynomial_discrimination}
Let $\mathcal{G}\subseteq\{\bm{g}:\mathcal{X}\to\mathbb{R}^d\}$ and, for any collection $\bm{x}_1^N=(\bm{x}_1,\ldots,\bm{x}_N)$, define $\mathcal{G}(\bm{x}_1^N):=\left\{\left(\bm{g}(\bm{x}_1),\ldots,\bm{g}(\bm{x}_N)\right):\bm{g}\in\mathcal{G}\right\}\subseteq(\mathbb{R}^d)^N$. We say that $\mathcal{G}$ has polynomial discrimination of order $\nu\geq 1$ if $\operatorname{card}(\mathcal{G}(\bm{x}_1^N))\leq(N+1)^{\nu}$ for every positive integer $N$ and every collection $\bm{x}_1^N$. The exponent $\nu$ may depend on the output dimension $d$.
\end{definition}

\begin{assumption}
\label{ass:vector_g_complexity}
The hypothesis class $\mathcal{G}$ has polynomial discrimination of order $\nu\geq 1$ in the sense of Definition~\ref{def:vector_polynomial_discrimination}, and there is a constant $B_1>0$ such that
\begin{equation*}
    D_{\mathcal G}(\bm x^N):=\sup_{\bm{g}\in\mathcal{G}}\left(\frac{1}{N}\sum_{i=1}^N\|\bm{g}(\bm{x}_i)\|_2^2\right)^{1/2}\leq B_1
\end{equation*}
for every collection $\bm{x}_1^N$.
\end{assumption}

\begin{example}[Polynomial discrimination class]
\label{ex:rspo_gen_pd}
Under Assumption~\ref{ass:vector_g_complexity}, the multivariate Rademacher complexity satisfies
\begin{equation}\label{eq:rad_pd}
    \mathfrak R_d(\mathcal G)
    \leq
    B_1
    \sqrt{
    \frac{2\nu\log(N+1)}{N}
    },
\end{equation}
and Theorem~\ref{thm:rspo_meta_generalization} therefore yields that, for any $\delta\in(0,1)$, with probability at least $1-\delta$,
\begin{equation*}
    R(\widehat{\bm{g}}_{\rm RSPO})
    -
    R^{\star}(\mathcal G)
    \leq
    \frac{8rB_2}{\gamma}
    \sqrt{
    \frac{\nu\log(N+1)}{N}
    }
    +
    b\sqrt{
    \frac{2\log(2/\delta)}{N}
    }.
\end{equation*}
Here the dimension $d$ does not appear explicitly; its effect is absorbed into the discrimination order $\nu$, which counts prediction patterns rather than parameters. The price of this generality is the factor $\sqrt{\log(N+1)}$, since the argument controls a finite but growing set of sample patterns instead of exploiting a parameterization.
\end{example}

\noindent\textit{Proof of Example \ref{ex:rspo_gen_pd}.}~
Fix a collection $\bm{x}_1^N$ and identify each sample evaluation $\left(\bm{g}(\bm{x}_1),\ldots,\bm{g}(\bm{x}_N)\right)$ with a vector in $\mathbb{R}^{Nd}$, so that the empirical multivariate Rademacher complexity equals $\frac{1}{N}\,\mathbb{E}_{\bm{\sigma}}\big[\sup_{\bm v\in\mathcal{G}(\bm{x}_1^N)}\langle\bm{\sigma},\bm v\rangle\big]$ with $\bm{\sigma}\in\{\pm1\}^{Nd}$. Since $\operatorname{card}(\mathcal{G}(\bm{x}_1^N))\leq(N+1)^{\nu}$ by Definition~\ref{def:vector_polynomial_discrimination} and $\sup_{\bm v\in\mathcal{G}(\bm{x}_1^N)}\|\bm v\|_2\leq\sqrt{N}B_1$ by Assumption~\ref{ass:vector_g_complexity}, Massart's finite class lemma \citep{massart2000some} yields
\begin{equation*}
    \frac{1}{N}\,\mathbb{E}_{\bm{\sigma}}\left[\sup_{\bm v\in\mathcal{G}(\bm{x}_1^N)}\langle\bm{\sigma},\bm v\rangle\right]
    \leq
    \frac{1}{N}\sqrt{2\nu\log(N+1)}\cdot\sqrt{N}B_1
    =
    B_1\sqrt{\frac{2\nu\log(N+1)}{N}} .
\end{equation*}
Taking expectation over the sample establishes \eqref{eq:rad_pd}.
\qed

\begin{remark}[The two classes are not nested]
\label{rem:affine_pd_not_nested}
Neither class contains the other. A bounded affine class generates infinitely many prediction patterns already on even a single sample point: for the identity context map and any $\bm x_0\neq\bm 0$, the choice $\bm W_a=\frac{a}{\|\bm x_0\|_2^2}\bm e_1\bm x_0^\top$ satisfies $\|\bm W_a\|_F\leq B$ for every $a\in[-B\|\bm x_0\|_2,B\|\bm x_0\|_2]$ while producing the distinct values $\bm W_a\bm x_0=a\bm e_1$, so $\operatorname{card}(\mathcal G_{\rm aff}(\bm x_0))=\infty$ and polynomial discrimination fails. Conversely, the scalar threshold class $\mathcal G_{\rm th}=\{g_t(x)=\mathbf 1\{x\leq t\}:t\in\mathbb R\}$ induces at most $N+1$ patterns on any sample and hence has polynomial discrimination of order $\nu=1$, yet the family $\{\mathbf 1\{x\leq t\}:t\in\mathbb R\}$ contains infinitely many linearly independent functions, whereas any class of the form $\{\bm w^\top\bm\varphi(\cdot)\}$ spans a function space of dimension at most $m$. Therefore, $\mathcal G_{\rm th}$ is not contained in any bounded affine class.
\end{remark}

\subsection{Dudley-entropy-integral Bounded Class} The preceding arguments rely either on a linear parameterization or on a finite number of sample prediction patterns. For more general classes we control the complexity through a uniform entropy integral. The bound below is Dudley's entropy integral theorem \citep[Theorem~5.22]{wainwright2019high}; the only step specific to our setting is verifying that the vector-valued Rademacher process is sub-Gaussian with respect to the scaled empirical metric $d_N/\sqrt N$, which we record in the proof.

\begin{definition}[Covering number]
\label{def:covering_number}
Fix a sample $\bm x^N$ and equip $\mathcal G$ with the empirical metric $d_N$. For $\epsilon>0$, a finite collection $\bm g_1,\ldots,\bm g_M\in\mathcal G$ is an \emph{$\epsilon$-cover} of $\mathcal G$ if every $\bm g\in\mathcal G$ satisfies $d_N(\bm g,\bm g_m)\leq\epsilon$ for some $m\in[M]$. The \emph{covering number} $\mathcal N(\epsilon,\mathcal G,d_N)$ is the smallest cardinality of such a cover, set to $+\infty$ when no finite cover exists. 
\end{definition}
The covering number measures how many representative predictors are needed to approximate every member of $\mathcal G$ to accuracy $\epsilon$ on the sample, and its logarithm $\log\mathcal N(\epsilon,\mathcal G,d_N)$ is referred to as the metric entropy of $\mathcal G$.

\begin{assumption}
\label{ass:metric_entropy_class}
Let $\mathcal G\subseteq\{\bm g:\mathcal X\to\mathbb R^d\}$ and, for any sample $\bm x^N$, define the empirical metric $d_N(\bm g,\widetilde{\bm g}):=\big(\frac1N\sum_{i=1}^N\|\bm g(\bm x_i)-\widetilde{\bm g}(\bm x_i)\|_2^2\big)^{1/2}$. There exist constants $\Delta_{\mathcal G}>0$ and $J_{\mathcal G}<\infty$, independent of $N$ and of the sample, such that $\operatorname{diam}_N(\mathcal G):=\sup_{\bm g,\widetilde{\bm g}\in\mathcal G}d_N(\bm g,\widetilde{\bm g})\leq\Delta_{\mathcal G}$ and $\int_0^{\Delta_{\mathcal G}}\sqrt{\log\mathcal N(\epsilon,\mathcal G,d_N)}\,d\epsilon\leq J_{\mathcal G}$, where $\mathcal N(\epsilon,\mathcal G,d_N)$ is the covering number of $\mathcal G$ under $d_N$.
\end{assumption}

\begin{example}[Dudley-entropy-integral bounded class]
\label{ex:rspo_gen_dudley}
Under Assumption~\ref{ass:metric_entropy_class}, there exists a universal constant $C>0$ such that
\begin{equation}\label{eq:rad_dudley}
    \mathfrak R_d(\mathcal G)
    \leq
    C\frac{J_{\mathcal G}}{\sqrt N},
\end{equation}
and Theorem~\ref{thm:rspo_meta_generalization} therefore yields that, for any $\delta\in(0,1)$, with probability at least $1-\delta$,
\begin{equation*}
R(\widehat{\bm g}_{\rm RSPO})
-
R^\star(\mathcal G)
\leq
\frac{4\sqrt{2}\,Cr}{\gamma}
\frac{J_{\mathcal G}}{\sqrt N}
+
b
\sqrt{
\frac{2\log(2/\delta)}{N}
}.
\end{equation*}
The entire geometry of the class is summarized by the single quantity $J_{\mathcal G}$, and neither the dimension $d$ nor any parameterization appears. This is what allows the guarantee to cover infinite-dimensional classes, at the cost of requiring covering numbers that are uniform over samples.
\end{example}
\noindent\textit{Proof of Example \ref{ex:rspo_gen_dudley}.}~
Fix a sample $\bm x^N=(\bm x_1,\ldots,\bm x_N)$ and for each $\bm g\in\mathcal G$, define the Rademacher process
\begin{equation*}
Z_{\bm g}
:=
\frac1N
\sum_{i=1}^N
\bm\sigma_i^\top \bm g(\bm x_i),
\end{equation*}
where $\bm\sigma_i=(\sigma_{i1},\ldots,\sigma_{id})$ and the coordinates $\{\sigma_{ij}\}$ are independent Rademacher random variables. Then, for any $\bm g,\widetilde{\bm g}\in\mathcal G$ and any $\lambda\in\mathbb R$,
\begin{align*}
\mathbb E_{\bm\sigma}
\left[
\exp\left\{
\lambda
\left(
Z_{\bm g}-Z_{\widetilde{\bm g}}
\right)
\right\}
\right]
&=
\mathbb E_{\bm\sigma}
\left[
\exp\left\{
\frac{\lambda}{N}
\sum_{i=1}^N
\bm\sigma_i^\top
\left(
\bm g(\bm x_i)-\widetilde{\bm g}(\bm x_i)
\right)
\right\}
\right]                                                        \\
&\leq
\exp\left\{
\frac{\lambda^2}{2N^2}
\sum_{i=1}^N
\left\|
\bm g(\bm x_i)-\widetilde{\bm g}(\bm x_i)
\right\|_2^2
\right\}                                                        \\
&=
\exp\left\{
\frac{\lambda^2}{2N}
d_N^2(\bm g,\widetilde{\bm g})
\right\}.
\end{align*}
Hence $\{Z_{\bm g}:\bm g\in\mathcal G\}$ is a centered sub-Gaussian process with respect to the metric $d_N/\sqrt N$.

Fix any $\bm g_0\in\mathcal G$. Since $\mathbb E_{\bm\sigma}[Z_{\bm g_0}]=0$, we have
$\mathbb E_{\bm\sigma}[\sup_{\bm g\in\mathcal G}Z_{\bm g}]
\leq
\mathbb E_{\bm\sigma}[\sup_{\bm g\in\mathcal G}(Z_{\bm g}-Z_{\bm g_0})]$. By Dudley's entropy integral bound \citep[Theorem~5.22]{wainwright2019high}, there exists a universal constant
$C>0$ such that
\begin{equation*}
    \mathbb{E}_{\bm{\sigma}}
    \left[
    \sup_{\bm{g}\in\mathcal{G}}
    \frac{1}{N}
    \sum_{i=1}^N
    \bm{\sigma}_i^\top \bm{g}(\bm{x}_i)
    \right]  
    =
    \mathbb E_{\bm\sigma}
    \left[
    \sup_{\bm g\in\mathcal G} Z_{\bm g}
    \right]
    \leq
    \frac{C}{\sqrt N}
    \int_0^{\operatorname{diam}_N(\mathcal G)}
    \sqrt{
    \log
    \mathcal N
    \left(
    \epsilon,\mathcal G,d_N
    \right)
    }
    \,d\epsilon .
\end{equation*}
Since $\mathcal N(\epsilon,\mathcal G,d_N/\sqrt N)=\mathcal N(\sqrt N\epsilon,\mathcal G,d_N)$, the substitution $u=\sqrt N\epsilon$ pulls out the factor $1/\sqrt N$. By Assumption~\ref{ass:metric_entropy_class}, the entropy integral is uniformly
bounded by $J_{\mathcal G}$, and therefore
\begin{equation*}
    \mathbb{E}_{\bm{\sigma}}
    \left[
    \sup_{\bm{g}\in\mathcal{G}}
    \frac{1}{N}
    \sum_{i=1}^N
    \bm{\sigma}_i^\top \bm{g}(\bm{x}_i)
    \right] 
    \leq
    C
    \frac{J_{\mathcal G}}{\sqrt N}.
\end{equation*}
Since this bound holds for every fixed sample $\bm x^N$, taking expectation
with respect to the sample gives
\begin{equation*}
    \mathfrak R_d(\mathcal G)
    \leq
    C
    \frac{J_{\mathcal G}}{\sqrt N}.
\end{equation*}
This establishes \eqref{eq:rad_dudley}, and the proof is complete.
\qed

Assumption~\ref{ass:metric_entropy_class} covers smooth infinite-dimensional classes that the previous two conditions do not. Let $\mathcal X=[0,1]^p$ and consider the vector-valued H\"older class $\mathcal G_{s,L}=\{\bm g=(g_1,\ldots,g_d):g_j\in C^s([0,1]^p),\ \|g_j\|_{C^s}\leq L\}$ with $s>p/2$. Since $d_N(\bm g,\widetilde{\bm g})\leq\sqrt d\max_j\|g_j-\widetilde g_j\|_\infty$ and the $C^s$ norm dominates the uniform norm, one may take $\Delta_{\mathcal G}=2L\sqrt d$; applying the standard entropy bound for scalar H\"older balls coordinatewise and taking product covers gives $\log\mathcal N(\epsilon,\mathcal G_{s,L},d_N)\leq C_{p,s}d(L\sqrt d/\epsilon)^{p/s}$, whose entropy integral is finite because $s>p/2$. The class contains a continuum of values at any fixed $\bm x_0$ and therefore lies outside the finite-pattern regime of Definition~\ref{def:vector_polynomial_discrimination}.

\subsection{Vector-valued RKHS class} Finally, we consider norm balls in a vector-valued reproducing kernel Hilbert space \citep{micchelli2005learning}, which fit the meta theorem in the same way as the previous three classes. The scalar kernel-class bound is classical \citep{bartlett2002rademacher}; see also \citet{mohri2018foundations}. We give a short proof of the vector-valued form with the coordinate-sum norm constraint of Assumption~\ref{ass:rspo_rkhs_class}, since this exact statement does not appear in these references.

\begin{assumption}
\label{ass:rspo_rkhs_class}
Let $\mathcal H$ be an RKHS on $\mathcal X$ with reproducing kernel $K$, and assume that $K$ is bounded, in the sense that $\sup_{\bm x\in\mathcal X}K(\bm x,\bm x)\leq \kappa^2$. For some $R_{\mathcal H}>0$, the prediction class is
\begin{equation*}
    \mathcal G_{\mathcal H}:=\left\{\bm g=(g_1,\ldots,g_d):g_j\in\mathcal H,\ \sum_{j=1}^d\|g_j\|_{\mathcal H}^2\leq R_{\mathcal H}^2\right\}.
\end{equation*}
\end{assumption}

\begin{example}[Vector-valued RKHS ball]
\label{ex:rspo_gen_rkhs}
Under Assumption~\ref{ass:rspo_rkhs_class}, the multivariate Rademacher complexity satisfies
\begin{equation}\label{eq:rad_rkhs}
    \mathfrak R_d(\mathcal G_{\mathcal H})
    \leq
    \kappa R_{\mathcal H}\sqrt{\frac{d}{N}},
\end{equation}
and Theorem~\ref{thm:rspo_meta_generalization} therefore yields that, for any $\delta\in(0,1)$, with probability at least $1-\delta$,
\begin{equation*}
    R(\widehat{\bm g}_{\rm RSPO})
    -
    R^{\star}(\mathcal G_{\mathcal H})
    \leq
    \frac{4\sqrt{2}\,r\kappa R_{\mathcal H}}{\gamma}\sqrt{\frac{d}{N}}
    +
    b\sqrt{\frac{2\log(2/\delta)}{N}}.
\end{equation*}
The guarantee has the same form as the affine case, with the Frobenius radius $B$ replaced by the RKHS radius $R_{\mathcal H}$ and the context bound $\kappa$ by the kernel bound. The class is infinite-dimensional, yet the $\sqrt{d/N}$ dependence persists: it is the dimension of the cost vector, not of the function space, that the complexity tracks.
\end{example}
\noindent\textit{Proof of Example \ref{ex:rspo_gen_rkhs}.}~
Fix a sample $\bm x^N=(\bm x_1,\ldots,\bm x_N)$. By the reproducing property, $g_j(\bm x_i)=\langle g_j, K(\cdot,\bm x_i)\rangle_{\mathcal H}$ for every $j\in[d]$ and $i\in[N]$. Hence, for any $\bm g=(g_1,\ldots,g_d)\in\mathcal G_{\mathcal H}$,
\begin{equation*}
    \frac{1}{N}\sum_{i=1}^N \bm\sigma_i^\top\bm g(\bm x_i)
    =
    \sum_{j=1}^d
    \left\langle
    g_j,\ h_j
    \right\rangle_{\mathcal H},
    \qquad
    h_j:=\frac{1}{N}\sum_{i=1}^N \sigma_{ij}K(\cdot,\bm x_i).
\end{equation*}
By the Cauchy--Schwarz inequality, applied first in $\mathcal H$ and then in $\mathbb R^d$,
\begin{equation*}
    \sum_{j=1}^d
    \left\langle
    g_j, h_j
    \right\rangle_{\mathcal H}
    \leq
    \left(\sum_{j=1}^d\|g_j\|_{\mathcal H}^2\right)^{1/2}
    \left(\sum_{j=1}^d\|h_j\|_{\mathcal H}^2\right)^{1/2}
    \leq
    R_{\mathcal H}
    \left(\sum_{j=1}^d\|h_j\|_{\mathcal H}^2\right)^{1/2}.
\end{equation*}
Moreover, since the Rademacher variables are independent and centered,
\begin{equation*}
    \mathbb E_{\bm\sigma}\left[\|h_j\|_{\mathcal H}^2\right]
    =
    \frac{1}{N^2}\sum_{i=1}^N\sum_{i'=1}^N
    \mathbb E[\sigma_{ij}\sigma_{i'j}]\,
    K(\bm x_i,\bm x_{i'})
    =
    \frac{1}{N^2}\sum_{i=1}^N K(\bm x_i,\bm x_i)
    \leq
    \frac{\kappa^2}{N}.
\end{equation*}
Therefore, by Jensen's inequality,
\begin{equation*}
    \mathbb{E}_{\bm{\sigma}}
    \left[
    \sup_{\bm{g}\in\mathcal{G}_{\mathcal H}}
    \frac{1}{N}
    \sum_{i=1}^N
    \bm{\sigma}_i^\top \bm{g}(\bm{x}_i)
    \right]  
    \leq
    R_{\mathcal H}\,
    \mathbb E_{\bm\sigma}\left[\left(\sum_{j=1}^d\|h_j\|_{\mathcal H}^2\right)^{1/2}\right]
    \leq
    R_{\mathcal H}
    \left(\sum_{j=1}^d\mathbb E_{\bm\sigma}\left[\|h_j\|_{\mathcal H}^2\right]\right)^{1/2}
    \leq
    \kappa R_{\mathcal H}\sqrt{\frac{d}{N}}.
\end{equation*}
Taking the expectation with respect to the sample establishes \eqref{eq:rad_rkhs}.
\qed

\newpage
\section{Experimental Details and Additional Results}
\label{app:experimental_details}

This appendix records the portfolio learning reformulations, quadratic-risk results, and case-specific calibration and implementation details supporting Section~\ref{sec:numeric_study}. The common data-generating, validation, testing, and evaluation procedures are stated in that section.

\subsection{Portfolio Learning Reformulations}
\label{app:portfolio_reformulation}
For the main portfolio experiment, the constraints in Section~\ref{subsec:portfolio_rspo+} define a bounded polyhedron after the standard epigraph reformulation of $\|\bm\Sigma\bm z\|_1$. Theorem~\ref{thm:rspo_plus_reformulation} therefore yields finite-dimensional learning formulations for linear predictors.

\noindent\underline{ERM formulations.}~
For linear predictors, the $\mathcal{RSPO}_+$ empirical risk minimization problem admits the following equivalent reformulation:
    \begin{equation*}
        \begin{aligned}
            \min\ &\frac{1}{N}\sum_{i\in[N]}\Big[\beta\zeta_i+\mu_i+\frac{1}{2a\gamma}\|\bm s_i\|_2^2+a\bm{z}^{\star}_{\gamma}(\bm{y}_i)^\top\bm{B}\bm{x}_i+\frac{a\gamma}{2}\|\bm{z}^{\star}_{\gamma}(\bm{y}_i)\|^2_2-v^{\star}(\bm{y}_i)\Big]+  \lambda \|\bm B\|_F^2\\
            \mathrm{s.t.}\ &  -\zeta_i\bm 1\leq \bm p_i \leq \zeta_i\bm 1 &i\in[N] \\
            &\bm s_i \geq \bm{y}_i-a\bm{B}\bm{x}_i-\bm\Sigma\bm{p}
            _i-\mu_i\bm{1} &i\in[N]\\
            &\bm{p}_i\in \bb{R}^d,\ \bm s_i \ge 0,\ \mu_i\ge 0,\ \zeta_i \ge 0 &i\in[N]\\
            &\bm{B}\in\mathbb{R}^{d\times p}.
            \end{aligned}
    \end{equation*}

For comparison, the ILO benchmark based on the nominal $\mathcal{SPO}_+$ loss can be reformulated as
\begin{equation*}
        \begin{aligned}
            \min\quad & \frac{1}{N}\sum_{i\in[N]}\Big[\beta\zeta_i+\mu_i+a\bm{z}^{\star}(\bm{y}_i)^\top\bm{B}\bm{x}_i-v^{\star}(\bm{y}_i)\Big] +  \lambda \|\bm B\|_F^2\\
            \text{s.t.} \quad & \bm{y}_i-a\bm{B}\bm{x}_i-\bm\Sigma\bm{p}
            _i-\mu_i\bm{1}\leq \bm{0} & i\in[N]\\
            & -\zeta_i\bm 1\leq \bm p_i \leq \zeta_i\bm 1 &i\in[N] \\
            &\bm{p}_i\in \bb{R}^d,\ \mu_i\geq 0,\ \zeta_i\geq 0  &i\in[N]\\
            &\bm{B}\in\mathbb{R}^{d\times p}.
    \end{aligned}
\end{equation*}

\phantomsection
\noindent\underline{Dual representation.}~
To obtain the first formulation, let $\bm c=\bm y-a\hat{\bm y}$ and consider
\begin{equation*}
\max_{\bm z\in\mathcal Z}\left\{\bm c^\top\bm z-\frac{a\gamma}{2}\|\bm z\|_2^2\right\},
\qquad
\mathcal Z=\{\bm z\geq\bm0:\ \bm1^\top\bm z\leq1,\ \|\bm\Sigma\bm z\|_1\leq\beta\}.
\end{equation*}
Introducing $\bm w=\bm\Sigma\bm z$, multiplier $\mu\geq0$ for $\bm1^\top\bm z\leq1$, and multiplier $\bm p$ for $\bm w=\bm\Sigma\bm z$ gives the Lagrangian of the equivalent minimization problem,
\begin{equation*}
L(\bm z,\bm w,\mu,\bm p)
=\frac{a\gamma}{2}\|\bm z\|_2^2-\bm c^\top\bm z
+\mu(\bm1^\top\bm z-1)+\bm p^\top(\bm\Sigma\bm z-\bm w).
\end{equation*}
Because $\beta>0$, a sufficiently small $\bm z>\bm0$ satisfies both inequalities strictly, so Slater's condition yields strong duality. Minimizing the Lagrangian over $\bm z\geq\bm0$ and $\|\bm w\|_1\leq\beta$ gives
\begin{equation*}
\max_{\bm z\in\mathcal Z}\left\{\bm c^\top\bm z-\frac{a\gamma}{2}\|\bm z\|_2^2\right\}
=
\min_{\substack{\mu\geq0,\,\bm p\in\mathbb R^d}}
\left\{\frac{1}{2a\gamma}\|(\bm c-\mu\bm1-\bm\Sigma\bm p)_+\|_2^2+\beta\|\bm p\|_\infty+\mu\right\}.
\end{equation*}
Introducing $\bm s_i\geq\bm y_i-a\bm B\bm x_i-\mu_i\bm1-\bm\Sigma\bm p_i$ with $\bm s_i\geq\bm0$ and $-\zeta_i\bm1\leq\bm p_i\leq\zeta_i\bm1$ gives the stated $\mathcal{RSPO}_+$ ERM formulation; the $\mathcal{SPO}_+$ formulation follows from the corresponding linear-program dual.

\subsection{Quadratic-Risk Portfolio Experiment}\label{app:portfolio_l2}
We complement the main portfolio study by considering portfolio instances with the quadratic variance constraint used in the classical Markowitz formulation and in the portfolio experiment of \citet{elmachtoub2022smart}. Given a predicted cost vector $\hat{\bm y}$, the deployed portfolio solves
\begin{equation*}
\bm z_\gamma^\star(\hat{\bm y})
\in\argmin_{\bm z}
\left\{\hat{\bm y}^\top\bm z+\frac{\gamma}{2}\|\bm z\|_2^2:
\bm z\geq\bm0,\ \bm1^\top\bm z\leq1,\ \bm z^\top\bm\Sigma\bm z\leq\beta\right\},
\end{equation*}
where $\bm z^\top\bm\Sigma\bm z$ is the portfolio variance and $\beta$ is the variance budget; setting $\gamma=0$ recovers the nominal decision problem. We use the factor-model return process from Section~\ref{subsec:portfolio_rspo+}, augment each context vector with an unpenalized intercept, and set $\beta=2.25\,\bm z_{\rm unif}^\top\bm\Sigma\bm z_{\rm unif}$ for $\bm z_{\rm unif}=d^{-1}\bm1$. Because the quadratic constraint defines a nonpolyhedral feasible set, we train the four predictors using stochastic first-order updates based on their respective loss gradients or subgradients. For each configuration, we return the stepsize-weighted average of up to $1{,}000$ iterates, using minibatches of size $10$, stepsize $0.1/\sqrt{t}$, and gradient-norm tolerance $10^{-5}$. The validation, testing, and replication procedures otherwise follow Section~\ref{sec:numeric_study}.

\begin{figure}[!t]
    \centering
    \includegraphics[width=0.95\linewidth]{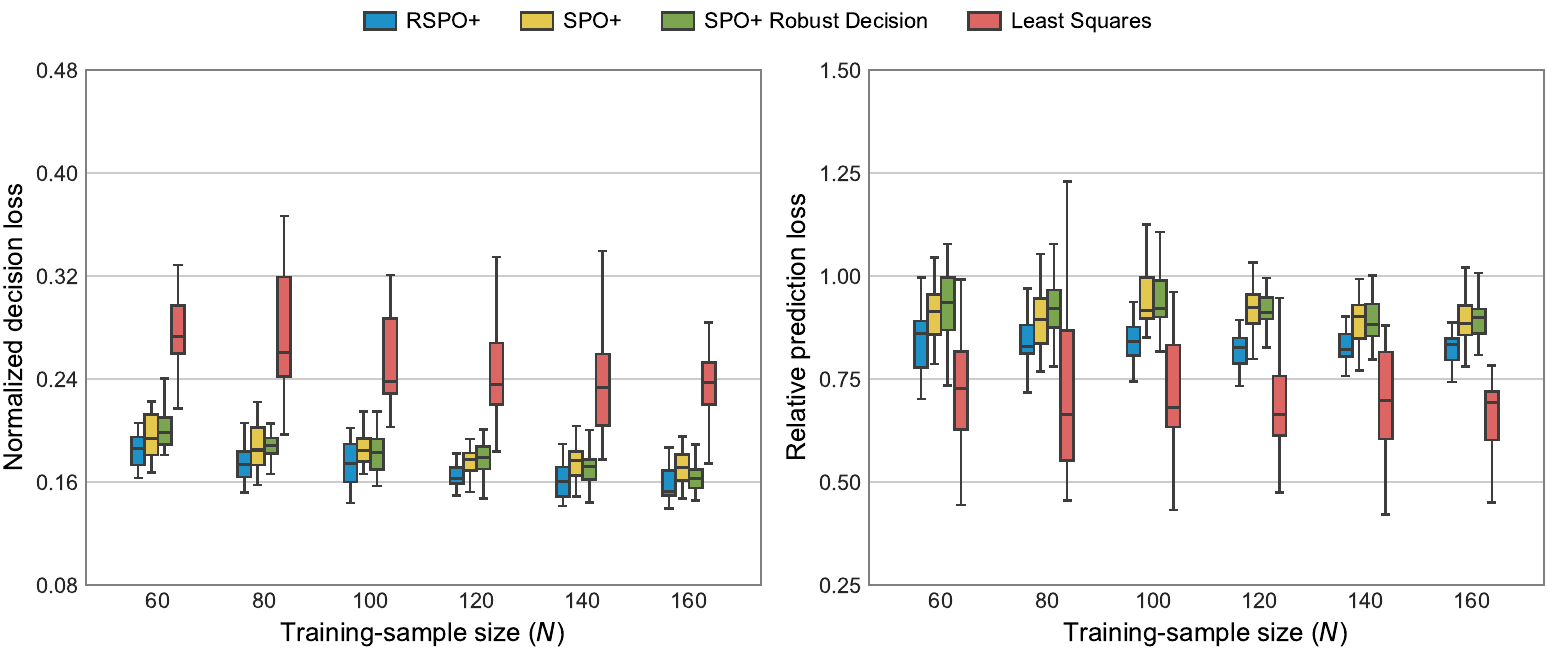}
    
    \caption{Normalized decision loss (left) and relative prediction loss (right) as $N$ varies under nonlinear returns on quadratic-risk portfolio optimization instances, with $\mathtt{deg}=4$, $p=80$, $d=60$, and $\tau=1$ fixed.}
    \label{fig:portfolio2_decision_dimension_combined_deg4}\vspace{-4mm}
\end{figure}

\begin{figure}[!t]
    \centering
    \includegraphics[width=0.95\linewidth]{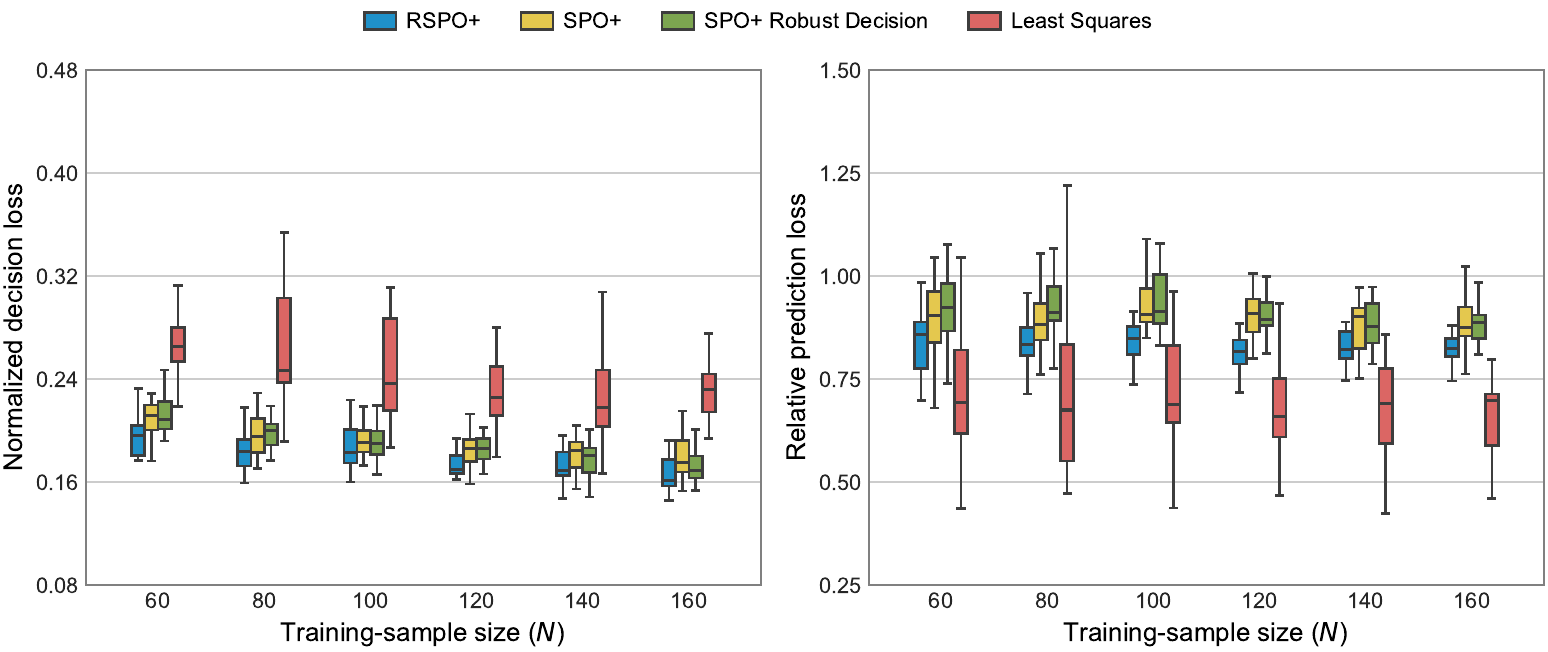}
    
    \caption{Normalized decision loss (left) and relative prediction loss (right) as $N$ varies under linear returns on quadratic-risk portfolio optimization instances, with $\mathtt{deg}=1$, $p=80$, $d=60$, and $\tau=1$ fixed.}
    \label{fig:portfolio2_decision_dimension_combined_deg1}\vspace{-4mm}
\end{figure}

Figures~\ref{fig:portfolio2_decision_dimension_combined_deg4} and~\ref{fig:portfolio2_decision_dimension_combined_deg1} examine the effect of the training-sample size under nonlinear and linear return specifications. In both cases, $\mathcal{RSPO}_+$ attains the lowest normalized decision loss at every reported $N$, and its decision loss declines overall as the training sample grows. Its advantage over $\mathcal{SPO}_+$ and $\mathcal{SPO}_+$ with robust decisions persists under both specifications, while least squares has the largest decision loss. The prediction panels show that least squares attains lower relative prediction loss, whereas $\mathcal{RSPO}_+$ converts its predictions into better portfolio decisions.

\begin{figure}[!t]
    \centering
    \includegraphics[width=0.95\linewidth]{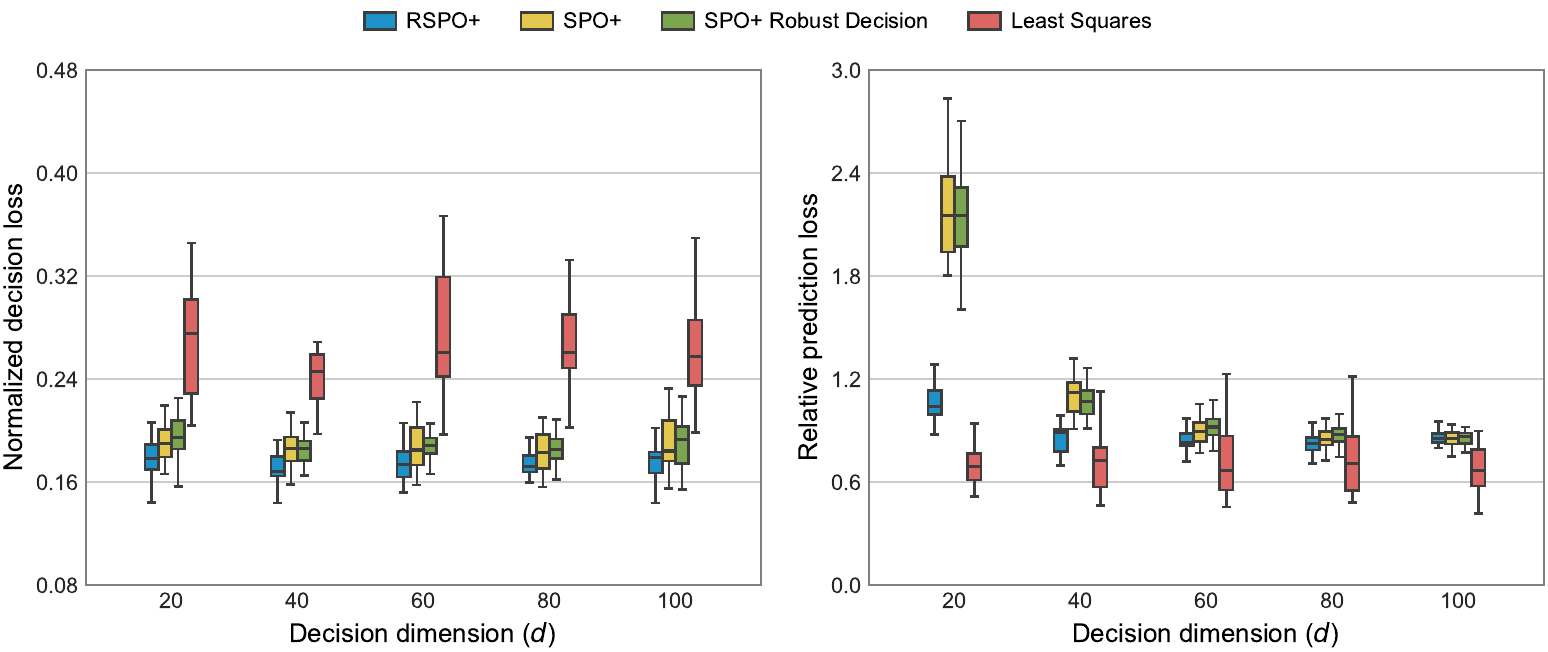}
    
    \caption{Normalized decision loss (left) and relative prediction loss (right) as $d$ varies on quadratic-risk portfolio optimization instances, with $N=80$, $p=80$, $\mathtt{deg}=4$, and $\tau=1$ fixed.}
    \label{fig:portfolio2_decision_dimension_combined}\vspace{-4mm}
\end{figure}

Figure~\ref{fig:portfolio2_decision_dimension_combined} examines the effect of the number of assets. Across $d\in\{20,40,60,80,100\}$, the decision loss of $\mathcal{RSPO}_+$ remains approximately stable and is the lowest among the four methods, while least squares has the largest decision loss. This advantage persists even as the relative prediction losses of $\mathcal{RSPO}_+$, $\mathcal{SPO}_+$, and $\mathcal{SPO}_+$ with robust decisions converge at the larger dimensions. Together, these results show that using the same robust decision map to define the training loss and generate deployed decisions continues to improve downstream decision quality under a smooth quadratic risk constraint, extending the empirical evidence beyond the polyhedral risk set used in the main experiment.

\subsection{Implementation and Calibration Details}
\label{app:data_implementation}
This subsection gives the risk-budget calibration for the main $\ell_1$-risk study and the implementation settings for gradient-based refinement.

\noindent\underline{Risk-budget calibration.}~
For the main $\ell_1$-risk experiment, we calibrate the budget against the uniform portfolio $\bm z_{\rm unif}=d^{-1}\bm1$. Defining its aggregate risk exposure as $\rho_{\rm base}=\|\bm\Sigma\bm z_{\rm unif}\|_1$, we set $\beta=2\rho_{\rm base}$.

\phantomsection
\label{app:gradient_implementation}
\noindent\underline{Gradient-based refinement.}~
At each iteration, Algorithm~\ref{alg:gradient-descent} predicts the cost vector, solves the robust portfolio problem, identifies the active constraints, and updates the predictor using the Jacobian in Proposition~\ref{prop:zstar_gradient_polyhedral}. The $\ell_1$ risk constraint is represented through its lifted epigraph, and the required Jacobian--vector products are obtained from the corresponding KKT system. We set the batch size to $B_{\mathrm s}=5$ and run $T=50$ iterations, with backtracking factor $\rho_{\rm ls}=0.5$ and at most $K=10$ trial steps per iteration. In place of the Armijo sufficient-decrease test of Algorithm~\ref{alg:gradient-descent}, our implementation accepts a trial step as soon as it does not increase the batch objective by more than a numerical tolerance $\epsilon=10^{-10}$, i.e., the acceptance rule with $c_{\rm ls}=0$ relaxed by $\epsilon$ to guard against floating-point error. The validation procedure jointly selects $(\gamma,\lambda)$ and the base stepsize $\eta_0\in\{10^{-3},3\times10^{-3},10^{-2}\}$ by minimizing validation $\mathcal{RSPO}$ loss.

\end{appendices}

\end{document}